\documentclass[11pt]{article}
\usepackage{amsmath}
\usepackage{amsxtra}
\usepackage{amscd}
\usepackage{amsthm}
\usepackage{amsfonts}
\usepackage{amssymb}
\usepackage{eucal}
\usepackage[all]{xy}
\usepackage{graphicx}
\usepackage{pifont}
\usepackage{comment}
\usepackage{tikz-cd}
\usepackage{enumitem}
\usepackage{latexsym}
\usepackage{framed}
\usepackage{fullpage}
\usepackage{todonotes}
\usepackage{hyperref}
    \hypersetup{colorlinks=true,citecolor=blue,urlcolor =black,linkbordercolor={1 0 0}}
\usepackage{hyperref}
\usepackage[noabbrev,capitalize]{cleveref}

\allowdisplaybreaks[1]

\theoremstyle{plain}
\newtheorem{corollary}[subsection]{Corollary}
\newtheorem{lemma}[subsection]{Lemma}
\newtheorem{proposition}[subsection]{Proposition}

\newtheorem{conjecture}{Conjecture}
\newtheorem{theorem}[subsection]{Theorem}

\theoremstyle{definition}
\newtheorem{definition}[subsection]{Definition}
\newtheorem{example}[subsection]{Example}
\newtheorem{remark}[subsection]{Remark}

\newcommand\nc{\newcommand}
\nc\on{\operatorname}
\nc\renc{\renewcommand}
\nc\ssec{\subsection}
\nc\sssec{\subsubsection}
\nc{\mf}{\mathfrak}
\nc\N{{\mathbf N}}
\nc\Z{{\mathbf Z}}
\nc\Q{{\mathbf Q}}
\nc\R{{\mathbf R}}
\nc\C{{\mathbf C}}
\nc\T{{\mathbf T}}
\nc\D{{\mathbf D}}
\nc{\B}{{\mathbf B}}
\nc{\G}{{\mathbf G}}
\nc\mO{{\mathcal O}}
\nc\oP{\overline{P}}
\nc\oB{\overline{B}}
\nc\oN{\overline{N}}
\nc{\tG}{\widetilde{G}}
\nc{\mc}{\mathcal }
\nc{\bb}{\mathbb }
\nc{\HH}{\mathrm{H}}
\nc{\eps}{\varepsilon}
\nc\F{{\mathbb F}}

\nc{\Hom}{\on{Hom}}
\nc{\Aut}{\on{Aut}}
\nc{\End}{\on{End}}
\nc{\Spec}{\on{Spec}}
\nc{\Reg}{\on{Reg}}
\nc{\Specm}{\on{Specm}}
\nc{\Spl}{\on{Spl}}
\nc{\Gal}{\on{Gal}}
\nc{\Cl}{\on{Cl}}
\nc{\GL}{\on{GL}}
\nc{\SL}{\on{SL}}
\nc{\Frob}{\on{Frob}}
\nc{\Mod}{\on{Mod}}
\nc{\Ind}{\on{Ind}}
\nc{\fs}{\on{fs}}
\nc{\la}{\on{la}}
\nc{\Lie}{\on{Lie}}
\nc{\Inf}{\on{Inf}}
\nc{\supp}{\on{supp}}
\nc{\Rep}{\on{Rep}}
\nc{\Spf}{\on{Spf}}
\nc{\Tor}{\on{Tor}}
\nc{\ps}{\on{ps}}
\nc{\lalg}{\on{lalg}}
\nc{\Sp}{\on{Sp}}
\nc{\weight}{\on{wt}}
\nc{\cyc}{\on{cyc}}
\nc{\CE}{\on{CE}}

\nc{\cts}{\on{cts}}

\nc{\ip}[1]{\langle #1 \rangle}
\nc{\ips}[1]{\langle #1,#1 \rangle}

\nc{\cis}{\cos\theta + i \sin\theta}
\nc{\Real}{\on{Re}}
\nc{\Imag}{\on{Im}}
\nc{\Res}{\on{Res}}
\nc{\TODO}{\textcolor{red}}

\nc\A{{\mathbb A}}
\nc\BP{{\mathbf P}}
\nc{\Proj}{\on{Proj}}

\nc\fa{{\mathfrak a}}
\nc\fp{{\mathfrak p}}
\nc\fq{{\mathfrak q}}
\nc\fm{{\mathfrak m}}
\nc\pt{\mathrm{pt}}
\nc{\bd}{\mathbf{d}}
\nc{\disc}{\on{disc}}
\nc{\Tr}{\on{Tr}}
\nc{\WD}{\on{WD}}
\nc{\rec}{\on{rec}}
\nc{\Art}{\on{Art}}
\nc{\tri}{\on{tri}}
\nc{\RHom}{\on{RHom}}
\nc{\Ext}{\on{Ext}}
\nc{\rbar}{\ol{r}}
\nc{\rhobar}{\ol{\rho}}

\nc\K{\mathbf{K}}
\nc{\Ban}{\on{Ban}}

\nc\ol{\overline}
\nc{\ul}{\underline}
\nc\wt{\widetilde}
\nc{\wh}{\widehat}
\nc{\one}{{\mathbf{1}}}
\nc{\id}{\mathrm{id}}
\nc{\uHom}{\ul\Hom}
\nc{\tHom}{\ul\uHom}

\nc\E{{\mathbb E}}
\nc\CC{{\mathcal C}}
\nc\CO{{\mathcal O}}
\nc\abs[1]{| #1 |}

\title{Eigenvarieties over CM fields and trianguline representations}

\author{Vaughan McDonald}
\date{}
\begin{document}

\maketitle
\begin{abstract}
    We show that the Galois representations associated to points on certain (derived) eigenvarieties for $\GL_n$ over a CM field are trianguline with the expected Sen weights, verifying an analogue of a conjecture of Hansen in many cases. The proof follows the strategy of passing to a larger unitary group $\wt{G}$ of signature $(n,n)$, where the key new input is an analytic continuation result for an eigenvariety for $\wt{G}$ localised at an Eisenstein maximal ideal. We also discuss the (subtle) relation of eigenvarieties for $\GL_n$ with the trianguline variety.
\end{abstract}
\setcounter{tocdepth}{1}
\tableofcontents

\section{Introduction}
The purpose of this article is to clarify the relation between families of finite slope $p$-adic automorphic forms, or so-called \textit{eigenvarieties}, and Galois representations. More precisely we show that points on certain eigenvarieties for the group $\GL_n$ over a CM field first constructed by Fu in \cite{FuDerived} have associated Galois representations which are \textit{trianguline} at $p$, confirming an analogue of a conjecture of Hansen \cite[Conjecture 1.2.2]{Han17}. 

To explain our result, we need some notation. See \cref{subsection: definingEigenvarietiesAndGaloisRepresentations} for more definitions and constructions. Let $F$ be an imaginary CM field. Fix a sufficiently small tame level $K^p \subset \GL_n(\A_F^{p,\infty})$, and let $\T(K^p)$ be a ``big'' spherical Hecke algebra as defined in \cite[2.1.10]{GeeNewton}. Let $\mf{m}\subset \T(K^p)$ be a maximal ideal which is non-Eisenstein. For any $i\ge 0$ we can construct from this data a rigid analytic space $\mc{E}^{i}(K^p)_{\mf{m}}$ following \cite[\textsection 6]{FuDerived}, the eigenvariety for $\GL_n/F$ of tame level $K^p$ and cohomological degree $i$, localized at $\mf{m}$. These spaces are a (derived) variant of Emerton's eigenvarieties from \cite{EmertonInterpolation}. The space $\mc{E}^i(K^p)_{\mf{m}}$ is naturally a closed subspace of $(\Spf \T(K^p)_{\mf{m}}^{\on{red}})^{\on{rig}}\times \widehat{T}$, where $\widehat{T}$ is the rigid space of locally analytic characters of the diagonal maximal torus $T_n(F_p)\subset \GL_n(F_p)$, and $F_p:= F\otimes_{\Q}\Q_p$. Scholze's seminal work \cite{Sch15} on torsion cohomology implies that there is a Galois representation $\rho:  \Gal_F \to \GL_n(\T(K^p)_{\mf{m}}^{\on{red}})$, which implies, for each point $(x, \delta) \in \mc{E}^i(K^p)_{\mf{m}}\subset (\Spf \T(K^p)_{\mf{m}}^{\on{red}})^{\on{rig}}\times \widehat{T}$, there is an associated Galois representation $\rho_x: \Gal_F \to \GL_n(\overline{\Q}_p)$. First, we need to define a character $\rho_{\cyc}^{G} \in \widehat{T}$ via the formula $(\rho_{\cyc}^{G})_{\nu}:= (1, \varepsilon_{\cyc}^{-1}\circ \Art_{F_{\nu}}, \dots, \varepsilon_{\cyc}^{-1 + n}\circ \Art_{F_{\nu}})$, where $\Art_{F_{\nu}}: F_{\nu}^{\times} \simeq W_{F_{\nu}}^{\on{ab}}$ is the local Artin map. Our main result then is as follows:
\begin{theorem}
\label{theorem: MainTheorem1}
    Suppose $\mf{m} \subset \T(K^p)$ is decomposed generic (in the sense of Caraiani--Scholze) and non-Eisenstein, $F$ contains an imaginary quadratic field $F_0$, and $p$ is a prime that splits in $F_0$. Then for any point $(x,\delta = (\delta_1,\dots, \delta_n)) \in \mc{E}^i(K^p)_{\mf{m}}$ and each place $\nu \mid p$ of $F$, the associated Galois representation $\rho_{x}|_{\Gal_{F_{\nu}}}$ is trianguline of parameter $(\delta + \rho_{\cyc}^{G})_{\nu} \cdot \delta'$, where $\delta'$ is an algebraic character, and the $\tau$-Sen weights are $\{\on{wt}_{\tau}(\delta_{1,\nu}) , \on{wt}_{\tau}(\delta_{2,\nu}) + 1, \dots, \on{wt}_{\tau}(\delta_{n,\nu}) + n-1\}$.
\end{theorem}
In essence, this result can be viewed as a $p$-adic (trianguline) local global compatibility result for $\GL_n$-eigenvarieties, which contextualizes recent progress in the Langlands program for CM fields in the framework of $p$-adic automorphic forms. A natural extension of \cref{theorem: MainTheorem1} would be to directly relate $\mc{E}^i(K^p)_{\mf{m}}$ to triangulation deformation spaces as studied in \cite{BHSAnnalen, BHSIHES}. See \cref{subsection: finalsubsection} for more discussion of this topic.

\subsection{History}

Let us review the history of relating finite slope $p$-adic automorphic forms (or eigenvarieties) to trianguline Galois representations. For $\GL_2/\Q$, this link was first established by \cite{Ki03} in the form of showing that if $f$ is an overconvergent $p$-adic eigenform of finite slope, then $\rho_f$ admits a crystalline period. Colmez \cite{Col08} then crucially reinterpetted admitting crystalline periods in terms of being a \textit{trianguline Galois representation}, i.e., the associated rigid analytic $(\varphi, \Gamma_K)$-module being totally reducible, which he studied in detail in the $2$-dimensional case. For eigenvarieties for definite unitary groups of any dimension, triangulinity results were proved in many cases by \cite{BC09} and later \cite{He12}. These results were vastly generalized with the development of the theory of families of $(\varphi, \Gamma_K)$-modules by Kedlaya--Pottharst--Xiao \cite{KPX14} and independently Liu \cite{Liu15}. In particular, their results imply there is a triangulinity result whenever the associated eigenvariety satisfies a suitable local-global compatibility for classical forms and the ``classical points'' are Zariski-dense (see \cite[Theorem 6.3.13]{KPX14}). 

Simultaneously, the theory of eigenvarieties for groups without discrete series was being developed by various authors \cite{EmertonInterpolation, AS08, Urb11, Han17, FuDerived}. Hansen \cite[Conjecture 1.2.2]{Han17} was the first to explicitly posit the connection between such eigenvarieties and trianguline representations. Our main result in essence proves his conjecture for the eigenvarieties of \cite{FuDerived} under some assumptions on $F$, after localizing at certain maximal ideals. The main novelty in the result we prove is that one \textit{does not expect} a Zariski-density of ``classical points'' for eigenvarieties for groups without discrete series, such as $\GL_n/F$ for $n \ge 2$ and $F$ an imaginary CM field! See \cite[\textsection 1]{CalMaz09} for further discussion of this issue.

\subsection{Proof method}

Let us discuss the proof strategy. We follow the lead of the 10 author paper \cite{10AuthorPaper}, which shows a local--global compatibility result in the ordinary case (see \cite[Theorem 5.5.1]{10AuthorPaper}). Their work proceeds integrally, which presents some additional complications, but is also easier in certain steps due to the (relative) simplicity of the ordinary case.

As in \cite{10AuthorPaper}, we prove our local-global compatibility result by passing to a larger unitary group $\wt{G} = U(n,n)/F^+$, and considering its eigenvariety (in a fixed middle degree!). It is useful to briefly mention how our eigenvarieties are constructed. The spaces $\mc{E}^i(K^p)_{\mf{m}}$ arise as the support of the Hecke algebra $\T(K^p)_{\mf{m}}$ acting on a certain coherent sheaf $\mc{M}_{\mf{m}}^i$ on the space $\widehat{T}$, whose global sections $\mc{M}_{\mf{m}}^i(\widehat{T}) = (\HH^i(J_{B_n}(\Pi(K^p)_{\mf{m}}^{\la})))'$ are the dual of a certain (derived version of) Emerton's Jacquet module applied to (a complex computing) completed cohomology. The method of attack is to realize this coherent sheaf in the boundary cohomology of $\wt{G}$, and interpolate from the interior. Namely, we have eigenvarieties $\mc{E}^d(\wt{K}^p)_{\wt{\mf{m}}}, \mc{E}_c^d(\wt{K}^p)_{\wt{\mf{m}}}, \mc{E}_{\partial}^i(\wt{K}^p)_{\wt{\mf{m}}}$ associated to the cohomology of $\wt{G}$, the compactly supported cohomology of $\wt{G}$, and the boundary cohomology of $\wt{G}$, respectively. These spaces arise as the Hecke module support of coherent sheaves $\mc{M}_{\wt{\mf{m}}}^d, \mc{M}_{c,\wt{\mf{m}}}^d, \mc{M}_{\partial,\wt{\mf{m}}}^i$, which fit into a short exact sequence (see \cref{lemma: FundamentalLongExactSequence})

\begin{equation}
\label{eq: FundamentalExactSequenceIntro}
0 \to \mc{M}_{\partial, \wt{\mf{m}}}^{d} \to \mc{M}_{\wt{\mf{m}}}^d\to \mc{M}_{c,\wt{\mf{m}}}^d \to \mc{M}_{\partial, \wt{\mf{m}}}^{d-1} \to 0.
\end{equation}
With this sequence in mind, our work essentially consists of four steps:
\begin{enumerate}
    \item ``\textit{Degree shifting}:'' relate the (co)homology $\mc{M}_{\mf{m}}^i$ of $\GL_n/F$ (in varying degrees $i$) to the boundary (co)homology $\mc{M}_{\partial, \wt{\mf{m}}}^{d-1}$.
    \item Show the boundary (co)homology $\mc{M}_{\partial, \wt{\mf{m}}}^{d-1}$ is ``small.''
    \item Show the middle degree (co)homology $\mc{M}_{c,\mf{m}}^d$ is ``big'' and then via Step 2 can be interpolated by classical \textit{cuspidal} automorphic forms for $\tG$.
    \item Use Steps 1 and 3 to deduce triangulinity results for $\GL_n/F$.
\end{enumerate}

We now say more about each step. An important player in our arguments is the so-called \textit{weight space} $\mc{W}:= \widehat{T_n(\prod_{\nu \mid p}\mO_{F_{\nu}})}$, the space of characters of the maximal compact torus $T_0 \subset T$. The projection map $\wh{T} \to \mc{W}$ induces a natural map $\kappa: \mc{E}^i(K^p)_{\mf{m}} \to \mc{W}$, and similar maps for the other eigenvarieties for $\tG$. 

On the level of spherical Hecke algebras, the unitary group $\wt{G}$ and $G := \GL_n/F$ are related via a Satake transfer map $\mc{S}:\wt{\T}(\wt{K}^p)_{\wt{\mf{m}}} \to \T(K^p)_{\mf{m}}$. Step 1 is then encapsulated in the following (see \cref{theorem: Embeddings of eigenvarieties} for a more precise statement):

\begin{theorem}
\label{theorem: Step1}
For all $i$ the map $\mc{S}: \wt{\T}(\wt{K}^p)_{\wt{\mf{m}}} \to \T(K^p)_{\mf{m}}$ induces a closed embedding $\mc{E}^i(K^p)_{\mf{m}} \hookrightarrow \mc{E}_c^d(\wt{K}^p)_{\wt{\mf{m}}}$ fitting into a diagram
\[
\begin{tikzcd}
    \mc{E}^{i}(K^p)_{\mf{m}} \ar[r, hook]\ar[d, "{\kappa}"] & \mc{E}_{c}^d(\wt{K}^p)_{\wt{\mf{m}}}\ar[d, "\wt{\kappa}"]\\
    \mc{W}\ar[r, "{w^{-1} \cdot_{\on{alg}} }"] & \mc{W}
\end{tikzcd},
\]
where $w \in W^{\oP}$ (see \cref{Section3: ParabolicInduction} for a definition) satisfies $\ell(w) = d - 1 - i$, and the bottom map denotes the dot action $\lambda \mapsto w^{-1} \cdot_{\on{alg}} \lambda := w^{-1}(\lambda + w_0^{\tG}\rho^{\tG}) - w_0^{\tG}\rho^{\tG}$, where $\rho^{\tG}$ is the half sum of positive roots for $\tG$, and $w_0^{\tG}$ is the long element of the Weyl group of $\tG$.
\end{theorem}
\noindent The proof consists of a computation of (a piece of) the nilpotent Lie algebra cohomology of a certain parabolic induction of a complex.

Step 2 can be summarized as the following (see \cref{corollary: SmallSupport} for notation):
\begin{theorem}
\label{theorem: SmallSupportIntro}
There is an admissible affinoid open cover $\bigcup_i U_i = \supp_{\mc{W}\times \G_m} J_B^{\vee}(C_{\bullet}^{\partial,\on{BS}}(\wt{K}^pI, D(I))) \subset \mc{W} \times \G_m$ such the image $\pi_{\mc{W}}=: W_i \subset \mc{W}$ under the map $\pi_{\mc{W}}: \mc{W}\times \G_m \to \mc{W}$ is an affinoid open, and for all $j\ge 0$, $M_{\partial, \wt{\mf{m}}}^j(U_i)$ is a finitely generated \emph{torsion} $\mc{O}_{\mc{W}}(W_i)$-module.
\end{theorem}

\noindent This result is not entirely new, as it follows the strategy of \cite[\textsection 4]{Han17}, who explained how to prove such a result for \textit{overconvergent cohomology}. In fact, our proof strategy is to adapt the tools Hansen uses (that is, technical tools about slope decompositions from \cite{AS08} and \cite{Urb11}) to our setting. The key point is that the \textit{classical} (cuspidal) cohomology of $\GL_n$ is only non-zero for certain weights $\lambda$, which live in a proper subvariety of weight space. A similar fact holds true for boundary cohomology. Then using a small slope classicality result, we can show many fibres of this boundary Jacquet module are zero, which shows torsion-ness.

For Step 3, we can phrase everything in terms of the middle degree eigenvariety $\mc{E}_{c}^d(\wt{K}^p)_{\wt{\mf{m}}}$:
\begin{theorem}
\label{theorem: Step 3}
    The eigenvariety $\mc{E}_c^d(\wt{K}^p)_{\wt{\mf{m}}}$ is equidimensional of dimension $\dim \mc{W} = 2n[F^+:\Q]$ with no embedded components, and the image of an irreducible component of $\mc{E}_c^d(\wt{K}^p)_{\wt{\mf{m}}}$ under the weight map $\wt{\kappa}$ is Zariski open in $\mc{W}$.     
    Moreover, $\mc{E}_c^d(\wt{K}^p)_{\wt{\mf{m}}}$ contains a Zariski-dense and accumulating set of points corresponding to classical cohomological cuspidal automorphic representations of $\wt{G}(\A_{F^+})$.
\end{theorem}

\noindent We note that in contrast to the torsion boundary $\mc{M}_{\partial, \wt{\mf{m}}}^{d-1}$, this middle degree unitary group (co)homology $\mc{M}_{c,\wt{\mf{m}}}^d$ will turn out to be locally \textit{torsion-free} over weight space! On the other hand, the fact that $\mc{M}_{\partial, \wt{\mf{m}}}^{d-1}$ is \textit{torsion} is absolutely crucial to the Zariski-density of classical cuspidal points in \cref{theorem: Step 3}. The reason is that while $\mc{E}^d(\wt{K}^p)_{\mf{m}}$ may not be equidimensional, it coincides with the usual eigenvariety constructed by Emerton, and thus many classical points contribute. We hope this strategy might be helpful to study other properties of eigenvarieties for groups without discrete series.

\begin{remark}[Completed versus overconvergent cohomology]
    A natural question is whether our results can be applied to other eigenvarieties constructed in the literature. For example, Hansen's original conjecture about triangulinity is for eigenvarieties coming from overconvergent cohomology. As is written, the only Step for which our argument seems to strongly use Jacquet functors is the ``degree shifting'' in Step 1. For partial progress of a degree-shifting nature in a more overconvergent context, see \cite{Fu22Kostant}. 
    
    While we do not pursue directly comparing eigenvarieties from completed and overconvergent (co)homology, we hope our results highlight the utility of using different features of the various formulations of eigenvarieties. Moreover, we mention ongoing work of Johansson--Tarrach which aims to compare (versions of) these two theories using Tarrach's theory of so-called $p$-arithmetic (co)homology \cite{Tarrach2023}.
\end{remark}

\begin{remark}
    Much of this paper requires doing homological algebra with locally analytic representations on LB or Fr\'{e}chet spaces, which presents annoying issues. Naturally, our saving grace is that eigenvarieties are constructed from \textit{coadmissible} modules over certain Fr\'{e}chet algebras, which form a nice abelian category where objects have canonical topologies and all maps are automatically strict. 
    
    One might wonder or expect that the notion of a ``derived eigenvariety'' should most naturally approached using the recent theory of condensed mathematics, specifically via the solid locally analytic representation theory developed in \cite{SolidLocallyAnalyticI,SolidLocallyAnalyticII}.  The current draft opts to work in more classical language, partially out of ignorance, and also because we feel it is useful to delineate where such theoretical advantages are truly necessary for proofs. As such, we do not use the theory of solid modules to run our arguments.

    Nevertheless, for future applications it might be convenient to recast parts of our work in the solid language. We hope to return to this in the near future.
\end{remark}

\subsection{Related work for CM fields}

We now discuss the relation between our work and the many recent works on proving $p$-adic Hodge theoretic properties (i.e. local-global compatibility at $\ell = p$) of Galois representations attached to (cohomological cuspidal) Hecke eigensystems for $\GL_n$ over a CM field. For any $n > 2$ the only known method for constructing Galois representations associated to Hecke eigenclasses for $\GL_n/F$ comes by passing to the $2n$-variable unitary group $\wt{G}$. This strategy was first realised by \cite{HLTT16} for characteristic $0$ eigenclasses, and then also by Scholze \cite{Sch15} for torsion classes. In fact, the construction in \cite{HLTT16} realises $R_{\pi} \oplus R_{\pi}^{c, \vee}$ (twisted) as a limit of \textit{finite slope} cohomology classes on the ordinary locus of the Shimura variety for $\wt{G}$.\footnote{In fact, it seems that Skinner was the first to envision such a strategy, and had in mind to use eigenvarieties for these constructions.} Although congruences are not sufficient for proving $p$-adic Hodge theoretic properties, this fact perhaps makes our work seem more plausible. 

Since then (due to motivations from the Calegari--Geraghty method for proving automorphy results) there has been much recent work on proving such properties. With the crucial input from Caraiani--Scholze's work on torsion vanishing \cite{CS19}, \cite[\textsection 4, \textsection 5]{10AuthorPaper} developed a robust strategy from local-global compatibility via so-called ``degree--shifting arguments.'' We briefly mention works which pursue this strategy: all under various technical assumptions, \cite[Theorem 4.5.1, Theorem 5.5.1]{10AuthorPaper} handle the Fontaine--Laffaille and ordinary cases, \cite[Theorem 4.3.3]{AC24} handles a rather general case in characteristic $0$ (via ideas of Caraiani--Newton), \cite{CN2023} proves the crystalline torsion case, and finally \cite{Hevesi23} proves the general potentially semistable torsion case. We note that these last two works must assume $F$ is not imaginary quadratic, whereas \cite[Theorem 5.5.1]{10AuthorPaper} and our work need no such assumption.
\subsection{What our work does}
Our work uses \cite{CS19} in an essential way (for example, to get \eqref{eq: FundamentalExactSequenceIntro}), and is the reason for the ``decomposed generic'' assumption in \cref{theorem: MainTheorem1}. We also use \cite{Sch15} to get the Galois representation $\rho$ on $\mc{E}^i(K^p)_{\mf{m}}$, although we likely could use our method to independently reconstruct $\rho$. Otherwise, our work does not rely on the further developments mentioned above.

As previously mentioned, our argument most closely mirrors the ordinary case of \cite[\textsection 5]{10AuthorPaper}. The key new inputs in our case are a degree shifting argument in the locally analytic context (where functions on Bruhat strata are more complicated) (Step 1), and the analytic continuation result for $\mc{E}_c^d(\wt{K}^p)_{\wt{\mf{m}}}$ (Step 3). While many of the arguments come via standard methods in theory of eigenvarieties, the overall strategy to combine these methods seems to not have been exploited before. 

We note another technical difference concerning the unitary group $\wt{G}$. The authors in \textit{loc. cit.} need to choose certain dominant weights for $\wt{G}$ very carefully, and then appeal to a trick with the centre to handle all possible degrees of cohomology for $\GL_n/F$. Our general analytic continuation result is much softer, and does not require explicit care in choosing weights only contributing to the interior.

\subsection{Contents}
We now review the contents of the article. \cref{Section2: Cohomology} reviews preliminaries on the (co)homology of locally symmetric spaces, recalls the derived Jacquet functors of \cite{FuDerived}, and shows two key properties: the sequence \eqref{eq: FundamentalExactSequenceIntro} and \cref{theorem: SmallSupportIntro}. \cref{Section3: ParabolicInduction} discusses some constructions related to parabolic inductions and filtrations by Bruhat strata. Then \cref{Section4: DegreeShifting} uses Lie algebra cohomology to relate the derived Jacquet functors for $\GL_n$ (of varying degrees) to a derived Jacquet module for $\wt{G}$ (in middle degree). This relation proves \cref{theorem: Step1}. \cref{Section5: UnitaryGroup} then establishes the necessary results about unitary group in middle degree, proving \cref{theorem: Step 3}. Lastly, \cref{Section6: Trianguline} combines the previous work to prove \cref{theorem: MainTheorem1}, and then discusses some further questions about trianguline deformations.

\subsubsection*{Acknowledgements.} 
We are immensely grateful to Richard Taylor for suggesting these lines of research, and for his continuous advice and support while working on this paper. We thank Lie Qian for inspiring discussions in the early stages of this paper and making us appreciate Lie algebra cohomology, as well as for various conversations afterwards. We thank Weibo Fu for laying the foundations that make this work possible, and for a discussion early on about his impressions on eigenvarieties. We thank Christian Johansson for some very helpful discussions about comparing eigenvarieties, overconvergent homology, and especially helping with \cref{lemma: FundamentalLongExactSequence}. We thank Johansson and also Brian Conrad for pointers regarding \cref{lemma: LocalZariskiDensity}. Lastly, we thank Eugen Hellmann for many discussions and his continued interest in our work. We also thank Sean Cotner, Matt Emerton, Dongryul Kim, and James Newton for helpful discussions. This research was conducted while the author was an NSF Graduate Fellow, and we warmly acknowledge their support.

\subsection{Notation and conventions}
% Number fields
For $F/\Q$ be a number field let $S_p(F)$ denote the set of $p$-adic places in $F$. For any finite place $\nu \in S_p(F)$, we let $F_{\nu}$ denote the associated complete local field, a finite extension of $\Q_p$. Let $\mc{O}_{F_{\nu}}\subset F_{\nu}$ its ring of integers, $\mf{m}_{\nu}\subset \mc{O}_{F_{\nu}}$ is maximal ideal, and $k_{\nu}:= \mc{O}_{F_{\nu}}/\mf{m}_{\nu}$ the residue field. We will often let $\varpi_{\nu}\in \mf{m}_{F_{\nu}}$ denote a uniformizer. We denote $q_{\nu}:= \#k_{\nu}$. Let $F_p:= F\otimes_{\Q}\Q_p$. Note we have a decomposition $F_p \simeq \prod_{\nu \in S_p(F)}F_{\nu}$. Similarly, define $\mc{O}_{F,p}:= \mc{O}_F\otimes_{\Z}\Z_p$.

%Algebra
If $R$ is a ring, let $\mathsf{D}(R)$ denote the derived category of $R$-modules, and $\mathsf{K}(R)$ the homotopy category of complexes of $R$-modules.

% Topological vector spaces
Let $L/\Q_p$ be a finite extension. If $V$ is a locally convex $L$-vector space, we let $V'$ denote its strong dual.
Recall that a locally convex topological $L$-vector space $V$ is of \textit{compact type} if it can be written as an inductive limit $V \simeq \varinjlim_n V_n$ for $\{V_n\}$ a countable inductive system of $L$-Banach spaces with compact, injective transition maps, where $V$ is endowed with the inductive limit topology.
These spaces are naturally reflexive, and $V \mapsto V'$ provides an anti-equivalence between compact type spaces and so-called nuclear Fr\'{e}chet spaces over $L$ (see \cite[Theorem 1.3]{STJAMS}).

Given a complex $C^{\bullet}$ of locally convex topological $L$-vector spaces with continuous differentials, we often topologize $\HH^{\ast}(C^{\bullet})$ with the induced subquotient topology. Note these spaces need not be Hausdorff.

% representations of p-adic groups

Let $H$ be a compact $p$-adic Lie group. Then we denote $\mO_L[[H]]:= \varinjlim_{m \ge 0,H' \lhd H}\mO_L/p^m[H/H']$, which is naturally a noetherian ring. We also let $L[[H]]:= \mO_L[[H]]\otimes_{\mO_L}L$. 
Let $G$ be a general locally $\Q_p$-analytic group. Similarly, we let $\mc{C}(G, L)$ denote the space of continuous functions from $G \to L$, and $\mc{C}^{\la}(G, L)$ the subspace of locally analytic functions. For $H$ compact, $\mc{C}^{\la}(H, L)$ is naturally a space of compact type. We let $\mathsf{Ban}_L^{\on{ad}}(G)$ denote the category of admissible Banach representations of $G$, and $\Rep_{\rm{la.c}}(G)$ the category of locally representations of $G$ on compact type vector spaces (see \cite[\textsection 3.1]{Emerton_Jacquet_I}). Let be $H_1\supset H_2$ $p$-adic Lie groups such that $H_1/H_2$ is compact. Then if $\pi$ is  Banach representation of $H_2$ we let $\on{ct}-\Ind_{H_2}^{H_1}\pi$ be the continuous induction, the set of continuous functions $H_1 \to \pi$ such that $f(h_2g) = \pi(h_2)f(g)$ for $h_2\in H_2$. Similarly if $\pi\in \Rep_{\rm{la.c}}(H_2),$ we denote $\Ind_{H_2}^{H_1}\pi$ the locally analytic induction defined via the same equivariance.

Similarly, let $D(G, L):= (\mc{C}^{\la}(G, L))'$ be the distribution algebra associated to $G$. When $G$ is compact, this is naturally a nuclear $L$-Fr\'{e}chet algebra. Suppose also that $G$ is a compact, uniform pro-$p$ group. For any $1/p< r < 1$, we let $D_r(G, L)$ and $D_{<r}(G,L)$ be the $L$-Banach spaces defined in \cite[\textsection 4]{ST03}. For any $1/p\le r'< r <1$ we have embeddings $D_{r}(G,L) \hookrightarrow D_{<r}(G,L) \hookrightarrow D_{r'}(G,L)$. Moreover if $r\in p^{\Q}$ then the maps $L[[G]]\hookrightarrow D_{r}(G,L) \hookrightarrow D_{<r}(G,L)$ are flat, and all three algebras are (left and right) noetherian integral domains by \cite[\textsection 4]{ST03}. Lastly, the distribution algebra is a limit $D(G,L) = \varprojlim_rD_r(G,L) = \varprojlim_rD_{<r}(G,L)$. 
If $H$ is in fact an abelian p-adic Lie group, we let $\wh{H}$ (the space of \textit{characters of }$H$) be the rigid space over $\Q_p$ representing the functor $\Sp(A) \mapsto \Hom_{\on{cont}}(H, A^{\times})$. See \cite[\textsection 6.4]{Emerton_la_reps} for more details.
% Galois theory

If $G$ is a split connected reductive group with maximal torus $T\subset G$, denote by $W_{G}$ the associated Weyl group. Let $w_0^G$ denote the long element. Let $X^{\ast}(T)$ be the character lattice. If we also choose a Borel $B\supset T$, we can define a notion of positive roots, and then form $\rho^G\in X^{\ast}(T)$, the half sum of positive roots. For any $w \in W_G$, we may define the dot action $w\cdot \lambda := w(\lambda + \rho^G) - \rho^G$.

Let $K/\Q_p$ be a finite extension. We take the geometric normalisation of local class field theory by taking the Artin map $\on{Art}_{K}: K^{\times} \xrightarrow{\sim} W_{K}^{\rm{ab}}$ to send uniformizers to geometric Frobenius. Let $\rec_K$ be the local Langlands correspondence as normalised in \cite{HT01}. We denote the cyclotomic character by $\varepsilon_{\on{cyc}}:\Gal_K  \to \ol{\Q}_p^{\times}$. Let $\chi: K^{\times} \to \ol{\Q}_p^{\times}$ be a continuous (and thus locally $\Q_p$-analytic) character. We then define its \textit{weight} to be $\on{wt}(\chi) = (\on{wt}_{\tau}(\chi))_{\tau \in \Hom(K, \ol{\Q}_p)}\in K\otimes_{\Q_p}\ol{\Q}_p$ by $\on{wt}(\chi):= -\frac{d}{dt}\chi(\exp(tx))|_{t=0}\in \Hom(K,\Q_p)\otimes_{\Q_p}\ol{\Q}_p\simeq K\otimes_{\Q_p}\ol{\Q}_p$ where the identification $\Hom(K,\Q_p) \simeq K$ is via the trace pairing. For any $\ul{k} = (k_{\tau})_{\tau} \in \Z^{\Hom(K,\ol{\Q}_p)}$, we define the character $z^{\ul{k}}: K^{\times}\to \ol{\Q}_p^{\times}$ via the formula $x \mapsto \prod_{\tau\in \Hom(K, \ol{\Q}_p)}(\tau(x))^{k_{\tau}}$. By construction, $\weight(x^{\ul{k}}) = -\ul{k}$. We call such characters \textit{algebraic}. We normalise Hodge-Tate(-Sen) weights so that $\on{HT}(\varepsilon_{\on{cyc}}) = -1$.
We let $\mc{R}_K$ be the Robba ring for $K$ and we denote $R_{L,K}:= L\otimes_{\Q_p}\mc{R}_K$ (see \cite[\textsection 2.2]{KPX14}).

\section{Preliminaries on (co)homology}
\label{Section2: Cohomology}

In this section we collect preliminaries of various sorts about (co)homology, Jacquet functors, and Hecke algebras. Part of this section builds on the foundations from \cite{FuDerived} by studying derived Jacquet functors on the homology side. The key results we will need in the sequel are \cref{lemma: FundamentalLongExactSequence} (sequence \eqref{eq: FundamentalExactSequenceIntro}) and \cref{corollary: SmallSupport} (\cref{theorem: SmallSupportIntro}).

\subsection{Locally symmetric spaces and completed (co)homology}
\label{subsection: BeginningGeneralitiesOnSymmetricSpaces}
We follow the setup of \cite[Section 2]{10AuthorPaper}.
Suppose $F$ is a number field and $G$ is (an $\mc{O}_F$-model of) a connected linear algebraic group over $F$. Let $X^G$ be the symmetric space for $\on{Res}_{F/\Q}G$, a homogeneous space for $G(F\otimes_{\Q}\R)$ (see \cite[Section 2]{BorelSerre} or \cite[3.1]{NT16}). 

Recall a compact open subgroup $K \subset G(\A_{F}^{\infty})$ is \textit{good} if it is neat and of the form $\prod_{\nu}K_{\nu}$ where $\nu$ runs over finite places of $F$. Let $K^p = \prod_{\nu \nmid p}K_{\nu}$. We then consider the quotients
\[
X_{K}^{G} := G(F)\backslash( X^{G}\times G(\A_F^{\infty})/K), \quad X_{K^p}^{G} := G(F)\backslash (X^G \times G(\A_F^{\infty})/K^p),
\]
where for the latter space we give $G(\A_F^{\infty})$ the discrete topology. These spaces then both have the structure of smooth manifolds, and the projection map $\pi_{K_p}:X_{K^p}^{G} \to X_{K}^{G}$ induces an identification $X_{K}^G = X_{K^p}^{G}/K_p$.

We likewise let $\overline{X}^G$ be the partial Borel--Serre compactification of $X^G$ as in \cite[7.1]{BorelSerre}. Defining
\[
\overline{X}_{K}^{G} := G(F)\backslash (\overline{X}^{G}\times G(\A_F^{\infty})/K), \quad \overline{X}_{K^p}^{G} := G(F)\backslash (\overline{X}^G \times G(\A_F^{\infty})/K^p),
\]
we again have an identification $\overline{X}_K^G = \overline{X}_{K^p}^G/K_p$, and in fact $\overline{X}_K^G$ is a compact smooth manifold with corners with interior $X_{K}^G$. Additionally, the inclusion $X_{K}^G \hookrightarrow \overline{X}_{K}^G$ is a homotopy equivalence. Such an equivalence is induced by a homotopy inverse $p: \ol{X}^G \to X^G$. In particular, this map $p$ gives us homotopy inverses $p_{K^p}: \ol{X}_{K^p}^G \to X_{K^p}^{G}$ and $p_{K}: \ol{X}_{K}^G \to X_{K}^{G}$ so that $p_{K^p}$ is $G_p:=G(F_p)$-equivariant and upon quotienting by $K_p$ induces $p_K$. 

We also define $\partial X^G:= \overline{X}^G - X^G$  and using this space we define $\partial X_{K}^G$ and $\partial X_{K^p}^G$ via analogous double quotients. For $K'\subset K \subset G(\A_F^{\infty})$ good compact open subgroups we then have diagrams 
\begin{equation}
\label{eq: HomotopyEquivalencesCompatibleWithLevel}
\begin{tikzcd}
    X_{K'}^G \ar[r, hook]\ar[d]& \overline{X}_{K'}^G\ar[d]\\
    X_{K}^G \ar[r, hook]& \overline{X}_{K}^G
\end{tikzcd} \quad\begin{tikzcd}
    X_{K'}^G \ar[d]& \overline{X}_{K'}^G\ar[d]\ar[l, "{p_{K'}}"]\\
    X_{K}^G & \overline{X}_{K}^G\ar[l, "{p_{K}}"]
\end{tikzcd}.
\end{equation}
 We also have analogous diagrams for $X_{K^p}^G$ and the boundary spaces, and these diagrams are also compatible with the projection maps $\pi_{K_p}$. 

We are interested in the (co)homology of $X_K^G, \partial X_K^G$, in various guises. For any topological space $M$ we let $C_{\bullet}(M)$ denote the complex of singular chains. Note then that $C_{\bullet}(X_{K^p}^{G})$ carries a right action of $G_p:= G(F_p)$ by right translation. Moreover, if $K_p\subset G_p$ is a compact open subgroup such that $K^pK_p$ is good, then $C_{\bullet}(X_{K^p}^{G})$ is in fact a complex of free $\Z[K_p]$-modules. For any left $\Z[K_p]$-module $M$, we let $C_{\bullet}(K^pK_p, M) := C_{\bullet}(X_{K^p}^G)\otimes_{\Z[K_p]} M$, the homology complex with coefficients $M$.

For the purposes of p-adic automorphic forms, it is important to represent $C_{\bullet}(X_{K^p}^G)$ by complexes with better finiteness properties. Not surprisingly, we essentially use simplicial homology. 
Fix $K^p$, and let $K_p\subset G_p$ such that $K = K^pK_p\subset G(\A_F^{\infty})$ is a good subgroup. 
Now fix a finite triangulation $\tau$ of $\overline{X}_{K^pK_p}^G$, chosen so that $\partial \overline{X}_{K^pK_p}^G$ occurs as a subtriangulation. Let $T_{n}^{\circ}(X_K^G)$ be the set of simplices of dimension $n$ occurring in $\tau$. Let $T_{n}^{\circ}(X_{K^p}^G)$ be the set of simplices lying above those in $T_{n}^{\circ}(X_{K}^G)$. Further fixing a simplex $\Delta \in T_n^{\circ}(X_K^G)$, let $T_n^{\circ}(X_{K^p}^G)/\Delta$ be the set of simplices in $\overline{X}_{K^p}$ lying over $\Delta$.

We then define $C_{\bullet}^{\on{BS}}(K^p):= \bigoplus_{\Delta\in T_{\bullet}^{\circ}(X_{K^p}^G)}\Z\cdot \Delta$. Since the projection map $\pi_{K_p}: \ol{X}_{K^p}^G \to \ol{X}_{K^pK_p}^G$ is a $K_p$-torsor (and thus a covering map since we are endowing $K_p$ with the discrete topology), by the homotopy lifting lemma we have that any $\Delta \in T_{n}^{\circ}(X_{K^pK_p}^G)$ has a lift to $\ol{X}_{K^p}^G$. Choosing such lifts, we may rewrite 
$C_{\bullet}^{\on{BS}}(K^p) = \bigoplus_{\Delta \in T_{\bullet}^{\circ}(X_{K}^G)}\bigoplus_{\Delta':\pi_{K_p}(\Delta') = \Delta} \Z \simeq \bigoplus_{\Delta \in T_{\bullet}^{\circ}(X_{K}^G)}\Z[K_p]$, so that $C_{\bullet}^{\on{BS}}(K^p)$ is in fact a bounded complex of free $\Z[K_p]$-modules.

By the equivalence of simplicial and singular homology, the natural $\Z[K_p]$-linear map 
\[
\iota_{\tau}: C_{\bullet}^{\on{BS}}(K^p) \to C_{\bullet}(K^p)
\]
is in fact a homotopy equivalence of complexes of $\Z$-modules. But in fact, \cite[Corollary 10.4.7]{WeibelHomAlg} (using that $C_{\bullet}(K^p)$ is a bounded \textbf{above} complex of free $\Z[K_p]$-modules), we may choose a $\Z[K_p]$-linear homotopy inverse $p_{\tau}: C_{\bullet}(K^p) \to C^{\on{BS}}(K)$. Likewise, let $C_{\bullet}^{\partial}(K^p) :=C_{\bullet}(\partial X_{K^p}^G)$ be the complex of singular chains for the boundary of the Borel--Serre compactification. We can then define $C_{\bullet}^{\on{BM}}(K^p)$ to be the cone of the natural inclusion
\[
0 \to C_{\bullet}^{\partial}(K^p) \to C_{\bullet}(K^p) \to C_{\bullet}^{\on{BM}}(K^p) \to 0.
\]
Using our choice $\tau$ of finite triangulation of $\ol{X}_{K^pK_p}^G$ which has $\partial X_{K^pK_p}^G$ as a finite subtriangulation, we can define an analogous chain complex $C_{\bullet}^{\partial,\on{BS}}(K^p)$ of finite free $\Z[K_p]$-modules with a $\Z[K_p]$-linear homotopy equivalence $C_{\bullet}^{\partial,\on{BS}}(K^p) \xrightarrow{\iota_{\partial}}C_{\bullet}^{\partial}(K^p)$. We may likewise define $C_{\bullet}^{\on{BM}, \on{BS}}(K)$ to be the cone of the map $C_{\bullet}^{\on{BS}}(K) \to C_{\bullet}^{\on{BS}}(K)$. For any set of decorations $\ast = \emptyset, \partial, \on{BS}, \on{BM}, (\on{BM},\on{BS}), (\partial, \on{BS})$ and any left $K_p$-module $M$, we define $C_{\bullet}^{\ast}(K^pK_p, M):= C_{\bullet}^{\ast}(K^p)\otimes_{\Z[K_p]}M$. We then have a commuting diagram of complexes
\[
\begin{tikzcd}
C_{\bullet}^{\partial}(K^pK_p, M)\ar[r, "f"]&C_{\bullet}(K^pK_p, M)\ar[r, "g"]& C_{\bullet}(K^pK_p, M)
\\
 C_{\bullet}^{\partial, \on{BS}}(K, M)\ar[u, "\iota_{\tau, \partial}"]\ar[r, "f^{\on{BS}}"]& C_{\bullet}^{\on{BS}}(K^pK_p, M) \ar[r, "g^{\on{BS}}"]\ar[u, "\iota_{\tau}"]& \on{C}_{\bullet}^{\on{BM},\on{BS}}(K, M)\ar[u, "{\iota_{\tau, \on{BM}}}"]
\end{tikzcd}
\]
where $\iota_{\tau, \partial}, \iota_{\tau}, \iota_{\tau,\on{BM}}$ are each quasi-isomorphisms. Again, we can use \cite[Corollary 10.4.7]{WeibelHomAlg} to form homotopy inverses $p_{\tau, \partial}, p_{\tau, \on{BM}}$. In the sequel will usually suppress the triangulation $\tau$ from the notation, even though these homotopy equivalences depend on such a choice.
\begin{remark}
    Unfortunately, it is not clear that $p_{\tau, \partial}, p_{\tau}$ in general can be chosen to be compatible on the nose, i.e. that $f^{\on{BS}} \circ p_{\tau, \partial} = p_{\tau} \circ f$. Of course, these compositions induce the same map on homology.
\end{remark}

\subsubsection{Completed (co)homology}
Let $L/\Q_p$ be a finite extension. Recall completed homology and cohomology are defined by 
\[
\wt{\HH}_{\ast}(K^p,\mO_L):= \varprojlim_{K_p \to 1}\HH_{\ast}(X_{K^pK_p}^{G}, \mO_L),\quad \wt{\HH}^{\ast}(K^p, \mO_L):= \varprojlim_{n}\varinjlim_{K_p \to 1}\HH^{\ast}(X_{K^pK_p}^{G}, \mO_L/p^n).
\]
Likewise, we let $\wt{\HH}_{\ast}(K^p,L):=\wt{\HH}_{\ast}(K^p,\mO_L)\otimes_{\mO_L}L, \wt{\HH}^{\ast}(K^p,L):=\wt{\HH}^{\ast}(K^p,\mO_L)\otimes_{\mO_L}L$.
The former groups $\wt{\HH}_{\ast}(K^p,\mO_L)$ are finitely generated modules over $\mO_L[[K_p]]$, while the latter $\wt{\HH}^{\ast}(K^p, L)$ are admissible Banach space representations of $G(F_p)$.  Recall that \cite[Theorem 1.1 (3)]{CalEmSurvey} implies $\wt{\HH}^{\ast}(K^p,L)\simeq \Hom_{\on{cont}}(\wt{\HH}_{\ast}(K^p,L), L)$ are dual under the natural anti-equivalence between finitely generated $L[[K_p]]$-modules and admissible Banach $K_p$-representations (see \cite[Theorem 3.5]{ST02}).
Connecting to our previous definitions, a basic fact is that $\HH_{\ast}(C(K^pK_p, \mO_L[[K_p]]))\simeq \wt{\HH}_{\ast}(K_p, \mO_L)$ (see \cite[Proposition 6.3.5]{JNWE24}), and thus $C^{\on{BS}}(K, \mO_L[[K_p]])$ is a complex of finite free $\mO_L[[K_p]]$-modules whose homology computes $\wt{\HH}_{\ast}(K^p, \mO_L)$. Analogous statements hold true for boundary and Borel--Moore homology.

We may also consider the space $\wt{\HH}^{\ast}(K^p, L)^{\la}$ of locally $\Q_p$-analytic vectors. For any locally $\Q_p$-analytic group $H$, write for short  $D(H) = D(H,L) = (\mc{C}^{\la}(H, L))'$. Recall then by \cite[Proof of Prop 7.1]{ST03} that if $V$ is an admissible Banach representation of $H$ then $(V^{\la})' = D(H)\otimes_{L[[H]]}V'$, and moreover $D(H)$ is faithfully flat over $L[[H]]$. These results immediately imply that the complex $C^{\on{BS}}_{\bullet}(K, D(K_p))$ computes the dual of locally analytic completed cohomology $\left(\wt{\HH}^{\ast}(K^p, L)^{\la}\right)'$.

\subsubsection{Complexes for completed cohomology} 
\label{subsubsection: ComplexesDefininfCompletedCohomology}
We also needs nice complexes computing cohomology. While one might expect that nice complexes (at least rationally) will just follow from taking the dual of the picture just described, there are potentially issues in dualizing the larger complex $C(K^p, L[[K_p]])$, which would make it harder to define an action of $G(F_p)$ up to homotopy. So instead we will recall the setup from \cite{FuDerived}.

Continue to take $K^pK_p\subset G(\A_F^{\infty})$ to be good.
Fix a countable basis of open normal subgroups of $K_p$, denoted by $(K_p^r)_{r = 1}^{\infty}$, and set $X_r:= X_{K^pK_p^r}^G, \overline{X}_r:= \overline{X}_{K^pK_p^r}^G$. Then we have a tower of maps 
\[
\begin{tikzcd}
    \dots \ar[r]&X_r\ar[r]\ar[d, hook]&\dots \ar[r]&X_1 \ar[r]\ar[d, hook]& X_0\ar[d, hook]\\
    \dots \ar[r]&\overline{X}_r\ar[r]&\dots \ar[r]&\overline{X}_1 \ar[r]& \overline{X}_0.
\end{tikzcd}
\]

Fixing again a triangulation $\tau$ of $\ol{X}_0$, let $T_{n}(X_0)$ be the set of all singular $n$-simplices of $\overline{X}_0$, we let $T_{n}^{\circ}(X_0)$ and the set of $n$-simplices occurring in $\tau$. As in \cref{subsection: BeginningGeneralitiesOnSymmetricSpaces}, we likewise define $T_{n}(X_r)$, $T_{n}^{\circ}(X_r)$, $T_n^{\circ}(X_r)/\Delta$. Then look at the singular (co)chain complexes 
\begin{align*}
    A_{r,s}^n := \prod_{\Delta \in T_{n}(X_r)}\Hom\left(\Z \Delta, \mc{O}_L/p^s\right) = \Hom(\bigoplus_{\Delta \in T_n(X_r)}\Z \Delta, \mc{O}_L/p^s)
\end{align*}
\begin{align*}
    \Pi_{r,s}^n := \bigoplus_{\Delta \in T_n^{\circ}(X_r)}\Hom(\Z \Delta, \mc{O}_L/p^s),
\end{align*}
As before, we can choose a $\Z[K_p]$-equivariant homotopy inverse $p$ of the inclusion $i:C_{\bullet}^{\on{BS}}(K^p) = \bigoplus_{T_{\bullet}^{\circ}(X_{K^p}^G)} \Z \to \bigoplus_{T_{\bullet}(X_{K^p}^G)}\Z = C(K^p)$. Such a map induces $\Z[K_p/K_p^r]$-linear homotopy inverses to the inclusions $\bigoplus_{T_{\bullet}^{\circ}(X_r)} \Z \to \bigoplus_{T_{\bullet}(X_r)} \Z$. 
Such inverses give us compatible $K_p/K_p^r$-linear homotopy inverses to the $K_p/K_p^r$-equivariant quasi-isomorphisms, $A_{r,s}^{\bullet} \to \Pi_{r,s}^{\bullet}$. Then upon setting $A^{\bullet}:= \varprojlim_s\varinjlim_rA_{r,s}, \Pi^{\bullet}:= \varprojlim_s\varinjlim_{r}\Pi_{r,s}^{\bullet}$, in the limit we get a $K_p$-equivariant homotopy equivalence $\begin{tikzcd}A^{\bullet}\ar[r, "p", shift left]  &\Pi^{\bullet}\ar[l, "i", shift left]\end{tikzcd}$.

By \cite[Proposition 1.2.12]{EmertonInterpolation}, we have $\HH^{\ast}(A^{\bullet}) \simeq \wt{\HH}^{\ast}(K^p, \mO_L)$ computes completed cohomology. It is also evident by the above torsor description that as $K_p$-representation, we have continuous isomorphisms $\Pi^{n} \simeq \mc{C}(K_p, \mc{O}_L)^{\oplus m_n}$ for some $m_n \ge 0$, and so $\Pi^{\bullet}\otimes_{\mc{O}_L}L$ is in fact a complex of \textbf{injective} admissible Banach representations of $K_p$. Thus, in the derived category of admissible Banach $K_p$-representations, $\Pi^{\bullet}$ computes completed cohomology. Likewise, by \cite[Proposition 5.3]{FuDerived}, $A\otimes_{\mc{O}_L}L$ is a Banach space with a continuous action of $G(F_{p})$. 

We can also make these same definitions for compactly supported cohomology and boundary cohomology and get complexes and $K_p$-equivariant homotopy equivalences $\begin{tikzcd}A_c^{\bullet}\ar[r, "p", shift left]  &\Pi_c^{\bullet}\ar[l, "i", shift left]\end{tikzcd}$, $\begin{tikzcd}A_{\partial}^{\bullet}\ar[r, "p", shift left]  &\Pi_{\partial}^{\bullet}\ar[l, "i", shift left]\end{tikzcd}$. Then there is a short exact sequence $0 \to \Pi_c^{\bullet} \to \Pi^{\bullet} \to \Pi_{\partial}^{\bullet} \to 0$ of complexes of injective admissible $K_p$-representations and a short exact sequence of complexes $0 \to A_c^{\bullet} \to A^{\bullet} \to A_{\partial} \to 0$ which are continuous maps of Banach spaces, which are compatible with the quasi-isomorphisms $\Pi_c^{\bullet} \to A_c^{\bullet}, \Pi^{\bullet} \to A^{\bullet}, \Pi_{\partial}^\bullet \to A^{\bullet}$. Again, we stress that there is no obvious reason the homotopy inverses $i$ should be compatible on the nose with these sequences. In the sequel, we will typically write $\Pi= \Pi(K^p)$ for the complexes constructed in this subsection.

\subsection{Review of $G$-extensions and derived Jacquet functors}

We review the theory of $G$-extensions that Fu uses to define his derived Jacquet functors of completed cohomology.
\begin{definition}
\label{def: Gextension}
    Let $\mathbf{H}/\mO_F$ be an algebraic group. Let $H = \mathbf{H}(F_p)$ and $H_0\subset H$ a compact open subgroup. An \textit{$H$-extension} is the data $(C, C', i, p)$ where $C$ is a bounded complex of admissible Banach representation of $H_0$, $C'$ is a bounded complex of Banach representations of $H$, and $i,p$ are $H_0$-equivariant, continuous homotopy equivalences
   $\begin{tikzcd}C\ar[r, "i"]&
    \ar[l, "p", shift left]C'\end{tikzcd}$.
\end{definition}

\begin{remark}
    Later on in \cref{Section3: ParabolicInduction} it will be convenient for us to call certain homotopy equivalences between complexes ``$B$-extensions'' even though they do not take the form above. In these cases, we will simply be referring to a $B_0$-equivariant homotopy equivalence to a complex with an action of $B \supset B_0$, where $B_0$ is a compact open subgroup. However, in our cases these $B$-extensions will be induced by an $H$-extensions of the form mentioned above.
\end{remark}

The motivation for this of course is that $(\Pi(K^p)^{\bullet}, A^{\bullet}, i, p)$ with a choice of homotopy equivalences form a $G(F_p)$-extension in the sense of \cref{def: Gextension}.

Suppose now that our algebraic group $G/F$ is in fact a connected reductive group such that $G$ is split over $F_{\nu}$ for all $\nu \mid p$. Let $(\B,\N)$ be a Borel and maximal torus of $G$ with Levi decomposition $\B = \T \N$ for $\T$ a maximal torus. Write $G_p = G(F_p), B = \B(F_p), N = \N(F_p), T = \T(F_p)$. Fix $K_p$ a compact open subgroup of $G_p$ and for any subgroup $H \subset G_p$ set $H_0:= H \cap K_p$. Then set $T^{+}:= \{t \in T: tN_0t^{-1} \subseteq N_0\}$.

Let $C$ be a complex of injective admissible Banach $K_p$-representations, equipped with a $G_p$-extension. We are interested in its locally analytic vectors $C^{\la, i}:= (C^{i})^{K_p-\la}$. The functor of locally analytic vectors is exact on admissible representations by a well-known result of Schneider--Teitelbaum \cite[Theorem 7.1]{ST03}. Likewise, \cite[Lemma 3.8, Corollary 3.9]{FuDerived} show that from a derived category perspective it is reasonable to consider the complex of $N_0$-invariants $C^{\la, N_0}$. 
Now let $z \in T^+$ be an element such that $\bigcap_{n=1}^\infty z^nN_0z^{-n} = 1$. Using the $G_p$-extension, we may define an action of $z$ on $C^{\bullet}$ via the formula $\wt{z} \cdot v:= (p \circ z \circ i)v$. This action commutes with the action of other elements of $T^+$ given by a similar formula, up to homotopy.
On $C^{\la, N_0}$ we may then define an action of $z$ via the formula
\[
U_zv:= \sum_{n \in N_0/zN_0z^{-1}}n\cdot\wt{z}\cdot v.
\]
\begin{remark}
\label{remark: UzOperatorNormalisation}
    We note that throughout this paper, our actions of $U_z$ differ from Emerton's normalisation in \cite{Emerton_Jacquet_I}, where he multiplies by an averaging factor $\delta_B(z) = [N_0: zN_0z^{-1}]^{-1}$. Thus, our action of $U_z$ is Emerton's $U_z$ action twisted by the modulus character $\delta_B^{-1}$. We use this normalisation to make some aspects of \cref{Section4: DegreeShifting} less cumbersome to write down.
\end{remark}

Now recall that if $H$ is a (topologically finitely generated) abelian locally $\Q_p$-analytic group, we say a representation $V\in \Rep_{\rm{la.c}}(H)$ is \textit{essentially admissible} if its strong dual $V'$ is coadmissible over the Fr\'{e}chet--Stein algebra $C^{\on{an}}(\widehat{H},L)\simeq \mO_{\wh{H}}(\wh{H})$. Equivalently by \cite[Proposition 2.3.2]{EmertonInterpolation} $V'$ is a global sections of a coherent sheaf on $\wh{H}$.
Letting $Y^+\subset T$ be the closed submonoid generated by $T_0,z$ and $Y$ the closed subgroup generated by $T_0, z, z^{-1}$, the above formula defines an action of $Y^+$ on $C^{\la, N_0}$. We can then take the \textbf{finite slope part} (as in \cite[Section 3]{Emerton_Jacquet_I}) using the formula
\[
(C^{i,\la, N_0})_{\rm{fs}} = \mc{L}_{b,Y^+}(\mc{C}^{\on{an}}(\widehat{Y},L), C^{i,\la, N_0}).
\]
We define the \textit{derived Jacquet functors} via the formula $(J_{B}(C^{\la}))^i:= (C^{\bullet, \la, N_0})_{\rm{fs}}$.

One main result of \cite{FuDerived} is that with respect to this natural definition, we have $J_{B}(C^{\la})$ is a complex of essentially admissible representations of $Y$ with $Y$-equivariant differentials (which are thus automatically strict). The spaces we are principally interested in are the cohomology groups $\HH^\ast(J_{B}(C^{\la}))$, which are in fact essentially admissible $T$-representations by \cite[Lemma 3.19]{FuDerived}. \cite[Proposition 4.9]{FuDerived} shows that we may compute $\HH^{i}(J_B(C^{\la}))\simeq (\HH_{\on{haus}}^i(C^{\la, N_0}))_{\on{fs}}$, where for a complex of topological vector spaces $(M^{\bullet}, d^{\bullet})$, $\HH_{\on{haus}}^i(M^{\bullet}):= \ker(d^i)/\ol{\on{im}(d^{i-1})}$, denotes the \textit{Hausdorff} cohomology (that is, we quotient by the \textit{closure} of the image of $d^{i-1}$). This property implies that $\HH^{i}(J_B(C^{\la}))$ is independent of the choice of $i$ (with $p$ fixed) or of the choice of $p$ (with $i$ fixed) by \cite[Corollary 4.10]{FuDerived}. Since the derived category formalism here is a little lacking, it is not a priori obvious that this definition is independent of the choice of representing complex $C^{\bullet}$. We content ourselves with the following lemma:
\begin{lemma}
\label{lemma: IndependenceOfHomotopyForJacquetFunctors}
    Suppose we have a $K_p$-equivariant homotopy equivalence $f:S_1 \to S_2$, $g: S_2\to S_1$ between complexes $S_1,S_2$ of injective admissible $K_p$-representations. Suppose also $S_2$ is a equipped with a $G$-extension. Then we have a topological isomorphism of essentially admissible representations of $T$
    \[
    \HH^i(J_{B}(S_1^{\la})) \simeq \HH^i(J_{B}(S_2^{\la})),
    \]
    where $J_{B}(S_1^{\la})$ is defined by composing the $G$-extension of $S_2$ with $f,g$.
\end{lemma}
\begin{proof}
   Let $(A,i,p)$ denote the $G$-extension on $S_2$. The homotopy equivalence between $S_1$ and $S_2$ induces one between $S_1^{\la, N_0}$ and $S_2^{\la,N_0}$. In particular the natural induced maps on cohomology $\HH^{\bullet}(S_1^{\la, N_0}) \to \HH^{\bullet}(S_2^{\la,N_0})$ are continuous isomorphisms, which moreover intertwines $g\circ p \circ z \circ i \circ f$ and $p \circ z \circ i$, since $g\circ f = f \circ g = \on{id}$ on cohomology. Thus, we may pass to Hausdorff quotients and finite slope parts.
\end{proof}

\subsection{Jacquet functors in terms of homology}

It is sometimes more convenient for us to work with Jacquet functors in terms of homology. While the above use of $G_p$-extensions did not make explicit reference to locally symmetric spaces, the homology approach will. Continue to assume $G/F$ is split at all places $\nu \mid p$. Then let $I := \ol{N}^1\cdot T_0 \cdot N_0$ be the Iwahori subgroup, i.e. the preimage of $B(\mO_{F}\otimes_{\Z}\F_p)$ under the reduction map $G(\mO_{F,p}) \to G(\mO_{F}\otimes_{\Z}\F_p)$. Likewise, for any $n\ge 1$ let $\ol{N}^{n} = \ol{N}^1 \cap I^{n} \subset \ol{N}^1$ be the intersection of $\ol{N}^1$ with the kernel of the reduction mod $p^n$ map $G(\mO_{F,p}) \to G(\mO_F\otimes_{\Z}\Z/p^n)$.

Now suppose that the tame level $K^p$ is sufficiently small, so that $K^pI$ is neat (and thus good). After fixing a finite triangulation of the Borel--Serre compactification $\ol{X}_{K^pI}^G$, we have a complex $C_{\bullet}^{\on{BS}}(K^p)$ and a $\Z[I]$-linear homotopy equivalence to $C(K^p)$, which induces a $D(I)$-linear homotopy equivalence
\[
\begin{tikzcd}
C^{\on{BS}}(K^pI, D(I)) \ar[r, "i"] & \ar[l, shift left, "p"] C(K^pI, D(I))= C(K^p)\otimes_{\Z[I]}D(I)\simeq C(K^p)\otimes_{\Z[G]}D(G)
\end{tikzcd}.
\]
Choosing $z \in T^{++}$, we get an induced action of an operator $\wt{z}:= p \circ z \circ i$ on $C_{\bullet}^{\on{BS}}(K^pI, D(I))$. Letting $C_{\bullet}^{\on{BS}}(K^pI, D(I))_{N_0}$ denote the Hausdorff coinvariants under the right $N_0$-action, we have an isomorphism $C_{\bullet}^{\on{BS}}(K^pI, D(I))_{N_0}\simeq C_{\bullet}^{\on{BS}}(K^pI, D(I)_{N_0})$, since for all $i$ we have $C_{i}^{\on{BS}}(K^pI, D(I))\simeq D(I)^{\oplus m_i}$ for some $m_i\ge 0$. We can then define an Atkin-Lehner-like operator via the formula $c\cdot U_1z:= \sum_{n' \in N_0/zN_0z^{-1}}cn'\wt{z}$. With respect to this action, we may pass to the (dual) \textit{finite slope part} $(C_{\bullet}^{\on{BS}}(K^pI, D(I)_{N_0}))_{\on{fs}}$ and define the dual Jacquet functor
\[
J_B^{\vee}(C_{\bullet}^{\on{BS}}(K^pI, D(I))) := C_{\bullet}^{\on{BS}}(K^pI, D(I)_{N_0})\widehat{\otimes}_{L[z]}L\{\{z^{\pm 1}\}\}.
\]
We can make analogous definitions for complexes involving $C^{\partial,\on{BS}}(K^p)$ or $C^{\on{BM},\on{BS}}(K^p)$.

Again, the proof of \cite[Proposition 3.20]{FuDerived} implies that $J_B^{\vee}(C_{\bullet}^{\on{BS}}(K^pI, D(I)))$ is a complex of coadmissible $\mc{O}_{\widehat{Y}}(\widehat{Y})$-modules. 
In fact, this proof (in particular, the arguments contained in \cite[Lemma 4.2.11, Lemma 4.2.19]{Emerton_Jacquet_I}) only needs the following setup: let $V_{\bullet}$ be a bounded complex of finite free $D(I)$-modules, and choose $z \in T^{++}$ such that $z^{-1}\ol{N}^nz \subset \ol{N}^{n+1}$. Suppose moreover $V_{\bullet}$ is equipped with a $T_0$-linear operator $\wt{z}$ such that for all $n \in N_0$ we have $\wt{z}n = znz^{-1}\wt{z}$ and for all $\ol{n}\in \ol{N}^1$, we have $\ol{n}\wt{z} = \wt{z}\cdot(z^{-1}\ol{n}z)$. Then the same proof as in \cite[Proposition 4.2.28]{Emerton_Jacquet_I} gives that with respect to the same formula for $U_1z$, the complex $(V_{\bullet, N_0})\otimes_{L[z]}L\{\{z^{\pm 1}\}\}$ is a complex of coadmissible $\mc{O}_{\widehat{Y}}(\wh{Y})$-modules.

This first lemma then shows that we can also take the finite slope part after taking homology.

\begin{lemma}
    \label{lemma: HausdorffHomology}
    Let $V_{\bullet}$ be a bounded complex of finite free $D(I)$-modules, and choose $z \in T^{++}$ such that $z\ol{N}^nz^{-1} \subset \ol{N}^{n+1}$. Suppose moreover $V_{\bullet}$ is equipped with a $T_0$-linear operator $\wt{z}$ such that for all $n \in N_0$ we have $\wt{z}n = znz^{-1}\wt{z}$ and for all $\ol{n}\in \ol{N}^1$, we have $\ol{n}\wt{z} = \wt{z}\cdot(z\ol{n}z^{-1})$. Then we have an isomorphism of coadmissible $\mc{O}_{\widehat{Y}}(\wh{Y})$-modules $\HH_{i}((V_{\bullet, N_0})_{\on{fs}})\simeq \HH_{i,\on{haus}}((V_{\bullet, N_0}))_{\on{fs}}$.
\end{lemma}
\begin{proof}
    
     Setting $d_n: V_{n} \to V_{n-1}$ for any $n$, we just need to show that $\ker(d_{n})\widehat{\otimes}_{L[z]} L\{\{z,z^{-1}\}\} = \ker(d_n \widehat{\otimes}_{L[z]} L\{\{z,z^{-1}\}\}$ and $\overline{\on{im}(d_n)} \widehat{\otimes}_{L[z]}L\{\{z,z^{-1}\}\} = \on{im}(d_n\widehat{\otimes}_{L[z]}L\{\{z,z^{-1}\}\}$.
    For the first isomorphism, we wish to show the map 
    \[
    \ker(d_n)\widehat{\otimes}_{L[z]} L\{\{z,z^{-1}\}\} \to V_n\widehat{\otimes}_{L[z]} L\{\{z,z^{-1}\}\}
    \]
    is injective. Note that this map is strict since both modules are coadmissible over $\mc{O}_{\widehat{Y}}$. The map $\ker(d_n) \to V_n$ is injective and $L\{\{z,z^{-1}\}\}$ is flat over $L[z]$, so the induced map before completing remains injective. Then \cite[Proposition 4.8]{FuDerived} gives the desired injectivity.

    The second isomorphism follows a similar strategy: the map $d_n$ factors through $\overline{\on{im}(d_n)}$, and it suffices to show $V_n\widehat{\otimes}_{L[z]} L\{\{z,z^{-1}\}\} \to \overline{\on{im}(d_n)}\widehat{\otimes}_{L[z]} L\{\{z,z^{-1}\}\}$ is surjective. Note this map is automatically strict since both are coadmissible over $\mc{O}_{\widehat{Y}}(\wh{Y})$. Now by the same duality argument (but going from Frechet spaces to compact type spaces instead of the other way around) as in \cite[Proposition 4.9]{FuDerived}, it suffices to show $(\overline{\on{im}(d_n)}')_{\on{fs}} \to ((V_n)')_{\on{fs}}$ is injective. $\overline{\on{im}(d_n)}' \to (V_n)'$ is already injective by Hahn--Banach, and so the desired isomorphism amounts to the left-exactness of the finite slope functor from \cite[Proposition 3.2.6 (iii)]{Emerton_Jacquet_I}.
\end{proof}
We can use this lemma to justify working on the homology side to study the derived Jacquet functors of \cite{FuDerived}.
\begin{corollary}
\label{corollary: DualCBSComplexesComputeJacquet}
    Let $\Pi(K^p)$ be the complex of injective admissible $I$-representations computing $\wt{\HH}^{\ast}(K^p, L)$ as defined in \cref{subsubsection: ComplexesDefininfCompletedCohomology}, equipped with a $G_p$-extension. Then the strong dual $\HH^{\ast}(J_B(\Pi(K^p)^{\la}))'$ is isomorphic to $\HH_{\ast}(J_B^{\vee}(C^{\on{BS}}(K^p,D(I))))$. Analogous isomorphisms hold for boundary and compactly supported cohomology.
\end{corollary}
\begin{proof}
    By \cref{lemma: HausdorffHomology}, the group $\HH_{\ast}(J_B^{\vee}(C^{\on{BS}}(K^p,D(I))))$ is equal to $\HH_{\ast, \on{haus}}(C^{\on{BS}}(K^pI, D(I)_{N_0}))\wh{\otimes}_{L[Y^+]}\mc{O}_{\widehat{Y}}$. Note that the complexes $C^{\on{BS}}(K^pI,D(I))$ and $(\Pi(K^p)^{\la})'$ are quasi-isomorphic perfect complexes of $D(I)$-modules, and so appealing again to \cite[10.4.7]{WeibelHomAlg} implies these complexes are $D(I)$-equivariantly homotopic. Thus, there is a homotopy equivalence between $((\Pi(K^p)^{\la})')_{N_0}\simeq (\Pi(K^p)^{\la, N_0})'$ and $C^{\on{BS}}(K^pI, D(I))_{N_0}\simeq C^{\on{BS}}(K^pI, D(I)_{N_0})$ (we can pass to Hausdorff convariants since the homotopy equivalences in both directions are automatically continuous). Moreover, the action of $\wt{z}$ on each complexes are homotopic, as both are induced from a choice of homotopy inverse to the quasi-isomorphism $C^{\on{BS}}(K^p) \to C(K^p)$. Thus, we in fact have a topological isomorphism $\HH_{\ast}((\Pi(K^p)^{N_0})') \simeq \HH_{\ast}(C^{\on{BS}}(K^pI, D(I)_{N_0}))$ equivariant for the $Y^+$ action on both sides, and thus an isomorphism $\HH_{\ast, \on{haus}}((\Pi(K^p)^{\la, N_0})')_{\on{fs}} \simeq \HH_{\ast, \on{haus}}(C^{\on{BS}}(K^pI, D(I)_{N_0}))_{\on{fs}}\simeq \HH_{\ast}(J_B^{\vee}(C^{\on{BS}}(K^pI, D(I))))$.
    
   The last observation needed is that for any $z \in T^{++}$ satisfying $z^{-1}N^{\on{1}}z \subset N^{\on{2}}$ and any lift of $\wt{z}$ of the $z$-action to the complex $(\Pi(K^p)^{\la,N_0})'$ satisfying the conditions in \cref{lemma: HausdorffHomology}, then taking the finite slope part with respect to $\wt{z}$ yields $\HH_{\ast, \on{haus}}((\Pi(K^p)^{\la, N_0})')_{\on{fs}}\simeq \HH_{\ast}(J_B(\Pi(K^p)^{\la})^{'})$. Again, the statements for boundary and compactly supported (co)homology are analogous.
\end{proof}

For proving \cref{theorem: SmallSupportIntro}, it will be convenient to have a variant of the Atkin--Lehner type action $U_z$, which we will denote $U_2z$ (as opposed to $U_1z$ defined above). We will show $U_1z, U_2z$ are homotopic to each other (and thus yield isomorphic finite slope parts on cohomology by \cref{lemma: HausdorffHomology}).

$U_2z$ comes from an action of $z$ on both $C^{\on{BS}}(K^p)$ and on certain coefficient modules. First, note there is a natural right $T^+$-action on the space $I/N_0\simeq \ol{N}^1T_0$ given by $(\overline{n} t N_0)\ast z :=z^{-1}\overline{n}z t N_0$. This action induces a left $T^+$-action on $\mc{C}^{\la}(I, L)^{N_0}\simeq \mc{C}^{\la}(I/N_0, L)$ given by $(z\cdot f)(g) = f(g\ast z)$ where $g\ast \delta$ is the action defined above. In particular, we may view this action as arising as a composition
    % \[
    % \mc{C}^{\la}(\ol{N}^1T_0, L) \xrightarrow{c_{z^{-1}}} \mc{C}^{\la}(z\ol{N}^1 T_0z^{-1}, L) \to \mc{C}^{\la}(\ol{N}^1T_0, L).
    % \]
    %     Alternatively it is equal to the factorization
         \[
    \mc{C}^{\la}(\ol{N}^1T_0, L) \xrightarrow{\on{res}} \mc{C}^{\la}(z^{-1}\ol{N}^1 T_0z, L) \xrightarrow{\on{conj}_{z^{-1}}}\mc{C}^{\la}(\ol{N}^1T_0, L).
    \]
 Passing to the topological dual $(\mc{C}^{\la}(I, L)^{N_0})'\simeq D(I)_{N_0}$ then gives a right $T^+$-action on $D(I)_{N_0}$ which occurs as a factorization
 $D(I)_{N_0}\simeq D(\ol{N}^1T_0) \xrightarrow{\on{conj}_{z^{-1}}} D(z^{-1}\ol{N}^1zT_0)\hookrightarrow D(\ol{N}^1T_0) \simeq D(I)_{N_0}$. Note that on distribution algebras, this action amounts to the natural left and right multiplication $v \mapsto z^{-1}vz$. Denote this action by $\cdot_{c}z$.
 % or $D(I)_{N_0} \to D(zIz^{-1})_{zN_0z^{-1}} \xrightarrow{c_{z^{-1}}} D(I)_{N_0}$. 
 Fixing $\Z[I]$-linear homotopy equivalences $i,p$ between $C^{\on{BS}}(K^p)$ and $C(K^p)$, we first define an action of $z$ on $C^{\on{BS}}(K^p)$ by the formula $\wt{z}:= p\circ z \circ i$. 
 
 We then define the operator $U_2z$ on $C^{\on{BS}}(K^pI, D(I)_{N_0})$ as follows: Note that the double coset $IzI\simeq I/(zIz^{-1}\cap I)\simeq N/zN_0z^{-1}$, so we can write $IzI = \bigsqcup_{n' \in N_0/zN_0z^{-1}}n' zI$. Then we define
\begin{equation}
\label{equation: U2zAction}
    (\sigma \otimes m)\cdot U_2z \mapsto \sum_{n' \in N_0/zN_0z^{-1}}(\sigma \otimes m)n'z = \sum_{n'}\sigma n'\wt{z} \otimes ((n')^{-1}m)\cdot_{c} z,
\end{equation}

\begin{lemma}
\label{lemma: U_1zhomotopicToU_2z}
    Fixing $z \in T^{++}$, the two Hecke actions $U_1z$ and $U_2z$ on $C^{\on{BS}}(K^PI, D(I)_{N_0})$ are $D(T_0)$-equivariantly homotopic. 
    
    % In particular, for any $\Omega \subset \mc{W}$ affinoid, the induced actions on $C^{\on{BS}}(K^PI, D(I)_{N_0}\widehat{\otimes}_{D(T_0)}\mc{O}(\Omega))$ are also $\mc{O}(\Omega)$-linearly homotopic.
\end{lemma}
\begin{proof}    

    Recall that the action $U_1z$ is induced by a homotopy equivalence from the right action of $z$ on $C(K^p)\otimes_{\Z[G]}D(G)$, which we denote $\widetilde{z}$. The formula for the action can then be written as
    \[
    C^{\on{BS}}(K^pI, D(I)_{N_0}) \xrightarrow{\sum_{n' \in N_0/zN_0z^{-1}}n'} C^{\on{BS}}(K^pI, D(I)_{zN_0z^{-1}}) \xrightarrow{\cdot \wt{z}} C^{\on{BS}}(K^pI, D(I)_{N_0}). 
    \]
    On the other hand, up to homotopy the Hecke action $U_2z$ is given by the composition 
    \[
    C(K^p)\otimes_{\Z[I]}D(I)_{N_0} \xrightarrow{\sum_{n'\in N_0/zN_0z^{-1}}(-)n'\otimes (n')^{-1}(-)} C(K^p)\otimes_{\Z[zIz^{-1}]}D(zIz^{-1})_{zN_0z^{-1}} \xrightarrow{r_z} C(K^p)\otimes_{\Z[I]}D(I)_{N_0},
    \]
    where $r_z(c\otimes v) = cz\otimes v\cdot_{c}z$.
    The key fact then is that natural right action of $g$ on $C(K^p)\otimes_{\Z[G]} D(G)$ 
     is chain homotopic to the action $r_g(c\otimes v) = cg\otimes g^{-1}vg$ by the discussion in \cite[p. 59]{JNWE24}. 
   Note that for $m \in D(I)_{N_0}$ we also have $(n')^{-1}m(n') = (n')^{-1}m$ for $n' \in N_0$, and thus $\sigma n' \otimes (n')^{-1}m = r_{n'}(\sigma \otimes m)$. In particular, the two formulae above are chain homotopic by the cited fact. 
\end{proof}

\subsection{Setting of CM fields}

We now recall the groups and locally symmetric spaces of specific interest. Let $F \supset F^+$ be an imaginary CM field and its totally real subfield. Let $c \in \Aut(F/F^+)$ denote complex conjugation. Let $\Psi_n$ be the $n\times n$ matrix with $1$'s on the anti-diagonal and $0$'s elsewhere. Then we let $\wt{G}/\mO_{F^+}$ be the group scheme with functor of points
\[
\widetilde{G}(R) := \{g \in \GL_{2n}(R\otimes_{\mO_{F^+}}\mO_F): {}^tg\begin{pmatrix}
    0 &\psi_n\\ -\psi_n & 0
\end{pmatrix}g^c = \begin{pmatrix}
    0 &\psi_n\\ -\psi_n & 0
\end{pmatrix}\}
\]
The generic fibre $\tG_{F^+}$ is a quasi-split unitary group $U(n,n)/F^+$ of signature $(n,n)$, which after the quadratic base-change $F/F^+$ becomes isomorphic to $\GL_{2n}/F$. In particular, if $\overline{\nu}\in S_{p}(F^+)$ is a place which splits in $F$, then any choice of $\nu | \overline{\nu}$ in $F$ determines an isomorphism 
\begin{equation}
\label{eq: iotav}
\iota_{\nu}:\widetilde{G}(F_{\overline{\nu}}^+) \simeq \GL_{2n}(F_{\nu}),
\end{equation} 
induced by the isomorphism $F_{\overline{\nu}}^+\otimes_{F^+} F\simeq F_{\nu}\times F_{\nu^{c}}$ and composing the resulting embedding $\widetilde{G}(F_{\overline{\nu}}) \subset \GL_{2n}(F_{\nu}) \times \GL_{2n}(F_{\nu^c})$ with projection onto the first group.

We let $T \subset B \subset \tG$ be subgroups of diagonal and upper triangular matrices in $\tG$, respectively. We also let $G \subset P\subset \tG$ be the Siegel parabolic consisting of block diagonal and block upper-triangular matrices with blocks of size $n \times n$, respectively. Then there is a Levi decomposition $P = U \rtimes G$, where $U$ is the unipotent radical of $P$ of block unipotent matrices. Likewise, we can identify $G = \Res_{\mO_F/\mO_{F^+}}\GL_n$ via the map $\begin{pmatrix}
    A & \\ & D
\end{pmatrix} \mapsto D \in \GL_n(R \otimes_{\mO_{F^+}}\mO_{F})$. We will also (for convenience mainly) use $\ol{P}\subset \tG$ the opposite parabolic of block lower triangular matrices, which has a Levi decomposition $\ol{P} = \ol{U} \rtimes G$ where $\ol{U}$ is the unipotent radical of block lower unipotent matrices.  We similarly let $T_n \subset B_n \subset G = \on{Res}_{\mc{O}_F/\mc{O}_{F^+}}\GL_n$ be the standard diagonal maximal torus and upper triangular Borel subgroup. 
  Note also under the identification $\iota_{\nu}$ in \eqref{eq: iotav}, $G(F_{\overline{\nu}}^+) \simeq \GL_n(F_{\nu}) \times \GL_n(F_{\nu})$ are the split block diagonal matrices, and $T(F_{\ol{\nu}}^+)\simeq T_n(F_{\nu})\times T_n(F_{\nu})$.

More generally, suppose that $p$ is a prime such that each place $\overline{\nu} \in S_{p}(F^+)$ splits in $F$. Then choose $\widetilde{S}_{p} \subset S_{p}(F)$  a subset of places such that $S_{p}(F) = \wt{S}_{p} \sqcup \wt{S}_p^{c}$.
Using \eqref{eq: iotav}, such a choice determines an isomorphism
\[
\widetilde{G}\otimes_{\mO_{F^+}}\mO_{F^+,p} \simeq \prod_{\wt{\nu} \in \wt{S}_p}\GL_{2n, \mO_{F_{\wt{\nu}}}}
\]
with again the parabolic $P\otimes_{\mO_{F^+}}\mO_{F^+,p}$ identifying with the block upper triangular matrices in $\GL_{2n}$, and likewise $T_{\mc{O}_{F^+,p}} \simeq (T_n)_{\mc{O}_{F,p}}$. We will implicitly use this identification throughout the rest of the paper. To simplify notation, for good compact open subgroups $\wt{K} \subset \tG(\A_{F^+}^{\infty}), K \subset \GL_n(\A_F^{\infty})$ we also let $\wt{X}_{\wt{K}}:= X_{\wt{K}}^{\wt{G}}$ and $X_K:= X_K^{G}$ be the symmetric spaces associated to $\tG$ and $G$. We note the dimensions are $\dim_{\R}\wt{X}_{\wt{K}} = 2d = 2n^2[F^+:\Q]$ and $\dim_{\R}X_K = d - 1 = n^2[F^+:\Q] - 1$. 

\subsubsection{Hecke algebras}
\label{subsection: HeckeAlgebras}
We now discuss Hecke algebras. For our work we need to localize at maximal ideals of Hecke algebras, sometimes on the level of complexes. We recall how to make sense of such notions with derived Jacquet functors.

% Similar to how we can either view the derived Jacquet functors via completed (co)homology or via homology of certain coefficient systems, we may also incorporate the action of Hecke algebras in two different ways. 
For any algebraic group $H$ over a number field $E$ and $S \supset S_p(F)$ a finite set of finite places of $E$ let $H^S:= H(\A_E^{S\cup \infty})$ and $K_H^S \subset H^S$ a compact open subgroup.
Then let $\mc{H}(H^S, K_H^S)$ be the abstract Hecke algebra of compactly supported, $K_H^S$-biinvariant, $\Z$-valued functions on $H^S$. Setting $K^S = \prod_{v \notin S \cup \infty}\GL_n(\mc{O}_{F_{\nu}})$, the spherical Hecke algebra for $G = \GL_n/F$ is defined to be $\T^S:= \mc{H}(\GL_n(\A_F^{S\cup \infty}), K^S)\otimes_{\Z}\Z_p$, which is naturally a commutative $\Z_p$-algebra. We can make similar definitions for $\wt{G}$. Suppose $S$ satisfies $S = S^c$, and let $\ol{S}$ be the set of places of $F^+$ lying below $S$. Then define $\wt{K}^S:= \wt{K}^{\ol{S}}=  \prod_{\ol{\nu}\notin \ol{S}\cup \infty}\wt{G}(\mc{O}_{F_{\ol{\nu}}^+})$, and $\wt{\T}^S:= \mc{H}(\wt{G}(\A_{F^+}^{\ol{S}\cup \infty}), \wt{K}^S)\otimes_{\Z} \Z_p$.

Any Hecke algebra (such as $\T^S$ or $\wt{\T}^S$) is spanned by indicator functions on double cosets $[K^SgK^S]$ for $g \in \GL_n(\A_F^{S\cup \infty})$.  Then for any left $K_p$-module $M$ we can define an action of $\T^S$ on $C(K^p)\otimes_{\Z[K_p]}M$ via double cosets $K^SgK^S = \sqcup_jh_jK^S$ and setting
\[
(\sigma\otimes m)[K^SgK^S]:= \sum_j\sigma h_j\otimes m. 
\]
Up to homotopy, this formula also defines the action of any double coset $[K^SgK^S]$ on $C^{\on{BS}}(K^pK_p,M)$.
Note that since we are looking at Hecke operators away from $p$, these operators make no reference to the coefficients $M$, and thus when $M = D(I), D(I)_{N_0}$, the action defined is dual (up to homotopy) to the Hecke action on the complex $A^{\bullet}$ computing completed cohomology.

We use ``big Hecke algebras,'' as defined in \cite[2.1.10]{GeeNewton}. 
For $U_p\subset K_p$ an open normal subgroup, we define
\[
\T_{G}^S(K^pU_p,\mO_L/p^m):= \on{im}(\T^S \to \End_{\mathsf{D}(\mc{O}_L/p^m[K_p/U_p])}(C(K^pU_p, \mO_L/p^m))).
\]
Then we define the \textit{big Hecke algebra} to be 
\[
\T_G^S(K^p):= \varprojlim_{U_p, m}\T_{G}^S(K^pU_p,\mO_L/p^m),
\]
where the transition maps are induced by the functorial maps $\End_{\mathsf{D}(\mc{O}_L/p^{m'}[K_p/U_p'])}(C(K^pU_p, \mO_L/p^m)) \to \End_{\mathsf{D}(\mc{O}_L/p^{m}[K_p/U_p])}(C(K^pU_p', \mO_L/p^{m'}))$ where $U_p' \subset U_p \subset K_p$ and $m' \ge m$. $\T_G^S(K^p)$ comes equipped with an inverse limit topology. A foundational fact \cite[Lemma 2.1.14]{GeeNewton} then is that $\T_G^S(K^p)$ is a semi-local $\Z_p$-algebra, and if $\mf{m}_1, \dots \mf{m}_r$ denotes its finite set of maximal ideals then there is a decomposition 
\begin{equation}
\label{eq: SemiLocalDecomp}
\T_G^S(K^p) = \T_G^S(K^p)_{\mf{m}_1}\times \dots \times \T_G^S(K^p)_{\mf{m}_r},
\end{equation}
where each localisation $\T_G^S(K^p)_{\mf{m}_i}$ is $\mf{m}_i$-adically complete and separated. Moreover, for any  $s$ and $U_p$ pro-$p$, the maximal ideals of $\T_G^S(K^pU_p, \mO_L/p^s)$ are in bijection with those of $\T_G^S(K^p)$, and likewise we have $\T_G^S(K^p)_{\mf{m}} \simeq \varprojlim_{U_p,m}\T_G^S(K^pU_p,\mO_L/p^m)_{\mf{m}}$.

Following \cite[\textsection 6]{FuDerived}, for any maximal ideal $\mf{m} \subset \T_G^S(K^p)$ we can define localised complexes $C_{\bullet}^{\on{BS}}(K^p, \mc{O}_L[[K_p]])_{\mf{m}}:= R\lim_{r}\lim_{n}C^{\on{BS}}(K^p, \mc{O}_L/p^r[K_p/K_p^n])_{\mf{m}}$, which via \eqref{eq: SemiLocalDecomp} admits a direct sum decomposition in the derived category $\mathsf{D}(\mO_L[[K_p]])$
\[
C^{\on{BS}}(K^p, \mc{O}_L[[K_p]])\simeq \bigoplus_{\mf{m} \in \on{MaxSpec}(\T_G^S(K^p))} C^{\on{BS}}(K^p, \mc{O}_L[[K_p]])_{\mf{m}}.
\]
We can then \textit{choose} a perfect complex of $\mO_L[[K_p]]$-modules representing $C^{\on{BS}}(K^p, \mc{O}_L[[K_p]])_{\mf{m}}$, which by abuse of notation will also denote by the same symbol. These complexes also carry a natural homotopy equivalence to the complex $C(K^p, \mO_{L}[[K_p]])_{\mf{m}}$, which possesses an action of the full group $G(F_p)$ and of the Hecke algebra $\T_G^S(K^p)_{\mf{m}}$.
Upon tensoring with $D(K_p)$ and taking $N_0$-coinvariants, we likewise get complexes
$C^{\on{BS}}(K^p, D(K_p)_{N_0})_{\mf{m}}$, whose homology carries a natural action of $T_G^S(K^p)_{\mf{m}}$.
Likewise on the dual side, denote the dual of $C^{\on{BS}}(K^p, L[[K_p]])_{\mf{m}}$ by $\Pi(K^p)_{\mf{m}}$. Then we have homotopies $\begin{tikzcd}\Pi(K^p)_{\mf{m}}\ar[r, "i"]&
    \ar[l, "p", shift left]\Pi(K^p)\end{tikzcd}$ which on homology are the maps induced by the direct summand $\T_G^S(K^p)_{\mf{m}} \to \T_G^S(K^p)$. Upon further composing, we then get a map $\Pi(K^p)_{\mf{m}} \to A^{\bullet} \to A_{\mf{m}}^{\bullet}$ (as in \cref{subsubsection: ComplexesDefininfCompletedCohomology}) which is indeed a homotopy equivalence, prescribing $\Pi(K^p)_{\mf{m}}$ with the structure of $G$-extension.
Analogous constructions may be made for the unitary group $\tG$, and for compactly supported and boundary cohomology.

We now recall the the Satake transfer $\mc{S}: \wt{\T}^S \to \T^S$.
Suppose $\wt{G}$ is any reductive group over $\mO_{F^+}$ and $P = M \ltimes N$ a parabolic subgroup. Given a good subgroup $K_{\tG} \subset \tG(\A_{F^+}^{\infty})$, with $K_{M}, K_{P}, K_{N}$ the corresponding intersections, we say that $K_{\tG}$ is \textit{decomposed} with respect to $P = M \ltimes N$ if $K_{P} = K_{M} \ltimes K_{N}$. If $K_{\tG}$ is decomposed with respect to $P$ and $S$ is a finite set of finite places such that $K_{\tG,\nu}$ is hyperspecial for all $\nu \notin S$, we can define two maps on Hecke algebras, ``restriction to $P$'' and ``integration along $N$''
\[
r_{P}: \mc{H}({\tG}^S, K_{\tG}^S) \to \mc{H}(P^S, K_P^S),\quad\quad r_M: \mc{H}(P^S,K_P^S) \to \mc{H}(M^S, K_M^S),
\]
and define the ``unnormalised Satake transform'' to be $\mc{S}: r_M\circ r_P: \mc{H}({\tG}^S, K_{\tG}^S) \to \mc{H}(M^S, K_M^S)$ (see \cite[2.2.3, 2.2.4]{NT16}).

We can also define the map $\mc{S}$ in the case of big Hecke algebras. We specialize to the case of $\tG/\mO_{F^+}$ the quasi-split unitary group, $P$ the Siegel parabolic with Levi decomposition coming from $M = G = \Res_{F/F^+}\GL_n$.  Suppose $\wt{U}_p\subset \tG(\mO_{F^+,p})$ is such that $\wt{K}^p\wt{U}_p$ is good and decomposed with respect to $P = GU$ and $\mf{m} \subset \T^S(K^pU_p, \mO_L/p^m)$ a non-Eisenstein maximal ideal (see \cref{subsection: definingEigenvarietiesAndGaloisRepresentations}). Then for all $m\ge 0$ the map $\mc{S}$ descends to a morphism
\[
\mc{S}: \T_{\partial}^S(\widetilde{K}^p\widetilde{U}_p, m)_{\wt{\mf{m}}} \to \T_{G}^S(K^pU_p, m)_{\mf{m}},
\]
by \cite[Theorem 2.4.4]{10AuthorPaper}, where \[
\T_{\partial \wt{G}}^S(\widetilde{K}^p\widetilde{U}_p, \mO_L/p^m)_{\wt{\mf{m}}}:= \on{im}(\wt{\T}^S \to \End_{\mathsf{D}(\mO_L/p^m[\wt{K}_p/\wt{U}_p])}(C^{\partial}(\wt{K}^p\wt{U}_p, \mO_L/p^m)))_{\wt{\mf{m}}}.
\]

Moreover, these morphisms are compatible with the transition maps induced by changing the level $\wt{U}_p$, since $\mc{S}$ exists on the level of abstract Hecke algebras. Passing to the inverse limit, we get our induced map of big Hecke algebras
\[
\mc{S}: \T_{\partial}^S(\wt{K}^p)_{\wt{\mf{m}}} \to \T_{G}^S(K^p)_{\mf{m}}.
\]
In the sequel, we will typically suppress the $G, \tG$ from the notation and instead just consider the big Hecke algebras $\wt{\T}^S(\wt{K}^p)_{\wt{\mf{m}}}, \T^S(K^p)_{\mf{m}}$.

\subsubsection{Weights for $G$ and $\wt{G}$}

To get a sense of what our eigenvarieties will be interpolating, we recall what dominant algebraic weights look like for both $G$ and $\tG$. First fix $E$ a finite extension which contains all the embeddings $F \hookrightarrow \overline{\Q}_p$.

We start with $G$. We have that $X^{\ast}((\Res_{F/\Q}T_n)_{E})\simeq (\Z^n)^{\Hom(F,E)}$ via the identification $(\lambda_{\tau, i})_{\tau, i} \mapsto (\on{diag}(t_1,\dots, t_n) \mapsto t_1^{\lambda_1}\dots t_n^{\lambda_n})$. A character $(\lambda_{\tau, i})_{\tau, i}$ is then $(\Res_{F/\Q}B)_{E}$-dominant if
\[
\lambda_{\tau, 1} \ge \dots \ge \lambda_{\tau, n}
\]
for all $\tau \in \Hom(F,E)$. We let $\Z^n_{\on{dom}}$ denote the set of dominant tuples of $n$ integers, so that $\lambda$ is $(\Res_{F/\Q}B)_{E}$-dominant iff $\lambda \in (\Z^n_{\on{dom}})^{\Hom(F,E)}$. For any dominant weight $\lambda$ we may define an irreducible algebraic representation $\mathsf{L}(\lambda)$ of $G_E$ of highest weight $\lambda$. If $\lambda \in (\Z^n)^{\Hom(F,E)}$, we can define its dual weight $\lambda^{\vee}$ via $\lambda_{\tau, i}^{\vee}:= -\lambda_{\tau,n+1 - i}$. Clearly if $\lambda$ is dominant then so is $\lambda^{\vee}$, and in fact we have an isomorphism $\mathsf{L}(\lambda)^{\vee}\simeq L(\lambda^{\vee})$. We will also let $\mathsf{L}(\lambda)$ denote the associated local system on $X_{K^pK_p}$ and $\ol{X}_{K^pK_p}$.

Now for $\tG$, suppose for simplicity (and it is the only case we will need) that every place $\ol{\nu} \in S_p(F^+)$ splits in $F$. For each $\ol{\nu}$, fix a $\nu \in S_p(F)$ such that $\ol{\nu}|\nu$. These choices fix a lift $\tau: F \hookrightarrow E$ for each $\ol{\tau}: F^+ \hookrightarrow E$, and thus yields an isomorphism 
\[
(\Res_{F^+/\Q}\tG)_{E}\simeq \prod_{\tau \in \Hom(F^+, E)}\GL_{2n, E},
\]
and thus gives an identification of the character group $X^{\ast}((\Res_{F^+/\Q}T)_{E})\simeq (\Z^{2n})^{\Hom(F^+,E)}$. The $(\Res_{F^+/\Q}B)_{E}$-dominant weights are precisely those in $(\Z^{2n}_{\on{dom}})^{\Hom(F^+,E)}$.

We also define a composite isomorphism $(\Z^n)^{\Hom(F,E)} \simeq X^{\ast}((\Res_{F/\Q}T_n)_E) \simeq X^{\ast}((\Res_{F^+/\Q}T)_{E})\simeq (\Z^{2n})^{\Hom(F^+,E)}$
mapping $\lambda \in (\Z^n)^{\Hom(F,E)}$ to $\wt{\lambda} \in (\Z^{2n})^{\Hom(F^+,E)}$ via 
\[
\wt{\lambda}_{\tau} = (-w_0^{\GL_n}\lambda_{\tau\circ c},\lambda_{\tau}) = (-\lambda_{\tau c, n},\dots, -\lambda_{\tau c, 1}, \lambda_{\tau, 1},\dots, \lambda_{\tau, n}).
\]
We can similarly associate to a dominant $\wt{\lambda}$ an irreducible algebraic representation $\mathsf{L}(\wt{\lambda})$, which in turn induces a local system on $\wt{X}_{\wt{K}^p\wt{K}_p}$ or $\ol{\wt{X}}_{\wt{K}^p\wt{K}_p}$.

The purpose for including these weights and local systems is because the cohomology groups $\HH^{\ast}(X_{K^pK_p}, \mathsf{L}(\lambda)),\HH^{\ast}(\wt{X}_{\wt{K}^p\wt{K}_p}, \mathsf{L}(\wt{\lambda}))$ can be understood as spaces of automorphic forms. Fix $\iota: \ol{\Q}_p \simeq \C$.  There are very general theorems due to Franke \cite[Theorem 18]{Fr98} relating the cohomology of symmetric spaces to automorphic forms. For $G$, we merely recall \cite[Theorem 2.4.10]{10AuthorPaper}, which tells us for any $\pi$ cuspidal regular algebraic automorphic representation of $\GL_n(\A_F)$ of weight $\iota\lambda$, we have for some good open subgroup $K \subset \GL_n(\A_F^{\infty})$ the ring map $f: \T^S \to \ol{\Q}_p$ associated to the Hecke eigenvalues of $(\iota^{-1}\pi^{\infty})^K$ factors through $\T^S \to \T^S(\HH^{\ast}(X_{K^pK_p}, \mathsf{L}(\lambda)))$, the Hecke algebra acting faithfully on $\HH^{\ast}(X_{K^pK_p},\mathsf{L}(\lambda))$. Conversely, if $\mf{m}\subset \T^S(\HH^{\ast}(X_{K^pK_p}, \mathsf{L}(\lambda)))$ is non-Eisenstein (see discussion before \cref{lemma: GaloisRepsAssociatedToEigenvarieties}) then any ring homomorphism $f: \T^S(\HH^{\ast}(X_{K^pK_p}, \mathsf{L}(\lambda)))_{\mf{m}} \to \ol{\Q}_p$ comes from a system of Hecke eigenvalues of a regular algebraic cuspidal $\pi$ of weight $\iota \lambda$. Moreover, a  ``purity lemma'' of Clozel \cite[Lemme 4.9]{Clo90} tells us that $\HH^{\ast}(X_{K^pK_p},\mathsf{L}(\lambda))_{\mf{m}} \neq 0$ implies that there exists a $w \in \Z$ such that for all $\tau \in \Hom(F, \ol{\Q}_p)$ and $1 \le i \le n$ that $\lambda_{\tau,i} + \lambda_{\tau c, n + 1 - i} = w$.

As eigenvarieties have weights that vary $p$-adically, we are interested in what possible weights they live over. Namely, let $\mc{W}:= \widehat{T(\mc{O}_{F^+,p})}\simeq \widehat{T_n(\mc{O}_{F,p})}$ be the rigid space of continuous characters of integral points of the maximal torus. These spaces track the possible weights that $p$-adic automorphic forms can take. Any weight $\lambda \in (\Z^n)^{\Hom(F,\ol{\Q}_p)} \simeq (\Z^{2n})^{\Hom(F^+,\ol{\Q}_p)}$ defines an associated point $\lambda \in \mc{W}$. We call such points \textit{algebraic}. Note there is an isomorphism $\Hom(F,\ol{\Q}_p)\simeq \Hom(F\otimes_{\Q}\Q_p, \ol{\Q}_p) = \sqcup_{\nu\in S_p(F)}\Hom(F_{\nu},\ol{\Q}_p)$. Under the isomorphism $\mO_{F,p}^{\times}\simeq \prod_{\nu\in S_p(F)}\mO_{F_{\nu}}^{\times}$, the character $\lambda \in \mc{W}$ is given by the formula $(\on{diag}(x_{\nu,1},\dots, x_{\nu, n}))_{\nu} \mapsto \prod_{i = 1}^n\prod_{\tau \in \Hom(F_{\nu}, \ol{\Q}_p)}\tau(x_{\nu})^{\lambda_{\tau,i}}$

Motivated by Clozel's ``purity lemma,'' we can then make the following definition.

\begin{definition}
\label{definition: CTG}
    Let $\lambda \in \mc{W}$ be a weight. Then we say $\lambda$ is \textit{CTG} (``cohomologically trivial for $G$'') if for all $w \in W^{\oP}\subset W_{\tG}$ (see \cref{Section3: ParabolicInduction} for definition) and for all integers $i_0 \in \Z$, there is an embedding $\tau \in \Hom(F, \ol{\Q}_p)$ such that the weight $w\cdot \lambda$ satisfies $(w\cdot\lambda)_{\tau} - (w \cdot \lambda)_{\tau\circ c}^{\vee} \neq (i_0, \dots i_0)$. 
    We denote the set of CTG weights by $\mc{W}_{CTG}$.
\end{definition}
A key reason to care about CTG weights is the following. Continuing to assume that $\mf{m}\subset \T^S(K^p)$ is non-Eisenstein. Now let $\wt{\mf{m}}:= \mc{S}^{\ast}(\mf{m})$. Then if $\wt{\lambda} \in (\Z^{2n})^{\Hom(F^+, \ol{\Q}_p)}$ is a dominant weight for $\wt{G}$ which is CTG, then \cite[proof of Theorem 2.4.11]{10AuthorPaper} and the purity lemma implies that $\HH^{\ast}(\partial \wt{X}_{\wt{K}^p\wt{K}_p},\mathsf{L}(\wt{\lambda}))_{\wt{\mf{m}}} = 0$ in all possible degrees.
Here is a straightforward consequence of the definition of CTG that we will need in the sequel:
\begin{lemma}
\label{Lemma: MostWeightsAreCTG}
    Let $U \subset \mc{W}$ be a connected affinoid open. Then $U \cap (\mc{W}- \mc{W}_{CTG})$ is contained in a proper analytic closed $Z \subset U$.
\end{lemma}
\begin{proof}
    By definition, $\delta \in \mc{W}- \mc{W}_{CTG}$ iff for some $w \in W^{\oP}$ and $i_0 \in \Z$ we have $(w\cdot \delta)_{\tau} - (w\cdot \lambda)_{\tau c}^{\vee} = (i_0,\dots, i_0)$. In other words, for all $\tau \in \Hom(F, \ol{\Q}_p)$ and all $1\le j \le n$ we have $(w\cdot \delta)_{\tau,j} - (w\cdot \delta)_{\tau\circ c, n+ 1 - j} = i_0$. $\delta$ thus lies one of the $\#W^{\oP}$ subvarieties $Z_w$ defined by the equations $w\cdot \delta)_{\tau,j} - (w\cdot \delta)_{\tau\circ c, n+ 1 - j} = (w\cdot \delta)_{\tau,j'} - (w\cdot \delta)_{\tau\circ c, n+ 1 - j'}$ for all $1\le j,j' \le n$. Taking $Z := U\cap \cup_{w \in W^{\oP}}Z_w$ gives the desired result.
\end{proof}
\subsubsection{Relevant Hecke operators}
In order to discuss Galois representations associated to automorphic data for $G$ and $\tG$ we need to discuss Hecke polynomials, and thus look at certain explicit Hecke operators. For any $1 \le i \le n$ we consider the double coset operators
\[
T_{\nu, i}:= [\GL_n(\mc{O}_{F_{\nu}})\on{diag}(\varpi_{\nu}, \dots , \varpi_{\nu}, 1, \dots, 1)\GL_n(\mc{O}_{F_{\nu}})] \in \mc{H}(\GL_n(F_{\nu}),\GL_n(\mc{O}_{F_{\nu}})),
\]
where there are $i$ $\varpi_{\nu}$'s along the diagonal. The associated Hecke polynomial is defined as
\[
P_{\nu}(X) = X^n - T_{\nu, 1}X^{n-1} + \dots + (-1)^iq_{\nu}^{i(i-1)/2}T_{\nu, i}X^{n-i} + \dots + q_{\nu}^{n(n-1)/2}T_{\nu, n} \in \mc{H}(\GL_n(F_{\nu}),\GL_n(\mc{O}_{F_{\nu}}))[X].
\]
Likewise if $\ol{\nu}$ is a place of $F^+$ unramified in $F$ and $\nu | \ol{\nu}$ and $1 \le i \le 2n$, there is an associated operator $\wt{T}_{\nu, i} \in \mc{H}(\tG(F_{\ol{\nu}}^+), \tG(\mc{O}_{F_{\ol{\nu}}^+}))\otimes_{\Z}\Z[q_{\ol{\nu}}^{-1}]$ defined in \cite[Proposition-Definition 5.2]{NT16} (denoted $T_{G,\nu, i}$ in that source). We can then define another Hecke polynomial:
\[
\wt{P}_{\nu}(X):= X^{2n} - \wt{T}_{\nu, 1}X^{2n-1} + \dots + (-1)^jq_{\nu}^{j(j-1)/2}\wt{T}_{\nu, j}X^{2n-j} + \dots + q_{\nu}^{n(2n-1)}\wt{T}_{\nu, 2n} \in (\mc{H}(\tG(F_{\ol{\nu}}^{+}), \tG(\mO_{F_{\ol{\nu}}^{+}}))\otimes_{\Z}\Z[q_{\ol{\nu}}^{-1}])[X].
\]
The behavior of the Satake transfer map $\mc{S}: \mc{H}(\tG(F_{\ol{\nu}}^{+}), \tG(\mO_{F_{\ol{\nu}}^{+}})) \to \mc{H}(\GL_n(F_{\nu}),\GL_n(\mc{O}_{F_{\nu}}))$ at unramified places can be described using these polynomials. If $f(X)$ is a polynomial of degree $d$ with invertible constant term $a_0$, then denote $f^{\vee}(X):= a_0^{-1}X^df(X^{-1})$.
\begin{proposition}
\label{proposition: EffectofHeckePolyUnderS}
    Let $\nu$ be a place of $F$ which is unramified over a place $\ol{\nu}$ of $F^+$. Then we have
    $\mc{S}(\wt{P}_{\nu}(X)) = P_{\nu}(X)q_{\nu}^{n(2n-1)}P_{\nu^c}^{\vee}(q_{\nu}^{1-2n}X)$.
\end{proposition}
\begin{proof}
    This follows from \cite[Proposition-Definition 5.3]{NT16}
\end{proof}

With the given setup, we are especially interested in the complexes $\Pi(K^p)_{\mf{m}}$ and 
$\Pi_{\ast}(\wt{K}^p)_{\wt{\mf{m}}}$ where $? = \emptyset, c, \partial$, and $\wt{\mf{m}} = \mc{S}^{\ast}(\mf{m})$. After passing to locally analytic vectors, we now want to study the cohomology $\HH^{\ast}(J_B(\Pi_{?}(\wt{K}^p)_{\wt{\mf{m}}}^{\la}))$, where we can, for example, take the Jacquet functor with respect to $z = \begin{pmatrix}
    p^{2n-1} & &&\\
& p^{2n-2} &&\\
&& \ddots &\\
&&&1
\end{pmatrix} \in T(F_p^+)^+$. We will pass to the dual side and study $\HH_{\ast}(J_B^{\vee}(C^{?, \on{BS}}(K^pI, D(I))))_{\wt{\mf{m}}}$, via the $U_2z$-action.

\subsection{Fredholm theory and slopes}

There are several approaches to truncating slopes for eigenvarieties, with different advantages. In this subsection, we express the finite slope functors of Emerton (and Fu) on completed (co)homology in terms of Ash--Stevens's theory of slope decompositions.

We first recall some details about slope decompositions. Let $L/\Q_p$ be a finite extension. Let $A$ be an affinoid $L$-algebra, and $Q(X) \in A[X]$ a polynomial of degree $d.$  We say $Q(X)$ is \textit{multiplicative} if its leading coefficient is a multiplicative unit in $A$. For any $h \in \Q_{\ge 0}$, we say a multiplicative $Q(X)$ \textit{has slope $\le h$} if for all $x \in \on{Sp}(A)$, the specialization $Q_x(X)$ has all edges of its Newton polygon having slope $\le h$. Recall if $Q$ is multiplicative with constant term $a_0$, we defined $Q^{\vee}(X):= a_0^{-1}X^dQ(1/X)$. Now if $M$ is an $A$-module equipped with a $A$-linear operator $T$, we say an element $m \in M$ has slope $\le h$ (with respect to $T$) if there exists a polynomial $Q(X) \in A[X]$ of slope $\le h$ such that $Q^{\vee}(T)m = 0$. Most crucially, we define a \textit{slope $\le h$ decomposition} of $M$ to be an $A[T]$-linear decomposition 
\[
M = M_{\le h} \oplus M_{ > h}
\]
such that $M_{\le h}$ is a finitely generated $A$-module and for every multiplicative polynomial $P(X) \in A[X]$ of slope $\le h$, we have $P^{\vee}(T): M_{> h}\xrightarrow{\sim} M_{> h}$ is an $A$-linear isomorphism. Such a decomposition is in fact unique. See \cite[Section 2.3]{Han17} for more properties of slope decompositions, including its functorial behavior. For the next lemma, recall that an $A$-Banach module $M$ is called \textit{orthonormalizable} (or \textit{ONable}) if there is an isomorphism $A\simeq V\widehat{\otimes}_{L}A$ for $V$ a $L$-Banach space. Moreover, an $A$-Banach module $M$ is called \textit{(Pr)} if there is another $A$-Banach module $Q$ such that $M\oplus Q$ is ONable.

\begin{lemma}
\label{lemma: PrComplexesHaveSlopeDecompositions}
    Let $C_{\bullet}$ be a bounded complex of (Pr)-Banach modules over an affinoid $\Q_p$-algebra $A$ with an $A[u]$-module structure, such that the map $u: \bigoplus_i C^i \to \bigoplus_i C^i$ is compact $A$-linear endomorphism. Then for any $x \in \on{Sp}(A)$ and any $h \in \Q_{\ge 0}$ there is an affinoid subdomain $\on{Sp}(A') \subset \on{Sp}(A)$ containing $x$ such that $C^{\bullet}\widehat{\otimes}_AA'$ admits a slope $\le h$ decomposition, and $(C^{\bullet} \widehat{\otimes}_AA')_{\le h}$ is a complex of finite flat $A'$-modules.
\end{lemma}
\begin{proof}
The case where all the $C^{\bullet}$ are ONable is  \cite[Proposition 2.3.3]{Han17}, which itself is a direct consequence of \cite[Theorem 4.5.1, Proposition 4.1.2]{AS08}.

The general case follows directly: since $\bigoplus_i C^i$ is (Pr), there is an ONable $F$ and (Pr) $S$ such that $(\bigoplus_i C^i) \oplus S = F$, and we can extend $u$ to a compact operator on $F$ via the formula $(u,0)$. In particular, by \cite[Theorem 4.5.1]{AS08} there is an affinoid subdomain $\on{Sp}(A')\subset \on{Sp}(A)$ such that the Fredholm power series for $u$ acting on $F\widehat{\otimes}_{A}A'$ admits a slope $\le h$ factorisation $F(T) = Q(T)\cdot R(T)$ in the sense of \textit{loc. cit.}, so that we have a slope $\le h$ decomposition for $F$. Since $u$ acts on $S$ via zero, the map $(F\widehat{\otimes}_{A}A')_{\le h} \to F\widehat{\otimes}_AA' \to \left(\bigoplus_i C^i\right)\widehat{\otimes}_{A}A'$ is in fact injective, and thus the image of this map gives a definition for slope $\le h$ decomposition $((\bigoplus_i C_i)\widehat{\otimes}_AA')_{\le h}$ for $(\bigoplus_i C_i)\widehat{\otimes}_AA'$. As the differentials in $C^{\bullet}$ are $u$-equivariant, this gives a slope $\le h$ decomposition for the complex $C^{\bullet}\widehat{\otimes}_AA'$.
\end{proof}

We would like to apply this lemma in the case of complexes computing (dual) Jacquet functors. Since dual Jacquet functors are instead \textit{limits} of Banach spaces, we need to reduce to the case of complexes of Banach modules.

To accomplish this reduction, as in \cite[5.2]{BHSAnnalen} we use the refinement of locally analytic vectors defined in \cite[0.3]{ColmezDospinescu}. Let $H$ be a uniform pro-$p$ group. Then for any locally analytic representation $\Pi$ of $H$ and $h \ge 1$ the authors in \textit{loc. cit.} define a subspace $\Pi_{H}^{(h)}$, denoting the space of vectors with, roughly speaking, $(h)$-analytic Mahler expansions with respect to $H$. This space is naturally a Banach representation of $H$. Moreover if $\mc{C}^{(h)}(H,L):= (\mc{C}^{\la}(H,L))_H^h$, then there are natural compact closed embeddings $\mc{C}^{(h)}(H,L) \hookrightarrow \mc{C}^{(h+1)}(H,L)$, and $\mc{C}^{\la}(H,L) =  \varinjlim \mc{C}^{(h)}(H,L)$. Let $D^{(h)}(H,L):= \mc{C}^{(h)}(H, L)'$.  Some facts about these constructions we need are the following \cite[Theoreme 0.6]{ColmezDospinescu}:
    \begin{enumerate}
\label{factsaboutColmezDospinescu}
        \item $\mc{C}^{(h)}(T \cap H, L)' \simeq D_{<p^{r_h}}(T \cap H, L)$, where $r_h = \frac{1}{p^{h-1}(p-1)}$.
        \item $\Pi_H^{(h+1)} = \Pi_{H^p}^{(h)}$ for any Banach representation $\Pi$ of $H$.
    \end{enumerate}
    Now choose $H \subset I$ an open uniform pro-$p$ subgroup, such that there is an Iwahori decomposition
    \[
    H\xrightarrow{\sim}(\ol{N}\cap H) \times (T \cap H) \times (N \cap H)=: \ol{N}_H\times T_H \times N_H.
    \]
    The Iwahori decomposition induces a topological isomorphism
    \[
    \mc{C}^{(h)}(H,L)\simeq \mc{C}^{(h)}(\ol{N}_H, L)\wh{\otimes}_L \mc{C}^{(h)}(T_H, L)\wh{\otimes}_L\mc{C}^{(h)}(N_H, L).
    \]
    
\begin{proposition}
\label{proposition: U_pfactorisationDiagram}
    For any weight $\lambda \in \mc{W}(\ol{\Q}_p)$, and $h \in \Q_{\ge 0}$, there is an affinoid open neighborhood $\Omega\subset \mc{W}$ of $\lambda$ such that $C^{\on{BS}}(K^PI, D(I)_{N_0}\widehat{\otimes}_{D(T_0)}\mc{O}(\Omega))$ admits a slope $\le h$ decomposition. More specifically, there is an increasing affinoid open covering $(U_m)_m\subset \mc{W}$ and complexes of (Pr)-Banach $A_m:= \mc{O}_{\mc{W}}(U_m)$-modules $V_{m, \bullet}$ such that we have a $T_0$-equivariant isomorphism of Fr\'{e}chet spaces
    \[
    C_{\bullet}^{\on{BS}}(K^PI, D(I)_{N_0})\simeq \varprojlim_mV_{m,\bullet}
    \]
    so that the action of $U_2z$ admits a factorisation
    \begin{equation}
    \label{eq: FactorizingUzOperator}
    \begin{tikzcd}
    V_{m+1,\bullet}\ar[r]\ar[d, "z_{m+1}"']&V_{m+1,\bullet}\otimes_{A_{m+1}}A_m\ar[d, "{z_{m+1}\otimes_{A_{m+1}}1_{A_m}}"']\ar[r,"\beta_m"]&V_{m,\bullet} \ar[dl, "\alpha_m"]\ar[d, "z_m"]\\
V_{m+1,\bullet}\ar[r]&V_{m+1,\bullet}\otimes_{A_{m+1}}A_m\ar[r, "\beta_m"]&V_{m,\bullet}
\end{tikzcd}
    \end{equation}
    and we have isomorphisms $C_{\bullet}^{\on{BS}}(K^pI, D(I)_{N_0} \widehat{\otimes}_{D(T_0)}\mc{O}(\Omega))_{\le h}\simeq (V_{s,\bullet}\widehat{\otimes}_{A_s}\mc{O}(\Omega))_{\le h}$ for all $s \gg 0$. Moreover, for any $\Omega' \subset \Omega$ affinoid open we also have 
    \begin{equation}
    C_{\bullet}^{\on{BS}}(K^pI, D(I)_{N_0}\widehat{\otimes}_{D(T_0)}\mc{O}(\Omega))_{\le h}\widehat{\otimes}_{\mc{O}(\Omega)}\mc{O}(\Omega') \simeq C_{\bullet}^{\on{BS}}(K^pI, D(I)_{N_0}\widehat{\otimes}_{D(T_0)}\mc{O}(\Omega'))_{\le h}. 
    \end{equation}
    Moreover, all of these isomorphisms are (up to homotopy) $\wt{\T}^S$-equivariant, and remain true for Borel-Moore and boundary homology.
\end{proposition}
\begin{proof}

    We heavily rely on \cite[Proposition 5.3]{BHSAnnalen} as our guide. First, choose $H\subset I$ an open uniform pro-$p$ group as above. The proof in \textit{loc. cit.} shows we can write $D(H)_{N_H} \simeq \varprojlim W_m$ where as a $(D(H), A_m)$-bimodule we have an isomorphism $W_m:= (D^{(m)}(H)_{N_H})\widehat{\otimes}_{D^{(m)}(T_H)}B_m$, where $B_m := D_{p^{-r_m}}(T_H)$. Moreover, the $W_m$ are naturally ONable Banach modules over $B_m$. Note that as right $B_m$-modules we have an isomorphism $W_m\simeq D^{(m)}(\ol{N}_H)\widehat{\otimes}_LB_m$. Via this isomorphism, we can describe the induced $z$-action as factorizing $z:D^{(m)}(H)_{N_H}\to D^{(m)}(z^{-1}\ol{N}_Hz)\widehat{\otimes}_{L}D^{(m)}(T_H) \to D^{(m+1)}(\ol{N}_H)\widehat{\otimes}_LD^{(m)}(T_H)$. With this factorisation in mind, we may get a diagram of bimodules as in \eqref{eq: FactorizingUzOperator} by defining $\alpha_m:= z \otimes 1_{B_m}$ and $\beta_m: D^{(m+1)}(\ol{N}_H)\widehat{\otimes}_{L}B_m \to D^{(m)}(\ol{N}_H)\widehat{\otimes}_{L}B_m$ the map induced by $D^{(m+1)}(\ol{N}^1) \to D^{(m)}(\ol{N}^1)$. Defining $V_m:= C_{\bullet}^{\on{BS}}(K^p)\otimes_{\Z[H]}W_m$ with its induced Hecke action, we get our desired diagram of complexes. In light of the isomorphism $C_{\bullet}^{\on{BS}}(K^p)\otimes_{\Z[I]}D(I)\simeq C_{\bullet}^{\on{BS}}(K^p)\otimes_{\Z[I]}(\Z[I]\otimes_{\Z[H]}D(H)) \simeq C_{\bullet}^{\on{BS}}(K^p)\otimes_{\Z[H]}D(H)$, these complexes $V_m$ have a natural action of $T_0$ on them, since $T_0$ is abelian, and are thus modules over the algebra $A_m:= B_{m}\otimes_{L[T_H]}L[T_0]$. In fact, the $V_m$ are naturally (Pr) Banach modules over $A_m$, since $A_m$ is a finite \'{e}tale $B_m$-algebra. Then the Amice transform (applied to $T_H$) also implies we have $A_m \simeq \mc{O}_{\mc{W}}(U_m)$ for some affinoid open $U_m\subset \mc{W}$. 
    
    \cref{lemma: PrComplexesHaveSlopeDecompositions} then implies that for any $\lambda \in U_m$, there is an affinoid open neighborhood $\Omega \subset U_m$ containing $\lambda$ such that $V_m\widehat{\otimes}_{A_m}\mO_{\mc{W}}(\Omega)$ admits a slope $\le h$ decomposition. Likewise, using \eqref{eq: FactorizingUzOperator} and same proof as \cite[Proposition 3.1.1]{Han17} shows that the characteristic power series $F_{\Omega, s}(X):= \det(1 - XU_z|V_{s,\bullet}\widehat{\otimes}_{A_{s}}\mO(\Omega))$ satisfies $F_{\Omega, s}(X) = F_{\Omega, s+1}(X)$ for all $s \gg 0$. This equality of characteristic power series implies that we can choose a single affinoid neighborhood $\Omega$ such that for all $s > 0$ satisfying $U_s\supset \Omega$ we have the complexes $V_{s,\bullet} \wh{\otimes}_{A_s}\mO(\Omega)$ admit slope $\le h$ decompositions, and in fact $((C_{\bullet}^{\on{BS}}(K^p)\otimes_{\Z[H]}W_s) \widehat{\otimes}_{A_s}\mc{O}(\Omega))_{\le h} \simeq (C_{\bullet}^{\on{BS}}(K^p)\otimes_{\Z[I]}(\Z[I]\otimes_{\Z[H]}W_s) \widehat{\otimes}_{A_s}\mc{O}(\Omega))_{\le h}$. For the Fr\'{e}chet limit $D(I)_{N_0}\widehat{\otimes}_{D(T_0)}\mO(\Omega)$, if we let $F = P\cdot Q$ be a slope $\le h$-factorisation of the Fredholm series in the sense of \cite[\textsection 2.3]{Han17} (in particular, $Q$ is a multiplicative polynomial of slope $\le h$),  we can define $(C_{\bullet}^{\on{BS}}(K^pI, D(I)\widehat{\otimes}_{D(T_0)}\mO(\Omega)))_{\le h} := \ker Q^{\ast}(\wt{U}_z)$, and using \cref{lemma: AffinoidInverseLimit} below deduce we have slope $\le h$ decompositions for $C_{\bullet}(K^pI, D(I)_{N_0}\widehat{\otimes}\mc{O}(\Omega))$ and in fact isomorphisms 
    \[
    C_{\bullet}^{\on{BS}}(K^pI, D(I)_{N_0}\widehat{\otimes}\mc{O}(\Omega))_{\le h} \simeq \varprojlim_m (C_{\bullet}^{\on{BS}}(K^pI, W_m\widehat{\otimes}_{A_m}\mO(\Omega)))_{\le h} \simeq C_{\bullet}^{\on{BS}}(K^pI, V_s\widehat{\otimes}_{A_s}\mc{O}(\Omega))_{\le h} 
    \] for $ s\gg 0$. The final statement follows just as in \cite[Proposition 3.1.4]{Han17}.
\end{proof}

\begin{lemma}
\label{lemma: AffinoidInverseLimit} Let $M_m := (\Z[I]\otimes_{\Z[H]}D^{(m)}(H)_{N_H})\wh{\otimes}_{D^{(m)}(T_0)}A_m$. Then for any $\Omega \subset \mc{W}$
we have a topological isomorphism of $(D(I), \mc{O}(\Omega))$-bimodules  $D(I)_{N_H}\widehat{\otimes}_{D(T_0)}\mc{O}(\Omega)\simeq \varprojlim_m M_m\widehat{\otimes}_{A_m}\mc{O}(\Omega)$. Moreover, the natural $z$-actions on both sides are compatible.
\end{lemma}
\begin{proof}
      First note there are evident maps of $(D(I), \mc{O}(\Omega))$-bimodules, $D(I)_{N_H}\widehat{\otimes}_{D(T_0)}\mc{O}(\Omega) \to M_m \widehat{\otimes}_{A_m}\mc{O}(\Omega)$, which induces a map $D(I)_{N_0}\widehat{\otimes}_{D(T_0)}\mc{O}(\Omega)\to\varprojlim M_m\widehat{\otimes}_{A_m}\mc{O}(\Omega)$. The lemma reduces to showing this map is an isomorphism of topological vector spaces. First note that since $D(T_0)$ is dense in $A_m$ for any $m \ge 0$, we have a natural topological isomorphism $M_m\widehat{\otimes}_{D(T_0)}\mc{O}(\Omega)\simeq M_m\widehat{\otimes}_{A_m}\mO(\Omega)$ for all $m$. Now by writing $D(I)_{N_H} = D(H)_{N_H}^{\oplus \#[I:H]}$ and $N_m := D^{(m)}(H)_{N_H}\widehat{\otimes}_{D^{(m)}(T_H)}B_m $, where $B_m:= D_{p^{-r_h}}(T_H)$, we may reduce to showing the map $D(H)_{N_H}\widehat{\otimes}_{D(T_H)}\mO(\Omega_H) \to \varprojlim_m N_m\widehat{\otimes}_{B_m}\mO(\Omega_H)$, for $\Omega_H\subset \widehat{T_H}$ denoting the image of $\Omega$ under $\widehat{T_0} \to \widehat{T \cap H}$. Note there is an isomorphism of right $D(T_H)$-modules $D(H)_{N_H}\simeq D(\ol{N}_H)\widehat{\otimes}_LD(T_H)$ and likewise $N_m \simeq D^{(m)}(\ol{N}_H)\widehat{\otimes}_L D^{(m)}(T_H)$, which natural fits into a diagram of right $D(T_H)$-modules
    \[
    \begin{tikzcd}
        D(H)_{N_H}\ar[d]\ar[r, "\simeq"]& D(\ol{N}_H)\widehat{\otimes}_{L} D(T_H)\ar[d]\\
        N_m \ar[r, "\simeq"]& D^{(m)}(\ol{N}_H)\widehat{\otimes}_{L} D^{(m)}(T_H)
    \end{tikzcd}
    \]
    Taking $-\widehat{\otimes}_{D(T_H)}\mc{O}(\Omega_H)$ of the entire diagram yields 
        \[
    \begin{tikzcd}
        D(H)_{N_H}\widehat{\otimes}_{D(T_H)}\mc{O}(\Omega_H)\ar[d]\ar[r, "\simeq"]& D(\ol{N}_H)\widehat{\otimes}_{L} \mc{O}(\Omega_H)\ar[d]\\
        N_m\widehat{\otimes}_{B_m}\mc{O}(\Omega_H) \ar[r, "\simeq"]& D^{(m)}(\ol{N}_H)\widehat{\otimes}_{L} \mc{O}(\Omega_H)
    \end{tikzcd}
    \]
    where for the top line we note $D(\ol{N}_H)\widehat{\otimes}_LD(T_H)\widehat{\otimes}_{D(T_H)}\mc{O}(\Omega_H)\simeq D(\ol{N}_H)\widehat{\otimes}_L\mc{O}(\Omega_H)$ by associativity of completed tensor product, and we note that likewise that $D(T_H)$ is dense in $D^{(h)}(T_H)$ to get the bottom row. Passing to the limit over $m$ in the bottom row then yields 
        \[
    \begin{tikzcd}
        D(H)_{N_H}\widehat{\otimes}_{D(T_H)}\mc{O}(\Omega_H)\ar[d]\ar[r, "\simeq"]& D(\ol{N}_H)\widehat{\otimes}_{L} \mc{O}(\Omega_H)\ar[d, "\simeq"]\\
        \varprojlim_m N_m\widehat{\otimes}_{B_m}\mc{O}(\Omega_H) \ar[r, "\simeq"]& \varprojlim_m D^{(m)}(\ol{N}_H)\widehat{\otimes}_{L} \mc{O}(\Omega_H)
    \end{tikzcd}
    \]
    But then the right vertical map is the evident isomorphism $D(\ol{N}_H)\widehat{\otimes}_L\mc{O}(\Omega_H)\simeq \varprojlim D^{(m)}(\ol{N}_H)\widehat{\otimes}_L\mc{O}(\Omega_H)$ coming from the Fr\'{e}chet structure, so the left vertical map is also an isomorphism, as desired.
\end{proof}
Let $J_{B}^{\vee}(C^{\on{BS}}(K^pI,D(I)_{N_0}))$ denote the dual Jacquet module, with the finite slope part taken with respect to the $U_2z$-action. These modules are naturally coadmissible $\mO_{\widehat{Y}}(\widehat{Y})$-modules (for example, one can use \eqref{eq: FactorizingUzOperator}!), and write $\mc{N}_{i}^{\on{BS}}$ for the associated coherent sheaves. Let $\mc{Z}_i:= \supp_{\mO_{\wh{Y}}(\wh{Y})}\mc{N}_{i}^{\on{BS}}$ be the support in degree $i$. Now let $U_{\le h}= \{p^{-h}\le |z|\le p^h\}\subset \G_m$ be the closed annulus of radius $p^{-h}$. As in \cite[Lemme 3.10]{BHSAnnalen}, we also have the following consequence.
\begin{corollary}
\label{corollary: ComplexesFredholmhypersurfaces}
   For each $i$, $\mc{Z}_i$ is a Fredholm hypersurface, and there is admissible affinoid cover $(U_j)_{j}$ of $\mc{Z}_i$ such that the projection $\pi_{\mc{W}}: U_j \to W_j \subset \mc{W}$ is a finite surjective map onto an affinoid open, and $\mc{N}_i^{\on{BS}}(U_j)$ is a finite projective $\mO_{\mc{W}}(W_j)$-module.
\end{corollary}
\begin{proof}
    Given \cref{proposition: U_pfactorisationDiagram}, this follows from the same proof as in \cite[Lemme 3.10]{BHSAnnalen}. Moreover, note that given $\lambda \in \mc{W}(\ol{\Q}_p)$, we could even take $U_i = \supp \mc{N}_i^{\on{BS}} \cap (\Omega \times U_{\le h})$, where $\Omega$ is an affinoid over which a slope $\le h$ decomposition exists, since the slope $\le h$ part of a (Pr)-module is finite projective over $\mO(\Omega)$.
\end{proof}

We now show how Emerton's finite slope part relates to slope decompositions.

\begin{lemma}[Comparing finite slope parts]
\label{lemma: RelateCompletedToOverconvergent}
  Suppose we are in the setting of \cref{proposition: U_pfactorisationDiagram}, so that $\Omega\subset \mc{W}$ is an affinoid open and $h \in \Q_{>0}$ is such that $C_{\bullet}^{\on{BS}}(K^pI, D(I)_{N_0}\widehat{\otimes}_{D(T_0)}\mO_{\mc{W}}(\Omega))$ admits a slope $\le h$ decomposition. Then we have a $T(F_p^+)$ and $\wt{\T}^S$-equivariant isomorphism
    \[
    \HH_{\ast}(J_B^{\vee}(C^{BS}(K^pI,D(I))))\otimes_{\mO_{\widehat{Y}}(\mc{W}\times \G_m)}(\mc{O}(\Omega)\widehat{\otimes}_L\mc{O}(U_{\le h}))\simeq \HH_{\ast}(C^{BS}(K^pI, D(I)_{N_0}\widehat{\otimes}_{D(T_0)}\mc{O}(\Omega))_{\le h}).
    \]
\end{lemma}
\begin{proof}
    We need to more explicitly compute $C^{\on{BS}}(K^pI, D(I)_{N_0})_{\fs}\otimes_{\mc{O}(\mc{W}\times \widehat{\G}_m)}(\mc{O}(\Omega)\widehat{\otimes} \mc{O}_{\G_m}(U_{\le h}))$. Noting the Amice transform isomorphism $D(T_0)\simeq \mc{O}_{\mc{W}}(\mc{W})$, we also then have $\mO_{\wh{Y}}(\mc{W}\times \wh{\G}_m)\simeq D(T_0)\wh{\otimes}_LL\{\{z^{\pm 1}\}\}$. Then we can write
    \begin{align*}
    &J_B^{\vee}(C^{\on{BS}}(K^pI, D(I)))\otimes_{\mO_{\wh{Y}}}(\mc{O}(\Omega)\widehat{\otimes}_L \mc{O}(U_{\le h}))\\ &\simeq (C^{\on{BS}}(K^pI,D(I)_{N_0})\widehat{\otimes}_{L[z]}L\{\{z^{\pm 1}\}\}) \widehat{\otimes}_{\mc{O}(\mc{W}\times \wh{\G}_m)}(\mc{O}(\Omega) \wh{\otimes}L\ip{p^hz, p^hz^{-1}})\\
    &\simeq (C^{\on{BS}}(K^pI, D(I)_{N_0})\widehat{\otimes}_{D(T_0)}\mc{O}(\Omega))\widehat{\otimes}_{L[z]}(L\{\{z^{\pm 1}\}\}\widehat{\otimes}_{L\{\{z^{\pm 1}\}\}}L\ip{p^{h}z, p^{h}z^{-1}})\\
    &\simeq C^{\on{BS}}(K^pI, D(I)_{N_0}\wh{\otimes}_{D(T_0)}\mO(\Omega))\widehat{\otimes}_{L[z]}L\ip{p^hz,p^hz^{-1}}.
    \end{align*}
    Now on $\Omega$ we have that by \cite[Proposition 4.5.1]{AS08} (as is used in  \cref{lemma: PrComplexesHaveSlopeDecompositions}) the Fredholm power series admits a slope $\le h$ factorisation $F_{\Omega}(X) = QR$, where $Q$ is a multiplicative polynomial of of slope $\le h$. Then in fact the proof of \cite[Lemme 3.10]{BHSAnnalen} implies the space $C^{\on{BS}}(K^pI, D(I)_{N_0}\wh{\otimes}_{D(T_0)}\mO_{\mc{W}}(\Omega))\widehat{\otimes}_{L[z]}L\ip{p^hz,p^hz^{-1}}$ is a projective module over $\mc{O}_{\mc{W}}(\Omega)$ which is precisely the locus on which $Q^{\ast}(U_z) = 0$. Recalling from \cite[2.2]{Han17} that a factorization $P = QR$ of the Fredholm series yields a description $M_{\le h}:= \{m \in M: Q^{\vee}(U_z) = 0\}$, we have our desired isomorphism.
\end{proof}

\subsection{An exact sequence on Jacquet functors}
We can now apply this formalism to prove the main useful tools we need for the Jacquet functors in the CM fields setting. The first goal is to use slope decompositions to get a long exact sequence on Jacquet functors.

\begin{lemma}
\label{lemma: CheckingExactnessOnACover}
    Let $X$ be quasi-Stein rigid space equipped with an admissible affinoid open cover $X = \cup_{i}U_i$. Then a sequence $\mc{M} \to \mc{N} \to \mc{P}$ of coherent sheaves on $X$ is exact if and only for all $i$ the sequence $\Gamma(U_i, \mc{M}) \to \Gamma(U_i, \mc{N}) \to \Gamma(U_i, \mc{P})$ is exact.
\end{lemma}
\begin{proof}
    Since $X$ is quasi-Stein, $\mc{M} \to \mc{N} \to \mc{P}$ is exact iff $\Gamma(X,\mc{M}) \to \Gamma(X, \mc{N}) \to \Gamma(X, \mc{P})$ is exact. On the other hand, exactness of sheaves can be checked locally, so $\mc{M} \to \mc{N} \to \mc{P}$ is exact iff $\mc{M}|_{U_i} \to \mc{N}|_{U_i} \to \mc{P}|_{U_i}$ is exact as coherent sheaves on $U_i$. Since $U_i$ is affinoid, $\mc{M}|_{U_i} \to \mc{N}|_{U_i} \to \mc{P}|_{U_i}$ is exact iff $\Gamma(U_i, \mc{M}) \to \Gamma(U_i, \mc{N})\to \Gamma(U_i, \mc{P})$ is exact, proving the lemma.
\end{proof}
Specializing the previous formalism to the case of quasi-split unitary group yields the following result.
\begin{lemma}
\label{lemma: FundamentalLongExactSequence}
    Let $G = \tG/\mO_{F^+}$ be the quasi-split unitary group. Let $\wt{\mf{m}} \subset \wt{\T}^S$ be maximal ideal in the support of $\HH^{\ast}(\wt{X}_{\wt{K}^p\wt{K}_p}, \ol{\F}_p)$ for some good level $\wt{K}^p\wt{K}_p$, such that $\wt{\mf{m}}$ is decomposed generic in the sense of \cref{definition: DecomposedGeneric}.
    Then we have a $\wt{\T}^S$-equivariant short exact sequence of essentially admissible $T(F_p^+)$-representations
    \[
    0 \to \HH^{d-1}(J_{B}(\Pi_{\partial}(\wt{K}^p)^{\la}))_{\wt{\mf{m}}} \to \HH^{d}(J_{B}(\Pi_c(\wt{K}^p)^{\la}))_{\wt{\mf{m}}} \to \HH^{d}(J_{B}(\Pi(\wt{K}^p)^{\la}))_{\wt{\mf{m}}} \to \HH^{d}(J_{B}(\Pi_{\partial}(\wt{K}^p)^{\la}))_{\wt{\mf{m}}} \to 0.
    \]
\end{lemma}
\begin{proof}
   First we claim there is a $\wt{\T}^S$-equivariant exact sequence of maps of essentially admissible $T(F_p^+)$-representations
      \[\begin{tikzcd}
\dots\ar[r]& \HH^{d-1}(J_B(\Pi(\wt{K}^p)_{\wt{\mf{m}}}^{\la})) \ar[d, phantom, ""{coordinate, name=Y}] \ar[r] &\HH^{d-1}(J_{B}(\Pi_{\partial}(\wt{K}^p)_{\wt{\mf{m}}}^{\la}))  \arrow[dll,
rounded corners, 
to path={ -- ([xshift=2ex]\tikztostart.east)
|- (Y) [near end]\tikztonodes
-| ([xshift=-2ex]\tikztotarget.west)
-- (\tikztotarget)}] \\
\HH^{d}(J_{B}(\Pi_c(\wt{K}^p)_{\wt{\mf{m}}}^{\la})) \ar[r] & \HH^{d}(J_{B}(\Pi(\wt{K}^p)_{\wt{\mf{m}}}^{\la}))  \ar[r] \arrow[d, phantom, ""{coordinate, name=Z}] &  \HH^{d}(J_{B_p}(\Pi_{\partial}(\wt{K}^p)_{\wt{\mf{m}}}^{\la})) \arrow[dll,
rounded corners,
to path={ -- ([xshift=2ex]\tikztostart.east)
|- (Z) [near end]\tikztonodes
-| ([xshift=-2ex]\tikztotarget.west)
-- (\tikztotarget)}, label = "...."]  \\
\HH^{d+1}(J_B(\Pi(\wt{K}^p)_{\wt{\mf{m}}}^{\la}))\ar[r] & \dots 
\end{tikzcd}
\] 

Given such a sequence, the result would follow by showing $\HH^{d-1}(J_B(\Pi(\wt{K}^p)_{\wt{\mf{m}}}^{\la})), \HH^{d+1}(J_B(\Pi_c(\wt{K}^p)_{\wt{\mf{m}}}^{\la})) = 0$. Indeed, both claims follow from, for example \cite[Lemma 5.4]{FuDerived} and the torsion vanishing statements in \cref{theorem: CaraianiScholzeTorsionVanishing}, due to Caraiani--Scholze and Koshikawa. The point is that $\HH^{d-1}(\wt{X}_{\wt{K}^p\wt{K}_p}, \F_p)_{\wt{\mf{m}}} = 0$ and $\HH_c^{d+1}(\wt{X}_{\wt{K}^p\wt{K}_p},\F_p)_{\wt{\mf{m}}} = 0$ imply a similar vanishing statement for any coefficient system admitting a suitable $\Z_p$-lattice. We now construct the above maps, and show the sequence is exact.

By dualizing it suffices to show we have a short exact sequence of coherent sheaves on $\widehat{T(F_p^+)}$
\[
\dots  \to M_{\partial}^d \to M^d \to M_c^d \to M_{\partial}^{d-1} \to \dots .
\]
The maps $M_{\partial}^d \to M^d \to M_c^d$ are naturally induced by maps of complexes, but we have not yet defined the boundary maps $M_c^i \to M_{\partial}^{i-1}$. Let $\mc{Z}$ be the union of the Fredholm hypersurfaces $\mc{Z}_i^{\on{BS}}, \mc{Z}_i^{\partial, \on{BS}}, \mc{Z}_i^{\on{BM},\on{BS}}$ from \cref{corollary: ComplexesFredholmhypersurfaces}, which is again a Fredholm hypersurface. Note that all the modules in the sequence are coadmissible over $\mc{Z}\subset \widehat{Y}$, a quasi-Stein space. Thus to define such a map $M_c^d \to M_{\partial}^{d-1}$, it suffices to define maps compatibly on an admissible open covering of $\mc{Z}$. Moreover, \cref{lemma: CheckingExactnessOnACover} tells us it suffices to check exactness on an admissible open covering of $\mc{Z}$. But taking any $x \in \mc{Z}$, and any $h \ge 0$, we can choose an affinoid open neighborhood $\Omega \subset \mc{W}$ such that all modules involved admit slope $\le h$ decompositions over $U_{h, \Omega}:= \mc{Z}\cap (\Omega\times U_{\le h})$. In particular, using \cref{corollary: DualCBSComplexesComputeJacquet} and \cref{lemma: RelateCompletedToOverconvergent} over $U_{h, \Omega}$, our sequence amounts to showing there is a sequence of the form
   \begin{equation}
   \label{eq: ExactSeqWithslopes}
   \HH_{d}(C_{\bullet}^{\partial}(K^pI, D(\Omega)))_{\le h} \to \HH_{d}(C_{\bullet}(K^pI, D(\Omega)))_{\le h} \to \HH_{d}(C_{\bullet}^{\on{BM}}(K^pI, D(\Omega)))_{\le h} \to \HH_{d-1}(C_{\bullet}^{\partial}(K^pI, D(\Omega)))_{\le h} 
   \end{equation}
   is exact.
However, note that there is a sequence of maps
\begin{equation}
\label{eq: ExactnessOverconvWithoutSlopes}
\dots \to \HH_{d}(C_{\bullet}^{\partial}(K^pI, D(\Omega))) \to \HH_{d}(C_{\bullet}(K^pI, D(\Omega))) \to \HH_{d}(C_{\bullet}^{\on{BM}}(K^pI, D(\Omega))) \to \HH_{d-1}(C_{\bullet}^{\partial}(K^pI, D(\Omega))) \to \dots
\end{equation}
arising from the distinguished triangle 
\[
0 \to C_{\bullet}^{\partial}(K^pI, D(\Omega)) \to C_{\bullet}(K^pI, D(\Omega)) \to C_{\bullet}^{\on{BM}}(K^pI, D(\Omega)) \to 0,
\]
and is thus exact\footnote{We thank Christian Johansson for noting we can apply slope decompositions \textit{after} passing to cohomology.}. The maps on cohomology are equivariant for the action of $T^+$, and slope $\le h$ decompositions are functorial direct summands, and are thus exact. So both the existence of the maps $M_{c}^{d} \to M_{\partial}^{d-1}$ and the exactness of \eqref{eq: ExactSeqWithslopes} follow from \eqref{eq: ExactnessOverconvWithoutSlopes}. The vanishing of the degrees beyond $d$ and below $d+1$ follow after localising at $\wt{\mf{m}}$.

\end{proof}

\subsection{Small slope classicality and torsionness of the boundary}
We can now prove \cref{theorem: SmallSupportIntro}. For this entire section we work with $\tG/\mO_{F^+}$, even though many of the intermediate lemmas work more generally.

For any continuous character $\delta: T_0 \to L^{\times}$, we denote standard action of $z$ on $\Ind_{B_0}^I\lambda$ by $z_{\on{std}}$, i.e. the action induced by conjugation $\mc{C}^{\la}(\ol{N}^1, \lambda) \xrightarrow{\on{res}}\mc{C}^{\la}(z^{-1}\ol{N}^1z, \lambda) \xrightarrow{\on{conj}_{z}}\mc{C}^{\la}(\ol{N}^1, \lambda)$. For any $w \in W_{\tG}$ we equip $\Ind_{B_0}^Iw_0(w\cdot \lambda)$ with the left $T^+$ action given by $z \ast_wf:= (w_0\lambda)(z)\cdot (w_0(w \cdot \lambda))^{-1}(z) \cdot (z_{\on{std}}f)$. Now let $\lambda\in \mc{W}$ be a dominant algebraic weight, so that $w_0 \lambda$ is anti-dominant. Then let $\Ind_{B_0}^{I,\on{alg}}w_0\lambda\simeq \mathsf{L}(\lambda)$ be the algebraic induction of $w_0\lambda$, which is isomorphic to the irreducible algebraic representation of highest weight $\lambda$. $\mathsf{L}(\lambda)$ may also be viewed as an irreducible algebraic representation of $I_m$ for any open subgroup $I_m\subset I$. Likewise, we let $L_{\on{lalg}}(\lambda) = (\Ind_{B_0}^Iw_0\lambda)^{\on{lalg}}$ denote the locally algebraic induction, which naturally has an inclusion $\mathsf{L}(\lambda) \hookrightarrow L_{\on{lalg}}(\lambda)$.

\begin{theorem}[Jones {\cite[Theorem 26]{Jon11}}]
\label{theorem: JonesBGGResolution}
    We have an exact sequence of admissible locally analytic $I$-representations, equivariant for $\ast_w$-actions of $T^+$
    \[
    \begin{tikzcd}
       0\to L_{\on{lalg}}(\lambda)\hookrightarrow \Ind_{B_0}^Iw_0\lambda \to \bigoplus_{\ell(w) = 1}\Ind_{B_0}^I(w_0(w\cdot \lambda)) \to \bigoplus_{\ell(w) = 2} \Ind_{B_0}^I(w_0(w\cdot \lambda)) \to \dots. 
    \end{tikzcd}
    \]
\end{theorem}
Passing to the dual side yields an exact sequence of $D(I)$-modules.
\[
\dots \to \bigoplus_{\ell(w) = 1}D(I)\otimes_{D(B_0)}w_0(w\cdot \lambda) \to D(I)\otimes_{D(B_0)}w_0\lambda \to (L_{\on{lalg}}(\lambda))^{\vee} \to 0.
\]
Note that we have a surjective map $(L_{\on{lalg}}(\lambda))^{\vee} \to \mathsf{L}(\lambda)^{\vee}$, where $\mathsf{L}(\lambda)^{\vee}$ is isomorphic to the irreducible algebraic representation of highest weight $-w_0\lambda$.
The (expected) punchline is that the homology of $(L_{\on{lalg}}(\lambda))^{\vee} $ is related to cohomology of algebraic representations, and thus classical automorphic forms.
\begin{lemma}
\label{lemma: ComputingClassicalCohomology}
We have an Hecke-equivariant isomorphism
    \[
    \HH_{\ast}(C^{\on{BS}}(\wt{K}^pI,  L_{\on{lalg}}(\lambda)^{\vee}))^{\vee}\simeq \varinjlim_{\wt{K}_p \supset B_0}\HH^{\ast}(\wt{X}_{\wt{K}^p\wt{K}_p},\mathsf{L}(\lambda)).
    \]
    Moreover, for any $h \in \Q_{\ge 0}$ we also have 
    \[
     \HH_{\ast}(C^{\on{BS}}(\wt{K}^pI,  L_{\on{lalg}}(\lambda)^{\vee})_{\le h})^{\vee}\simeq \HH^{\ast}(\wt{X}_{\wt{K}^pI},\mathsf{L}(\lambda))_{\le h}.
    \]
    Analogous results hold for compactly supported and boundary cohomology.
\end{lemma}
\begin{proof} Essentially by definition
$\HH_{\ast}(C_{\partial}^{\on{BS}}(\wt{K}^pI,  L_{\on{alg}}(\lambda)^{\vee}))^{\vee}$ computes $\HH^{\ast}(\wt{K}^pI, L_{\lalg}(\lambda)).$
Writing $L_{\on{lalg}}(\lambda) = \varinjlim_{m}L_{\on{m}-\on{alg}}(\lambda)$, where $L_{\on{m}-\on{alg}}(\lambda) \simeq \Ind_{I_m}^I\mathsf{L}(\lambda)$, and $I_m\subset I$ is a cofinal system of open subgroups containing $B_0$. Then Shapiro's lemma (see, for example, \cite[Corollary 2.21]{NT16}) implies $\HH^{\ast}(\wt{X}_{\wt{K}^pI}, \Ind_{I_m}^I\mathsf{L}(\lambda))\simeq \HH^{\ast}(\wt{X}_{\wt{K}^pI_m}, \mathsf{L}(\lambda))$, and thus

\[
\HH^{\ast}(\wt{X}_{\wt{K}^pI},L_{\on{lalg}}(\lambda)) = \varinjlim \HH^{\ast}(\wt{X}_{\wt{K}^pI}, L_{m-\on{alg}}(\lambda))\simeq \varinjlim \HH^{\ast}(\wt{X}_{\wt{K}^pI_m}, \mathsf{L}(\lambda)).
\]

By the first isomorphism, the second isomorphism reduces to showing that the natural maps $\HH^{\ast}(\wt{X}_{\wt{K}^pI_m}, \mathsf{L}(\lambda))_{\le h} \to \HH^{\ast}(\wt{X}_{\wt{K}^pI_{m+1}}, \mathsf{L}(\lambda))_{\le h}$ are isomorphisms. This claim follows from a variant of \cref{proposition: U_pfactorisationDiagram} (see e.g. \cite[Lemma 4.3.6]{Urb11}).
\end{proof}

We then get to our desired torsion-ness statement via two spectral sequences. For $y\in \mc{W}(\ol{\Q}_p)$, let $D(\lambda):= D(I)\otimes_{D(B_0)}y$.
\begin{lemma}
\label{lemma: UniversalCoeffSpecSeq}
    We have a Hecke-equivariant, $z$-equivariant spectral sequence of $\mc{O}(\Omega)$-modules:
    \[
    E_{p,q}^2 = \Tor_{p}^{\mc{O}(\Omega)}(\HH_q(C_{\bullet}^{\on{BS}}(\wt{K}^pI, D(\Omega))_{\le h}), \lambda)\implies \HH_{p+q}(C_{\bullet}^{\on{BS}}(\wt{K}^pI, D(\lambda)))_{\le h},
    \]
    and analogous sequences for Borel-Moore and boundary homology.
\end{lemma}
\begin{proof}
    The proof is essentially identical to \cite[Theorem 3.3.1]{Han17}, but with slightly different notation: this spectral sequence arises from the hyperhomology for the right exact functor $(-)\otimes_{\mc{O}(\Omega)}\lambda,$ but with several additional observations to get the slope-$\le h$ decompositions and Hecke and $T^+$-equivariance.
    Consider the hypertor group $\mathbf{Tor}_{p+q}^{\mc{O}(\Omega)}(C_{\bullet}(\wt{K}^pI, D(\Omega)), \lambda))$. Since $(C_{\bullet}(\wt{K}^pI, D(\Omega))$ is a complex of $\wt{\T}^S$-modules with $T^+$-action, we have a natural $\wt{\T}^S$ and $T^+$-action on this hyperhomology group. This group also admits a slope $\le h$ decomposition, which can also be computed as
    $\HH_{p+q}(C_{\bullet}^{\on{BS}}(K^pI, D(\Omega))_{\le h}\otimes_{\mc{O}(\Omega)}\lambda) = \HH_{p+q}(C_{\bullet}^{\on{BS}}(K^pI, D(\lambda))_{\le h})$, where we are using that $C_{\bullet}^{\on{BS}}(K^pI, D(\Omega))_{\le h}$ is complex of finite projective $\mc{O}(\Omega)$-modules. The hyperhomology spectral sequence gives a sequence of $\mc{O}(\Omega)$-modules
    \[
    E_{pq}^2 = \Tor_p^{\mc{O}(\Omega)}(\HH_q(C_{\bullet}(\wt{K}^pI, D(\Omega)))_{\le h},\lambda)\implies \HH_{p+q}(C_{\bullet}^{\on{BS}}(\wt{K}^pI, D(\lambda))_{\le h}),
    \]
    which is naturally equivariant for the $\wt{\T}^S$ and $T^+$-actions.
    
\end{proof}
For the next lemma, for any $w \in W_{\tG}$ and $h \in \Q_{>0}$, set $h(\lambda, w):= h - v_p((w\cdot \lambda)(w_0z)) + v_p(\lambda(w_0z))$.
\begin{lemma}
\label{lemma: BGGSpectralSequence}
    For any $w\in W_{\tG}$ and $h\in \Q_{>0}$ we have a Hecke and $T^+$-equivariant spectral sequence
    \[
    E_{1}^{p,q} = \bigoplus_{\ell(w) = q}\HH_{p}(C_{\bullet}^{\on{BS}}(\wt{K}^pI, D(w_0(w\cdot \lambda))))_{\le h(\lambda,w)} \implies \HH_{p+q}(C_{\bullet}^{\on{BS}}(\wt{K}^pI, \mathsf{L}(\lambda)^{\vee}))_{\le h},
    \]
    and analogous sequences for Borel--Moore and boundary homology.
\end{lemma}
\begin{proof}
This is essentially \cite[Theorem 4.4.1]{Urb11} in our context. We recall the construction for completeness. Since $C_{\bullet}^{\on{BS}}(\wt{K}^pI)$ is a complex of finite free $\Z[I]$-modules, tensoring the BGG resolution from \cref{theorem: JonesBGGResolution} with $C_{\bullet}^{\on{BS}}(\wt{K}^pI)$ yields a double complex with exact rows
    \[
   0 \to  C_{\bullet}^{\on{BS}}(\wt{K}^p)\otimes_{\Z[I]}D(w_0(w_0\cdot \lambda)) \to \dots \to  C_{\bullet}^{\on{BS}}(\wt{K}^p)\otimes_{\Z[I]}D(w_0(\lambda)) \to  C_{\bullet}^{\on{BS}}(\wt{K}^p)\otimes_{\Z[I]}L_{\lalg}(\lambda)^{\vee}\to 0.
    \]
    Equipping each term with its $U_2z$-action, this double complex is in fact $U_2z$ equivariant. The entire complex has slope $\le h$-decompositions. Unraveling the slope conditions for the $\ast_w$-actions gives a double complex with exact rows
    \[
   \dots \to \bigoplus_{\ell(w) = 1}C_{\bullet}^{\on{BS}}(\wt{K}^pI,D(w_0(w\cdot \lambda)))_{\le h(\lambda, w)}\to (C_{\bullet}^{\on{BS}}(\wt{K}^pI,D(w_0\lambda))_{\le h} \to (C_{\bullet}^{\on{BS}}(\wt{K}^pI),L_{\lalg}(\lambda)^{\vee})_{\le h} \to 0.
    \]
    In particular, the total homology in degree $i$ of the double complex 
    \[
       0 \to  C_{\bullet}^{\on{BS}}(\wt{K}^pI,D(w_0(w_0\cdot\lambda)))_{\le h(\lambda, w_0)} \to \dots \to  C_{\bullet}^{\on{BS}}(\wt{K}^pI,D(\lambda))_{\le h}
    \]
    is isomorphic to $\HH_{\bullet}(C^{{\on{BS}}}(\wt{K}^pI, \mathsf{L}(\lambda)^{\vee}))_{\le h}$ by \cref{lemma: ComputingClassicalCohomology}, since taking first homology of the columns degenerates by the BGG resolution and passing to slope $\le h$ parts is also exact. Passing first to homology of the columns yields the first page of a spectral sequence with $E_{p,q}^1 \simeq \bigoplus_{\ell(w) = q}H_p(C_{\bullet}^{\on{BS}}(\wt{K}^pI, D(w\cdot \lambda))_{\le h(\lambda, w)}$, as desired.
\end{proof}
We say $h$ is a \textit{small slope for} $\lambda$ if for all $w \in W_{\tG} \setminus \{1\}$ we have $h <v_p((w \cdot \lambda)(w_0z)) - v_p(\lambda(w_0z))$. Equivalently, $h(\lambda, w) < 0$ for all $w \neq 1$.
\begin{corollary}[Small slope classicality]
\label{corollary: SmallSlopeClassicality}
    If $h$ is a small slope for $\lambda$ then we have a Hecke-equivariant isomorphism
    \[
    \HH_{p}(C^{BS}(\wt{K}^pI, D(w_0\lambda)))_{\le h} \simeq \HH_p(C^{BS}(\wt{K}^pI, \mathsf{L}(\lambda)^{\vee}))_{\le h},
    \]
    as well as analogous statements for Borel-Moore and boundary homology.
\end{corollary}
\begin{proof}
    By \cref{lemma: BGGSpectralSequence} it suffices to show that for $w \neq 1$ that $\HH_p(C_{\bullet}^{\on{BS}}(\wt{K}^pI, D(w_0(w\cdot \lambda)))_{\le h(\lambda, w)} = 0$. All that is needed is that the action of $z$ stabilizes a $\Z_p$-lattice of $V_p:=\HH_p(C_{\bullet}^{\on{BS}}(\wt{K}^pI, D(w_0(w\cdot \lambda)))$. Indeed, such a fact would imply that the $U_2z$-operator must only have nonnegative slopes, and if $h$ is small slope then $h(\lambda ,w) < 0$ for all $w \neq 1$, so that $(V_{p})_{\le h(\lambda ,w)} = 0$. 
    
    For showing the stability of a lattice, note that the homotopy equivalence on $C^{\on{BS}}(\wt{K}^p)$ $\begin{tikzcd}C^{\on{BS}}(\wt{K}^p)\ar[r, "i"]&
    \ar[l, "p", shift left]C(\wt{K}^p)\end{tikzcd}$ is defined integrally, and thus the right $z$-action is also defined integrally. Likewise, it is clear the $z$-action on the Banach module $D^{(s)}(I)\widehat{\otimes}_{D^{(s)}(B_0)}\lambda$ preserves the unit ball of this Banach space, since on the dual side the map is induced by a restriction map $C^{(s)}(\ol{N}^1T_0, L) \to C^{(s)}(z^{-1}\ol{N}^1zT_0, L)$, which does not increase the norm. In particular, $C^{\on{BS}}(\wt{K}^p)\otimes_{\Z[I]} (D^{(s)}(I)\widehat{\otimes}_{D^{(s)}(B_0)} \lambda)^{\on{Norm} \le 1}$ is stable under the $U_2z$-action, and is a $\Z_p$-lattice of $C^{(s)}:= C^{\on{BS}}(\wt{K}^p)\otimes_{\Z[I]} (D^{(s)}(I)\widehat{\otimes}_{D^{(s)}(B_0)} w_0(w\cdot \lambda))$, and so $H_p(C^{(s)})_{\le h(\lambda, w)} \simeq \HH_p(\wt{K}^pI, D(w_0(w\cdot \lambda)))_{\le h(\lambda, w) } = 0$ for $s\gg 0$ and $h$ a small slope for $\lambda$. 
\end{proof}
\begin{corollary}
\label{corollary: SmallSupport}
    Let $\mc{Z}= \supp_{\mc{O}_{\widehat{Y}}(\wh{Y})}(J_B^{\vee}(C_{\bullet}^{\partial,\on{BS}}(\wt{K}^pI, D(I)_{N_0})))$. For any $i$, let $V_i:= (\HH^i(J_B(\Pi_{\partial}(\wt{K}^p)_{\wt{\mf{m}}}^{\la})))'$. Then let $U_{h, \Omega}:= \mc{Z} \cap(\Omega\times U_{\le h})$, where $\Omega$ is chosen so that over $\Omega$, $C_{\bullet}^{\partial,\on{BS}}(\wt{K}^p, D(I)_{N_0}\widehat{\otimes}_{D(T_0)}\mO(\Omega))$ admits a slope $\le h$ decomposition. Then $V_i(U_{h,\Omega})$ is a finitely generated \textit{torsion} module over $\mc{O}(\Omega)$.
\end{corollary}
\begin{proof}
By the above comparison theorems we have $V_i(U_{h, \Omega})= \HH_i(C^{\partial, \on{BS}}(\wt{K}^pI, D(I)_{N_0}\wh{\otimes}_{D(T_0)}\mO(\Omega)))_{\wt{\mf{m}},\le h}=:H_i$. We first prove torsion-ness in the case $\Omega$ contains a an algebraic weight. Thus by \cref{Lemma: MostWeightsAreCTG} there is an accumulating set of dominant algebraic weights which are CTG in the sense of \cref{definition: CTG}, and by the same argument we can in fact choose $w_0\lambda\in \Omega$ for $\lambda$ a dominant CTG algebraic weight. 
Applying
\cref{lemma: UniversalCoeffSpecSeq} gives a spectral sequence
\[
E_{i,j}^2 = \on{Tor}_{i}^{\mc{O}(\Omega)}(\HH_j,w_0\lambda) \implies  \HH_{i+j}(C^{\partial,\on{BS}}(\wt{K}^pI, D(w_0\lambda)))_{\wt{\mf{m}}, \le h}.
\]
When $h$ is a small slope for $\lambda$, \cref{corollary: SmallSlopeClassicality} gives an isomorphism
\[
\HH_{i+j}(C^{\partial,\on{BS}}(\wt{K}^pI, D(w_0\lambda)))_{\wt{\mf{m}}, \le h}\simeq \HH_{i+j}(\partial{\wt{X}}_{\wt{K}^pI}, \mathsf{L}(\lambda)^{\vee})_{\wt{\mf{m}}, \le h}.
\]
But then by \cref{lemma: ComputingClassicalCohomology} and the previously cited \cite[Proof of 2.4.11]{10AuthorPaper} that $\HH_{i+j}(\partial{\wt{X}}_{\wt{K}^pI}, \mathsf{L}(\lambda)^{\vee})_{\wt{\mf{m}}} =0$ if $\lambda$ is CTG.

We claim that all the $\HH_i$ are not supported at $w_0\lambda$, and are thus torsion. Let $\mf{p}_{w_0\lambda} \subset \mO(\Omega)$ be the prime ideal corresponding to the point $w_0\lambda$. For contradiction, let $j$ be the smallest degree such that $\on{supp}_{\mc{O}(\Omega)}\HH_j \ni \mf{p}_{w_0\lambda}$, i.e. $\HH_{j,\mf{p}_{w_0\lambda}} \neq 0$. Then by Nakayama's lemma we have $\HH_j\otimes_{\mc{O}(\Omega)}w_0\lambda \neq 0$, and for all $k < j$ and $d \ge 0$, so that that $\on{Tor}_d^{\mc{O}(\Omega)}(\HH_k, w_0\lambda)\simeq (\on{Tor}_d^{\mc{O}(\Omega)}(\HH_k, \lambda))_{\mf{p}_{w_0\lambda}} \simeq \on{Tor}_d^{\mc{O}(\Omega)}((\HH_k)_{\mf{p}_{w_0\lambda}}, w_0\lambda) = 0$. Thus, the term $\HH_j\otimes_{\mc{O}(\Omega)}w_0\lambda$ survives to the $E_{\infty}$-page of the spectral sequence, implying $\HH_j\otimes_{\mc{O}(\Omega)}w_0\lambda$ is identified with a submodule of $\HH_{j}(C^{\partial, \on{BS}}(\wt{K}^pI, D(w_0\lambda)))_{\wt{\mf{m}}, \le h} = 0$ by the previous discussion, giving a contradiction. So no such $j$ exists, and all the $\HH_{j}$ are torsion over $\mc{O}(\Omega)$, provided $\Omega$ contains a locally algebraic weight.

The case where $\Omega$ does not contain a locally algebraic weight follows exactly as Step 2 of \cite[proof of Thm 4.4.1]{Han17}.
\end{proof}

\subsection{Eigenvarieties, and their associated Galois representations}
\label{subsection: definingEigenvarietiesAndGaloisRepresentations}

We recall the definitions of the eigenvarieties we need for our purposes. Recall an \textit{eigenvariety datum} (\cite[Definition 6.1]{FuDerived}) is a tuple $D = (T, \mc{M}, \T, \psi)$ where $T$ is a topologically finitely generated abelian locally $\Q_p$-analytic group, $\mc{M}$ a coherent sheaf on the space of characters $\widehat{T}$, $\T$ a commutative $\Z_p$-algebra, and $\psi: \T\otimes_{\Z_p}\Q_p \to \End_{\mc{O}(\widehat{T})}(\mc{M})$ a $\Q_p$-algebra homomorphism. Associated to this data we can form a variety $\mc{E} = \mc{E}(D)$, defined by the relative Spec of $\mc{A}$ over $\widehat{T}$, where $\mc{A}$ is the $\mc{O}(\widehat{T})$-algebra generated by the image of the map $\psi$. Since $\mc{M}$ is a coherent sheaf, the natural projection map $\mc{E} \to \widehat{T}$ (which is a map of rigid analytic spaces) is finite. Very concretely, the points $\mc{E}(\ol{\Q}_p)$ is the collection of points $(x, \delta) \in \Hom(\T, \ol{\Q}_p) \times \widehat{T}(\ol{\Q}_p)$ such that $\Hom_T(\delta, \mc{M}[x]) \neq 0$ (see \cite[Proposition 6.2]{FuDerived}). In practice, $T$ is the points of a maximal torus of a reductive group, $\mc{M}$ is the dual of a derived Jacquet module, $\T$ is a spherical Hecke algebra, and $\psi: \T \to \End_{\mc{O}(\widehat{T})}(M)$ comes from an action on a locally symmetric space. Moreover, suppose that in fact $\T$ is a complete Noetherian local $\Z_p$-algebra. Then we can form an analogous construction of an $\mc{E}$ as a subspace of $(\on{Spf} \T^{\on{red}})^{\on{rig}} \times \widehat{T}$, where $(-)^{\on{rig}}$ denotes the rigid generic fibre (see \cite[\textsection 0.2]{deJong95}), and $\T^{\on{red}}$ denotes the (algebraic) reduction of $\T$. \footnote{Taking the reduced quotient is only needed for the global Galois representation on the eigenvariety for $G$.}

We apply this formalism to $G$ and $\wt{G}$. Namely, we consider the following eigenvariety data
\begin{align*}
    D^n(G)_{\mf{m}} &= (\widehat{T_n(F_p)}, \HH^n(J_B(\Pi(K^p)_{\mf{m}}^{\la}))', \T^S(K^p)_{\mf{m}},\psi_{\mf{m},n})\\
    D^n(\wt{G})_{\wt{\mf{m}}} &= (\widehat{T(F_p^+)}, \HH^n(J_B(\Pi(\wt{K}^p)_{\wt{\mf{m}}}^{\la}))', \wt{\T}^S(\wt{K}^p)_{\wt{\mf{m}}},\psi_{\wt{\mf{m}},n})\\
    D_{c}^n(\wt{G})_{\wt{\mf{m}}} &= (\widehat{T(F_p^+)}, \HH^n(J_B(\Pi_{c}(\wt{K}^p)_{\wt{\mf{m}}}^{\la}))', \wt{\T}^S(\wt{K}^p)_{\wt{\mf{m}}},\psi_{\wt{\mf{m}},n})\\
    D_{\partial}^n(\wt{G})_{\wt{\mf{m}}} &= (\widehat{T(F_p^+)}, \HH^n(J_B(\Pi_{\partial}(\wt{K}^p)_{\wt{\mf{m}}}^{\la}))', \wt{\T}^S(\wt{K}^p)_{\wt{\mf{m}}},\psi_{\wt{\mf{m}},n})
\end{align*}
where the $\psi$'s are the induced continuous Hecke actions on the modules in question, which exist essentially by transport of structure and localizing. From these data we define eigenvarieties $\mc{E}^n(K^p)_{\mf{m}},$ $\mc{E}^n(\wt{K}^p)_{\wt{\mf{m}}}, \mc{E}_c^n(\wt{K}^p)_{\wt{\mf{m}}}, \mc{E}_{\partial}^n(\wt{K}^p)_{\wt{\mf{m}}}$.
% \[
% \psi_{\mf{m}}: \T_{G, \mf{m}}^S \to \End_{\mc{O}(\widehat{T})}((\HH^{\ast}(J_B(\pi(G, K^p))_{\mf{m}}))')
% \]
\begin{remark}
One might wonder if we could use the abstract Hecke algebras $\wt{\T}^S$ and $\T^S$ instead of the ``big'' Hecke algebras $\wt{\T}^S(\wt{K}^p)$ and $\T^S(K^p)$. Indeed, we claim that this choice is irrelevant and the $\mc{O}(\widehat{T})$-algebras generated by the images of $\wt{\T}^S$ and $\T^S$ coincide with those coming from the big Hecke algebras. For simplicity, consider the case of $\psi_{\mf{m}, n}:\T_{\mf{m}}^S(K^p) \to  \End_{\mO_{\wh{T}}}(\mc{M})$, where $\mc{M} = \HH^n(J_B(\Pi(K^p)_{\mf{m}}^{\la}))'$. Let $\mc{L}$ be the coherent sheaf of endomorphisms generated by the image of $\psi_{\mf{m}, n}$. By construction, the natural map $\T^S \to \T^S(K^p)$ has dense image. It also suffices to work locally and thus consider a small enough neighborhood $(\Omega \times U_{\le h}) \subset \widehat{Y}$, so that $\mc{M}(\Omega \times U_{\le h})$ is finitely generated over $\mc{O}(\Omega)$ by \cref{proposition: U_pfactorisationDiagram}. By definition, $\mc{L}$ is the sheaf constructed by the image of the map $\wt{\T}^S(\wt{K}^p) \times \mc{O}_{\widehat{Y}} \to \End(\mc{M})$, and so the image of $\T^S(\wt{K}^p)\times \mc{O}_{\widehat{Y}}(\Omega \times U_{\le h}) \to \mc{L}(\Omega \times U_{\le h})$ is surjective, and $\mc{L}(\Omega \times U_{\le h})$ is finitely generated over $\mc{O}(\Omega)$. The image of $\T^S\times \mc{O}_{\widehat{Y}}(U_{\le h}\otimes \Omega)$ is an $\mc{O}(\Omega)$-submodule, and thus automatically closed in $\mc{L}(\Omega \times U_{\le h})$, so is equal to the image of its closure, as required.
\end{remark}

However, our eigenvariety data with the big Hecke algebras provides a cleaner connection to Galois deformation spaces. First, recall that given any maximal ideal $\mf{m} \subset \T^S(K^p)$, Scholze \cite{Sch15} has constructed the associated mod $p$ Galois representation $\ol{\rho}_{\mf{m}}: \Gal_F \to \GL_n(\ol{\F}_p)$. In other words, $\ol{\rho}_{\mf{m}}$ is a continuous, semisimple representation unramified outside a finite set of places $S$, such that if $\nu \notin S$, then the characteristic polynomial of $\ol{\rho}_{\mf{m}}(\Frob_{\nu})$ is equal to $P_{\nu}(X)\mod \mf{m}$. We say $\mf{m}$ is \textit{non-Eisenstein} if $\ol{\rho}_{\mf{m}}$ is absolutely irreducible. We then let $R_{\ol{\rho}_{\mf{m}}}$ be the universal deformation ring (which exists when $\ol{\rho}_{\mf{m}}$ is irreducible) and $R_{\ol{\rho}_{\mf{m}}}^{\square}$ the associated framed deformation ring (which always exists). Let $X(\ol{\rho}_{\mf{m}})$ and $X^{\square}(\ol{\rho}_{\mf{m}})$ be the associated rigid spaces.
An essentially immediate consequence of Scholze's construction of Galois representations associated to torsion classes then is the following. 

\begin{lemma}
\label{lemma: GaloisRepsAssociatedToEigenvarieties}
    Suppose $\mf{m} \subset \T^S(K^p)$ is a non-Eisenstein maximal ideal. Then
    there is a continuous $n$-dimensional Galois representation
    \[
    \rho:\Gal_{F} \to \GL_n(\T^S(K^p)_{\mf{m}}/I)
    \]
    for $I \subset\T^S(K^p)_{\mf{m}}$ a nilpotent ideal, which is unramified outside $S$
    such that for all $\nu \notin S$, the characteristic polynomial of $\rho(\Frob_{\nu})$ equals the image of $P_{\nu}(X)$ in $\T^S(K^p)_{\mf{m}}/I[X]$.

    As a result, there is a closed embedding $\mc{E}^i(K^p)_{\mf{m}} \hookrightarrow X^{\square}(\overline{\rho}_{\mf{m}}) \times \widehat{T_n(F_p)}$ (alternatively into $X(\ol{\rho}_{\mf{m}})\times \wh{T_n(F_p)}$).
\end{lemma}
\begin{proof}
    As mentioned in \cite[Remark 3.3.4]{GeeNewton}, we have a continuous $n$-dimensional representation valued in $\T^{S}(K^p)/I$ with $I$ a nilpotent ideal with $I^4 = 0$. Thus, there also exists a representation in the reduction.

    Now consider the Hecke action $\psi_{n,\mf{m}}: \T^S(K^p)_{\mf{m}} \to \End_{\mc{O}(\widehat{T_n(F_p)})}(\HH^{n}(J_B(\Pi(K^p)_{\mf{m}}^{\la})))'$.
Let $\mc{A}_{i}$ be the reduced $\mc{O}_{\wh{T}}$-algebra generated by $\on{im}(\psi_{i, \mf{m}})$. We thus have an algebra homomorphism $R_{\overline{\rho}_{\mf{m}}} \to (\T^{S}(K^p)_{\mf{m}})^{\on{red}} \to \mc{A}_i$. Passing to spectra, we have a natural map $\mc{E}^i(K^p)_{\mf{m}} \hookrightarrow (\Spf (\T^S(K^p)_{\mf{m}}^{\on{red}}))^{\on{rig}}\times \widehat{T} \to (\Spf R_{\overline{\rho}_m})^{\on{rig}} \times \widehat{T}$.
\end{proof}

We need to shift $\wh{T_n(F_p)}$ to match the weights of Galois representations. Let $j_n: \wh{T_n(F_p)}\to \wh{T_n(F_p)}$ be the map defined by 
\[
j_n(\chi)_{\nu}:= \chi_{\nu}\cdot(1, \eps_{\on{cyc}}^{-1}\circ \Art_{F_{\nu}}, \dots, \eps_{\on{cyc}}^{1-n}\circ \Art_{F_{\nu}})
\]

For notational convenience, in the sequel we will write $\rho_{\cyc}^{G}: T_n(F_p) \to \ol{\Q}_p^{\times}$ to denote the character $(\rho_{\cyc}^{G})_{\nu}:= (1, \varepsilon_{\cyc}^{-1}\circ \Art_{F_{\nu}}, \dots, \varepsilon_{\cyc}^{-1}\circ \Art_{F_{\nu}})$. Similarly, we define $\rho_{\cyc}^{\tG}: T(F_p^+) \to \ol{\Q}_p^{\times}$ defined by the formula $(\rho_{\cyc}^{\tG})_{\ol{\nu}} = (1, \varepsilon_{\cyc}^{-1}\circ \Art_{F_{\nu}}, \dots, \varepsilon_{\cyc}^{1-2n}\circ \Art_{F_{\nu}})$. Moreover, given a character $\delta: T_n(F_p)\to \ol{\Q}_p^{\times}$, we will write the product $\delta \cdot \rho_{\cyc}^{G}$ (respectively, $\delta \cdot \rho_{\cyc}^{\tG}$) as $\delta + \rho_{\cyc}^{G}$ (respectively, $\rho  + \rho_{\cyc}^{\tG}$). In particular, $j_n(\delta):= \delta + \rho_{\cyc}^{\tG}$, and likewise if we view $\delta: T_n(F_p) \to \ol{\Q}_p^{\times}$ as a character of $T(F_p^+)$, we define $j_{2n}(\delta):= \delta + \rho_{\cyc}^{\tG}$.

In essence, the goal of the paper is to constrain the image of these maps to the deformation space
\[
\iota_{\mf{m}}^i:\mc{E}^i(K^p)_{\mf{m}} \hookrightarrow X^{\square}(\overline{\rho}_{\mf{m}}) \times \wh{T_n(F_p)},
 \]
 sending $(x,\delta) \mapsto (\rho_x,j_n(\delta))$.

\subsection{Reduction to a parabolic induction}
The relation between the completed cohomology of $G = \Res_{F/F^+}\GL_n$ and the boundary (completed) cohomology of $\tG$ is packaged nicely as follows: 
\begin{proposition}
\label{prop: InductionAsDirectSummandOfBoundary}
    If $\mf{m}\subset \T^S(K^p)$ is non--Eisenstein and $\wt{K}\subset \tG(\A_{F^+}^{\infty})$ is a good subgroup decomposed with respect to $\oP,$ then in the derived category $D^b(\mathsf{Ban}_L^{\on{ad}}(\wt{G}_0)$ we have
$\on{ct}-\Ind_{\oP_0}^{\wt{G}_0}\Pi(K^p)_{\mf{m}}$ is a $\wt{\T}^S$-equivariant direct summand of $\Pi_{\partial}(\wt{K}^p)_{\mf{m}}$.
\end{proposition}
\begin{proof}
    This is \cite[Proposition 7.1]{FuDerived}, which in turn is deduced from the calculation \cite[Theorem 5.4.1]{10AuthorPaper}.
\end{proof}

For the sake of taking derived functors we want a complex of \textit{injective} $\wt{G}_0$-representations representing $\on{ct}-\Ind_{\oP_0}^{\wt{G}_0}\Pi(K^p)_{\mf{m}}$. Indeed, let $\Pi_{\oP}(\wt{K}^p):= \Pi(\wt{K}^p \cap \oP(\A_{F^+}^{\infty, p}))$ be a complex of injective admissible $\oP_0$-representations computing the completed cohomology for the (opposite) Siegel parabolic subgroup $\oP$. Recall that $X_{\wt{K}_{\oP}}^{\oP} \to X_K^G$ has the structure of a torus bundle. Then (as in \cite[(5.4.2)]{10AuthorPaper} and comments thereafter) there is a $\wt{\T}^S$-equivariant map of complexes of $\oP_0$-representations $\Pi(K^p) \to \Pi_{\oP}(\widetilde{K}^p)$, where the $\wt{\T}^S$-action on the left is induced by the Satake transfer $\mc{S}:\wt{\T}^S \to \T^S$. In fact, letting $A_{\tG}^{\bullet}$ be the complex of Banach spaces which makes $\Pi(K^p)$ a $G$-extension, and making the analogous construction $A_{\oP}^{\bullet}$ for $\oP$, pullback induces a map $\pi_G^{\ast}:A_{\tG}^{\bullet} \to A_{\oP}^{\bullet}$, and the map $\Pi(K^p) \to \Pi_{\oP}(\widetilde{K}^p)$ can be chosen to be compatible with $\pi_{G}^{\ast}$. Since the completed cohomology of a unipotent subgroup $\ol{N}$ is concentrated in degree $0$ (in particular, see the second half of \cite[Proof of Theorem 5.4.1]{10AuthorPaper}), this map is in fact a quasi-isomorphism. Moreover by passing to the dual side and again invoking \cite[Corollary 10.4.7]{WeibelHomAlg}, we have a $G_0$-equivariant homotopy inverse $\Pi_{\oP}(\wt{K}^p)\to \Inf(\Pi(K^p))$. This map is in fact also continuous, since any map $\mO_L[[G_0]]$-linear map $\mO_L[[G_0]] \to \mO_L[[\oP_0]]$ is automatically continuous.

Moreover for any maximal ideal $\wt{\mf{m}}\subset \wt{\T}^{S}(\wt{K}^p)$ we can also run the same procedure as in \cref{subsection: HeckeAlgebras} to get a complex $\Pi_{\oP}(\wt{K}^p)_{\wt{\mf{m}}}$ of injective admission $\oP_0$-representations equipped with a $\oP$-extension.
Thus, $\on{ct}-\Inf_{\oP_0}^{\tG_0}\Pi_{\oP}(\wt{K}^p)_{\wt{\mf{m}}}$ is a complex of injective $G_0$-representations representing $\on{ct}-\Ind_{\oP_0}^{\tG_0}\Pi(\wt{K}^p)_{\mf{m}}$. Moreover, by \cite[Proposition 3.13]{FuDerived} we also have homotopy equivalences 
\[
\on{ct}-\Ind_{\oP_0}^{\widetilde{G}_0}\Pi_{\oP}(\widetilde{K}^p)_{\wt{\mf{m}}} \to \on{ct}-\Ind_{\oP}^{\widetilde{G}_p}A_{\oP,\mf{m}}^{\bullet}
\]

As a result, $\on{ct}-\Ind_{\oP_0}^{\widetilde{G}_0}\Pi_{\oP}(\widetilde{K}^p)_{\wt{\mf{m}}}$ is a complex of injective admissible Banach $\wt{G}_0$-representations, equipped with a $\wt{G}_p$-extension, whose cohomology computes $\on{ct}-\Ind_{\oP}^{\tG}\wt{\HH}^{\ast}(K^p, L)_{\mf{m}}$. Then using \cite[Theorem 3.6]{FuDerived}, along with our previously defined formulae, we may consider the associated Jacquet functors \[
J_B(\Ind_{\oP_0}^{\widetilde{G}_0}\Pi_{\oP}(\widetilde{K}^p)_{\wt{\mf{m}}}^{\la}) = (\Ind_{\oP_0}^{\widetilde{G}_0}\Pi_{\oP}(\widetilde{K}^p)_{\wt{\mf{m}}}^{\la})_{\on{fs}}^{N_0} 
\]
with respect to the $\wt{G}$-extension provided. We can then likewise define an eigenvariety datum $D_{\mf{m}}^{\Ind}(G) = (\widehat{T(F_p^+)}, \HH^n(J_B(\Ind_{\ol{P}_0}^{\wt{G}_0}(\Pi(K^p)_{\mf{m}}^{\la})))', \wt{\T}^S(\wt{K}^p)_{\wt{\mf{m}}}, \psi_{\Ind,\mf{m}})$ and consider an associated eigenvariety $\mc{E}_{\Ind}^n(K^p)_{\wt{\mf{m}}}$.

In the case that $\wt{\mf{m}}\subset \wt{\T}^S(\wt{K}^p)$ is so that $\ol{\rho}_{\wt{\mf{m}}}$ is decomposed generic as defined in \cref{definition: DecomposedGeneric}, by \cref{lemma: FundamentalLongExactSequence} we have an injective (resp. surjective) map of essentially admissible representations
\begin{align}
\HH^{d-1}(J_B(\Ind_{\oP_0}^{\tG_0}\Pi_{\oP}(\wt{K}^p)_{\wt{\mf{m}}}^{\la})) &\hookrightarrow \HH^d(J_B(\Pi_c(\wt{K}^p)_{\wt{\mf{m}}}^{\la}))
    \\
\HH^d(J_B(\Pi(\wt{K}^p)_{\wt{\mf{m}}}^{\la}))&\twoheadrightarrow \HH^{d}(J_B(\Ind_{\oP_0}^{\tG_0}\Pi_{\oP}(\wt{K}^p)_{\wt{\mf{m}}}^{\la})).
\end{align}
Implicitly, here we are using that the direct summand $\HH^{i}(J_B(\Ind_{\oP_0}^{\tG_0}\Pi_{\oP}(\wt{K}^p)_{\mf{m}}^{\la})) \to \HH^{i}(J_B(\Pi_{\partial}(\wt{K}^p)_{\wt{\mf{m}}}^{\la}))$ is $T(F_p^+)$-equivariant, which follows from the fact that \cite[Theorem 5.4.1]{10AuthorPaper} in fact gives the original direct summand is $\tG(F_p^+)$-equivariant (up to homotopy).

So to gain an understanding of Jacquet functors for $\Pi(K^p)_{\mf{m}}^{\la}$, we will start by understanding the Jacquet functors of a parabolic induction.

\section{Locally analytic parabolic inductions}
\label{Section3: ParabolicInduction}

In this section, we make a more detailed study of parabolic inductions. We introduce the following notation.
Let $F/\Q$ be a number field and recall $F_p:= F\otimes_{\Q}\Q_p= \prod_{\nu \in S_{p}(F)}F_{\nu}$.  Let $L/\Q_p$ be a finite extension.

Let $\tG = \tG(F_p)$ be the $F_p$-points of a split connected reductive group over $F_p$, $B = B(F_p) = T(F_p)N(F_p)$ a Borel with choice of split maximal torus $T\subset B$ and $P = P(F_p) = G(F_p)U(F_p) \supset B$ a parabolic subgroup containing $B$ with Levi $G$ and unipotent radical $U$. Let $\oP := \oP(F_p)$ be the corresponding opposite parabolic. Even though we will only need to take $\tG = \GL_{2n}(F_p)$ and $P = G(F_p)U(F_p)$ the (opposite) Siegel parabolic, the following section applies to any split reductive group.

Let $\pi$ be a locally analytic representation of $G(F_p)$ on a compact type vector space over $L$.
We are interested in studying locally analytic parabolic inductions
\begin{equation}
\label{eq: LocallyAnalyticInductionDefinition}
\Ind_{\oP}^{\tG}\pi := \{f \in C^{\rm{la}}(\tG, \pi): f(pg) = \pi(p)f(g) \text{ for all }p \in \overline{P}, g \in \tG\}.
\end{equation}
as a representation of the Borel $B$. Since $\oP\backslash \tG$ is compact, $\Ind_{\oP}^{\tG}\pi$ is again a compact type $L$-vector space.

Let $W_{\oP}^r := W_G\backslash W_{\tG}$ be the relative Weyl set. Then a generalized Bruhat decomposition gives a stratification into locally closed subsets
\[
\tG(F_p) = \bigsqcup_{w \in W_{\oP}^r}\oP(F_p)wB(F_p)  = \bigsqcup_{w \in W_{\oP}^r}\oP(F_p)wN(F_p),\]
so that in particular $W_{\oP}^r:= \oP(F_p)\backslash \tG(F_p)/B(F_p)$, which additionally satisfies that $\oP N \subset G$ is Zariski open and dense, and if $C_w:= \oP(F_p)wN(F_p)$, then the closure $\overline{C}_w = \sqcup_{w': \ell(w') \le \ell(w)}C_{w'}$, where $\ell(w)$ is understood to be the minimal length of a representative element in $W_{\tG}$ of the coset of $w$.

Let $W^{\oP}$ be a collection of minimal length representatives for each class in $W_{\oP}^r$. Now fix an ordering $\{w_i\}_{i = 1}^{\#{W_{\oP}^r}} = W^{\oP}$ such that if $i \le j$ then $\ell(w_i) \le \ell(w_j)$. Then define for all $k \ge 1,$
\[
\tG_k := \bigsqcup_{1\le i \le k}\oP(F_p)w_iN(F_p),
\]
which is naturally an open and dense subset of $\tG(F_p)$.
We first note that we can write coordinates for an open cell. 

\begin{lemma}
\label{lemma: chartopencell}
    We have an inclusion $\ol{P}wN\subset \ol{P}Uw$. Moreover we have an isomorphism of locally $\Q_p$-analytic manifolds 
    \[
    \oP w N \times (w^{-1}Uw\cap \ol{N}) \simeq \oP U w
    \]
    induced by the product $(x,y) \mapsto x \cdot y$.
\end{lemma}
\begin{proof}
    This inclusion is a general fact about reductive groups. Let $\Phi^+$ (resp. $\Phi^{-}$) be the set of positive (resp. negative) roots with respect to $(B,T)$. Moreover let $\Phi_U^+$ be the set of roots in the unipotent radical $U\subset P$. Now first note that it suffices to show that $wNw^{-1} \subset \oP N = \oP U$. The key fact is that as algebraic varieties (or schemes) over $F_p$, we have $N = \prod_{\alpha \in \Phi^{+}}U_{\alpha}$, where $U_{\alpha}$ is the root group  associated with a positive root $\alpha$, and the product can be taken \textit{in any order}. As a result, we may rewrite
    \begin{equation}\label{eq: IdentifyUnipotentInOpenCell}
    wNw^{-1} = \prod_{\alpha \in w(\Phi^+)w^{-1}}U_{\alpha} = \left(\prod_{\alpha \in w(\Phi^+)w^{-1}\cap \Phi^-}U_{\alpha}\right)\cdot \left(\prod_{\beta \in w(\Phi^+)w^{-1}\cap \Phi^+}U_{\beta}\right) \subseteq \oP \cdot N,
    \end{equation}
    proving the first part. For the second part, we note first that $\oP U w \simeq \oP \times Uw$ as algebraic varieties. 
 Likewise, \eqref{eq: IdentifyUnipotentInOpenCell} implies that $\oP wN \simeq \oP \times \left(\prod_{\beta \in w(\Phi^+)w^{-1}\cap \Phi_U^+}U_{\alpha}\right)w$. For the result, it then suffices to show the $U \cap w\overline{N}w^{-1}\simeq \prod_{\alpha \in  w\Phi^-w^{-1}\cap \Phi_U^+}U_{\alpha}$, and that $\Phi_U^+\cap w\Phi^-w^{-1} \sqcup w\Phi^-w^{-1}\cap \Phi_U^+ = \Phi_U^+$, so that all the relevant roots are represented exactly once.
    \end{proof}

\cref{lemma: chartopencell} gives us a convenient way to discuss open neighborhoods of Bruhat strata. Namely, if $U_{\ul{x}} \subset w^{-1}Uw \cap \overline{N}$ is a compact open subset, then the subset $U_{w,\ul{x}}:=\ol{P}wN \times \left(U_{\ul{x}}\right)\subset \ol{P}Uw$ is naturally an open set containing the Bruhat stratum $C_w$.
% \cap w\overline{U}w^{-1}
We recall that any function $f \in \Ind_{\oP}^{\tG}\pi$ has a well-defined support $\supp f \subseteq \oP\backslash G$, given by the locus on which the germ of $f$ is nonzero. By definition, the support is a compact and open subset of the quotient $\oP\backslash \tG$.

Now, for any open subset $\Omega \subset \oP\backslash \tG$, we define $(\Ind_{\oP}^{G}\pi)(\Omega)$ to be the subspace of $f$ such that $\supp f \subseteq \Omega$. For shorthand we also write $I_{\Omega}\pi:= (\Ind_{\oP}^{G}\pi)(\Omega)$.

We now make our key definitions for the Bruhat filtration on locally analytic inductions:
\begin{definition}
 For any $k \ge 0$ and $w \in W_{\oP}^r$, define the functors $I_{\le k}, I_w: \Rep_{\rm{la.c}}(P) \to \Rep(B)$ by
\[
I_{\le k}(\pi) := (\Ind_{\oP}^{\tG}\pi)(\tG_k),\quad
\]
\[
I_w(\pi) = \varinjlim (\Ind_{\oP}^{\tG}\pi)(U_{w, \ul{x}}),
\]
with transition maps $\Ind_{\oP}^{\tG}V(U_{w, \ul{x}}) \to \Ind_{\oP}^{\tG}V(U_{w, \ul{y}})$ is the map $f \mapsto f \cdot \mathbf{1}_{U_{w, \ul{y}}}$.
\end{definition}

Some remarks are in hand. In these definitions,  $I_{\le k}\pi\subset \Ind_{\oP}^{\tG}\pi$ is naturally a closed subspace in $\Rep_{\rm{la}.c}(B)$.
For $I_{w}$, note that the transition maps are well-defined since if $\supp f \subset U_{w, \ul{x}}$ is a compact open subset then $(\supp f) \cap U_{w,\ul{y}}$ remains compact open, since $\left(U_{\ul{y}} \cap w\overline{U}w^{-1}\right) \subset \left(U_{\ul{x}} \cap w\overline{U}w^{-1}\right)$, and so $(\supp f) \cap U_{w, \ul{y}} \subset \supp f$ is closed, thus compact. We then endow  $I_w$ with the locally convex inductive limit topology.

\begin{remark}
    % - to remedy the $I_w$-definition, can I just look elementwise? Say two elements are the same iff they have some common refined subset $K'$ such that $f\cdot 1_{K'} = g\cdot 1_{K'}$ such that $\supp f \cap \oP wN = \supp f \cap \overline{\oP w N} = \supp g \cap \oP w N = \supp K' \cap \oP w N$.
    As in \cite{10AuthorPaper}, we could define a certain sub-functor $I_w^{\circ}\pi \subset I_w\pi$ defined by germs of functions $f$ such that some equivalent function satisfies $\supp f_{\ul{x}_0} \cap PwN \subseteq PwN_0$. A key point in \cite[\textsection 5.3]{10AuthorPaper} is that in the smooth mod $p$ case, the ordinary parts of $I_w^{\circ}$ and $I_w$ coincide, and the ordinary parts of $I_w^{\circ}$ is particularly simple and related to $G$. In our setting, we would expect that the Jacquet functors of $I_w^{\circ}$ and $I_w$ agree on $T$-\textit{eigenspaces}, but $I_w^{\circ}$ remains a somewhat complicated object. As our argument does not use this subfunctor, we do not study it.
\end{remark}

A variant of \cite[Lemma 1.4.10]{Emerton_Jacquet_II} implies that $I_w\pi$ is also of compact type if $\pi\in \Rep_{\rm{la}.c}(\oP)$ is of compact type.
% \begin{lemma}
%     If $\pi \in \Rep_{\rm{la}.c}(\oP)$ then $I_w\pi$ equipped with its locally convex inductive limit topology is a space of compact type.
% \end{lemma}
It will be convenient to note the following alternative description:

\begin{lemma}
If $\pi\in \Rep_{\rm{la}.c}(\oP)$ then we have a $B$-equivariant isomorphism of topological vector spaces \[
    I_{\le k}\pi \simeq \varinjlim_{\substack{U \subset \oP\backslash \oP\tG_k\\ \text{compact open}}}(\Ind_{\oP}^{\tG}\pi)(U).
    \]
\end{lemma}
\begin{proof}
    Since any function $f \in \Ind_{\oP}^{G}\pi$ has compact support we indeed have a continuous bijective map 
    \[
    \varinjlim_{\substack{U \subset \oP\backslash \oP\tG_k\\ \text{compact open}}}(\Ind_{\oP}^{\tG}\pi)(U) \to \Ind_{\oP}^{\tG}\pi(\tG_k) = I_{\le k}\pi.
    \]
    Moreover, both spaces are of compact type (the left side being compact type because for $U \subsetneq V \subset \oP\backslash \tG$ compact open, the inclusion map $\Ind_{\oP}^{\tG}\pi(U) \to \Ind_{\oP}^{\tG}\pi(V)$ is compact, so by the open mapping theorem we have a topological isomorphism.
\end{proof}
\begin{lemma}
\label{lemma: ExactnessLocallyanalyticFunctions}
    If $0 \to U \to V \to W \to 0$ is a strict exact sequence of $L$-vector spaces of compact type and $X$ any compact locally $L$-analytic manifold, the sequence $0 \to \mc{C}^{\rm{la}}(X, U) \to \mc{C}^{\rm{la}}(X, V) \to \mc{C}^{\rm{la}}(X, W) \to 0$ is also strict exact. 
\end{lemma}
\begin{proof}
    The maps in the stipulated sequence are clearly continuous $L$-linear and all the spaces are of compact type by \cite[2.1.28]{Emerton_la_reps}, and the map $\mc{C}^{\rm{la}}(X, U) \to \mc{C}^{\rm{la}}(X, V)$ is a closed embedding by \cite[2.1.29]{Emerton_la_reps}, so by the open mapping theorem it suffices to show that the sequence is exact as abstract $L$-vector spaces. 
    
    Note that when $X$ is compact and $M\simeq \varinjlim_{n}M_n$ is of compact type, we have a topological isomorphism $\mc{C}^{\rm{la}}(X,M) \simeq \varinjlim_{\{X_i\}}\prod_i\varinjlim_{n}\mc{C}^{\rm{an}}(\mathbb{X}_i, L)\widehat{\otimes}_L M_n$ where the $\{X_i\}$ runs over disjoint partitions of $X$ by analytic charts (see \cite[p. 38]{Emerton_la_reps}). Since $X$ is compact, we can take each of these products to be finite. Moreover by \cite[2.1.10]{Emerton_la_reps} that we have a topological isomorphism $\mc{C}^{\rm{an}}(\mathbb{X}_i,L) \simeq c_0(\N, L)$, where $c_0(\N, L)\subseteq \mc{C}(\N \cup \{\infty\}, L)$ is the subspace of continuous functions on the one-point compactification of $\N$ which vanish at $\infty$. Since the $M_n$ are Banach (in particular, Fr\'{e}chet) we have $\mc{C}(\N \cup \{\infty\}, L)\widehat{\otimes}_LM_n \simeq \mc{C}(\N \cup \{\infty\}, M_n)$, and this identification will clearly match up $c_0(\N, L)\widehat{\otimes}_LM_n \simeq c_0(\N , M_n)$.
    Therefore we may rewrite any
    \[
    \mc{C}^{\rm{la}}(X, M) \simeq \varinjlim_{\{X_i\}}\prod_i\varinjlim_{n}c_0(X, M_n).
    \]
    However, by \cite[2.1.6]{Emerton_la_reps} we have a natural continuous bijection $\varinjlim_{n}c_0(X, M_n) \to c_0(X, M)$ (the $M_n$ are envelopes for the so-called $BH$-subspaces of $M$). It is not clear if this map is a homeomorphism, but we are only interested in the vector space structure. In particular, since filtered colimits and finite products are exact on vector spaces, it suffices to show the exactness of the sequence 
    \[
    0 \to c_0(\N, U) \to c_0(\N, V) \to c_0(\N, W) \to 0.
    \]
    But such an exactness is clear.
\end{proof}

The relevance of this result is to show exactness of the functor $I_{\le k}$, as follows. First note that the projection map $\tG \to \oP\backslash\tG$ has a locally analytic section $s: \oP \backslash \tG \to \tG$, and thus as topological vector spaces, we have an isomorphism $\Ind_{\oP}^{\tG}\pi \simeq \mc{C}_{c}^{\rm{la}}(s(\oP \backslash \tG ), \pi)$.
 For an arbitrary open subset $\Omega\subset \oP\backslash \tG$ this isomorphism induces isomorphisms  $I_{\Omega}\pi\simeq \mc{C}_c^{\rm{la}}(\Omega, \pi)$. As a result, $I_k(\pi) = \mc{C}_c^{\rm{la}}(\oP\backslash (\oP \cdot \tG_k), \pi)$. 
We can then prove the following:
 \begin{lemma}
    All the functors $I_{\le k},$ $I_w$ send strict exact sequences in $\Rep_{\rm{la}.c}(G)$ to strict exact sequences in $\Rep_{\rm{la.c}}(B)$. Moreover for all $k$ we have a strict exact sequence 
    \begin{equation}
    \label{eq: InductiveBruhatSequence}
    0 \to I_{\le k}(\pi) \to I_{\le k+1}(\pi) \to I_{w_{k+1}}(\pi) \to 0.
    \end{equation}
\end{lemma}
\begin{proof}
    $I_{\le k}$ is exact by \cref{lemma: ExactnessLocallyanalyticFunctions} combined with the above discussion.
    Likewise, $I_w$ is a filtered colimit of $I_{V}\pi$ for $V$ compact open, which is exact. All these sequences are strict since the spaces in discussion are vector spaces of compact type, and the first maps will then be closed embeddings.

    For showing \eqref{eq: InductiveBruhatSequence}, the first map is clearly a closed embedding and the second map is clearly surjective. For exactness in the middle, suppose $f\in I_{\le k+1}(\pi)$ maps to zero in $I_{w_{k+1}}\pi$. By definition, for some compact open neighborhood $U_{w_{k+1}, \ul{x}} \subseteq \oP\backslash \wt{G}$ we have $f|_{U_{w_{k+1}, \ul{x}}} = 0$. Thus, $\supp f \subset (U_{w_{k+1}, \ul{x}})^c$. Since $U_{w_{k+1}, \ul{x}} \supset \oP w_{k+1}N$, we have $\oP w_{k+1}N\cap \supp f\subseteq  \left(\oP w_{k+1}N\right)\cap (U_{w_{k+1}, \ul{x}})^c = \emptyset$. 
    
    % Dividing up $\supp f = \supp f\cap \oP\backslash \oP\tG_k \cup \supp f \cap \oP w_{k+1}N,$ we then get that  
    % \begin{align*}\supp f = \supp f \cap (K')^c &= (\supp f\cap \oP\backslash \oP\tG_k \cup \supp f \cap P w_{k+1}N)\cap (K')^c\\
    % &= \supp f \cap \oP\backslash \oP\tG_k \cup K' \cap Pw_{k+1}N \cap (K')^c \\
    % &= \supp \cap \oP\backslash \oP\tG_k.
    % \end{align*}
    
    This emptyness implies $\supp f \subset \oP\backslash \oP\tG_k$, so the kernel is contained in the image. Conversely if $f \in I_{\le k}(\pi)$ then $\supp f \subset \tG_k$ is a compact open subset, and thus the complement $(\supp f)^{c}$ is also compact open. In particular, there exists an open neighborhood of $\oP w_{k+1} N$ on which $f = 0$, and thus the image is contained in the kernel, proving exactness.
    
    Clearly these maps are continuous. These maps are all of compact type spaces, so open mapping theorem implies these maps are strict.
\end{proof}

We want to make sense of these definitions on the level of complexes. While we defined $I_{\le k}, I_w$ for representations of $\oP$, for complexes it is advantageous to work with compact groups. Thus, let $\tG_0:= \tG(\mO_{F,p})$ as well as $B_0 = B\cap \tG_0 = T_0N_0$, $\oP_0 = \oP \cap \tG_0 = G_0\ol{U}_0$. Now suppose $\pi$ is a locally analytic representation of $\oP_0$. Note since $\oP\backslash \tG \simeq \oP_0\backslash \tG_0$, for any $U \subset \oP\backslash \tG$, we may define $I_{\le k}\pi:= (\Ind_{\oP_0}^{\tG_0}\pi)(\tG_{\le k})$, and similarly $I_w\pi$. Note if $\pi$ is in fact a locally analytic representation of $\oP$, then there is a $\tG_0$-equivariant isomorphism $\Ind_{\oP}^{\tG}\pi \simeq \Ind_{\oP_0}^{\tG_0}\pi$ which induces isomorphisms on $I_{\le k}, I_w$, so these definitions are unambiguous. Now suppose $\Pi$ is a complex of injective admissible Banach representations of $\oP_0$. We can form complexes $I_{\le k}\Pi^{\la}, I_w\Pi^{\la}$. Then we have the following:
\begin{lemma}
\label{lemma: CompatibilityofGextensionsForBruhat}
    If $(\wt{\Pi}, i, p)$ is a $\oP$-extension of $\Pi$, then each functor $\mathsf{F} = I_{\le k}, I_w$ induces $B$-extensions $\begin{tikzcd}\mathsf{F}\Pi^{\la} \ar[r, "i"]&
    \ar[l, "p", shift left]\mathsf{F}\wt{\Pi}^{\la}\end{tikzcd},$ i.e. homotopy equivalences to a complex of compact type spaces with a locally analytic action of $B$. Moreover, these homotopy equivalences fit into a diagram 
    \begin{equation}
    \label{eq: CompatibleHomotopiesForBruhat}
    \begin{tikzcd}
        0 \ar[r]& I_{\le k}\wt{\Pi}^{\la} \ar[r, "\wt{f}"]\ar[d, shift left, "{\wt{p}}"] & I_{\le k + 1}\wt{\Pi}^{\la} \ar[r,"\wt{g}"]\ar[d, shift left, "{\wt{p}}"]& I_{w_{k+1}}\wt{\Pi}^{\la}\ar[r]\ar[d, shift left, "{\wt{p}}"] & 0 \\
        0 \ar[r]& I_{\le k}\Pi^{\la} \ar[r, "f"]\ar[u, shift left, "{\wt{i}}"] & I_{\le k + 1}\Pi^{\la} \ar[r, "g"]\ar[u, shift left, "{\wt{i}}"]& I_{w_{k+1}}\Pi^{\la}\ar[r]\ar[u, shift left, "{\wt{i}}"] & 0 
    \end{tikzcd}
    \end{equation}
    where we have $\wt{i} \circ f = \wt{f} \circ \wt{i}$, $\wt{p} \circ \wt{f} = f \circ \wt{p}$, $\wt{i} \circ g = \wt{g} \circ \wt{i}$, $\wt{p} \circ \wt{g} = g \circ \wt{p}$, respectively.
\end{lemma}
\begin{proof}
This result more or less follows from \cite[Proposition 3.13]{FuDerived}.

Let $\mathsf{F}$ be any of the functors $I_{\le k}, I_w$. Each of these functors are built from spaces of the form $I_K\pi$ for $K \subseteq \oP_0\backslash \wt{G}_0 = \oP\backslash \wt{G}$ open (compact or not).  Then from \cite[Proposition 3.13]{FuDerived} we have a $\wt{G}$-extension 
\[
\begin{tikzcd}
\Ind_{\oP_0}^{\wt{G}_0}\Pi^{\la}\ar[r, "\wt{i}", shift left]& \Ind_{\oP}^{\tG}\wt{\Pi}^{\la}\ar[l, "{\wt{p}}", shift left].
\end{tikzcd}
\]
where concretely, if $f \in\Ind_{P_0}^{\wt{G}_0}\Pi^{\la}$ then $\wt{i}(f)=: f'$ is the unique function such that $f'(g) = i(f(g))$ for $i: \Pi \to \wt{\Pi}$ and $g \in \tG_0$, and $\wt{p}(f') =: f$ satisfies $f(g) = p(f'(g))$ for $p: C' \to C$ and any $g \in \tG_0$. We can now contemplate support conditions in this $\tG$-extension. If $\supp f \subset K \subset \oP_0\backslash \tG_0$ then clearly we have $\supp \wt{i}(f)\subseteq \ol{P}\backslash (\ol{P}\cdot K)\subseteq \oP\backslash \tG$. But we then simply note that under the identification $\ol{P}_0\backslash \tG_0 \simeq \ol{P}\backslash \tG$, we have $K \simeq \ol{P}\backslash(\ol{P}\cdot K)$, so $\wt{i}(I_K\Pi) \subset I_K\wt{\Pi}$. Likewise, if $f' \in \Ind_{\oP}^{\tG}C_{\bullet}'$ has support $K\subseteq \oP\backslash \tG$ then $\wt{p}(f')$ has support contained in $\oP\backslash \oP_0 \cdot K \cap \tG_0 \simeq \oP\backslash \oP\cdot K$ when considered as a subset of the flag variety $\oP_0\backslash \tG_0$. Thus if $U\subseteq \oP\backslash \tG$ is an open invariant under the right action by the Borel $B$, then the $\tG$-extension on the parabolic induction induces $B_0$-equivariant homotopy equivalences
\[
\begin{tikzcd}
\Ind_{\oP_0}^{\wt{G}_0}\Pi^{\la}(U)\ar[r, "\wt{i}", shift left]& \Ind_{\oP}^{\tG}\wt{\Pi}^{\la}(U)\ar[l, "{\wt{p}}", shift left].
\end{tikzcd}
\]
This handles the case of $I_{\le k}$. For the $I_w$ recalling $I_w\pi = \varinjlim_{\ul{x}}I_{U_{w, \ul{x}}}\pi$, the above maps are not necessarily $B$-equivariant for a given $U_{w, \ul{x}}$, but we still have well-defined homotopy equivalences. Then we get an induced homotopy equivalence on the colimits, which a postiori are $B_0$-equivariant, and thus produces the relevant ``$B$-extension.''

The existence and commutativity of the diagram follows directly from the fact that all of the homotopies are induced from the single choice of homotopy equivalence between $\Ind_{\oP_0}^{\tG_0}\Pi^{\la}$ and $\Ind_{\oP}^{\tG}\Pi^{\la}$.
\end{proof}

Crucially, \cref{lemma: CompatibilityofGextensionsForBruhat} allows us to define derived Jacquet functors $J_B(I_{\le k}\Pi^{\la})$ and $J_B(I_{w}\Pi^{\la})$. Moreover, the commutativity of the squares in \eqref{eq: CompatibleHomotopiesForBruhat} \textit{on the nose} (not just up to homotopy!) constructs for us natural maps $J_B(I_{\le k}\Pi^{\la}) \to J_B(I_{\le k+1}\Pi^{\la}) \to J_B(I_{w}\Pi^{\la})$. We would like to say more about these maps (and their effect on cohomology). These features are best understood by first passing to $\mf{n}$-invariants where $\mf{n}:= \Lie N$ is the Lie algebra of the unipotent radical $N\subset B.$

% Let's at least examine the one dimensional stratum. In that case, we have $(\Ind_{P}^G)_{e}$. Any open neighborhood is given by $Pw_0U_{\ul{x}}w_0$. For (say) polynomial functions supported on 
% \[
% Pw_0U_{\ul{x}}w_0
% \]
% the action of $N$.

The key is the following lemma:

\begin{lemma}
\label{lemma: LieAlgebraCohomologyDirectSummand}
    For any two compact open subsets $U \subset V\subset \oP\backslash \tG$ and any Lie subalgebra $\mf{h}\subset \mf{g}$ we have that 
    \[
    \HH^i(\mf{h}, I_U(\pi)) \hookrightarrow \HH^i(\mf{h}, I_V(\pi))
    \]
    is an injection.
\end{lemma}
\begin{proof}
    Note that if $U \subset \oP \backslash \tG$ is compact open then the complement $U^c$ is also compact and open, and thus we can write any function $f = f|_{U} + f|_{U^c}$ uniquely as a sum of two functions each support on $U$ and $U^c$, respectively. The Lie algebra action does not affect the support, so we have a $\mf{g}$-equivariant decomposition $I_V\pi \simeq I_{U}\pi\oplus I_{V\setminus U}\pi$. Passing to $\mf{h}$-cohomology gives the desired result.
\end{proof}

\begin{corollary}
\label{corollary: IwInjectivestoAcyclics}
    Let $\pi$ be an injective admissible Banach $\oP_0$-representation. Then for all $k\ge 0$ and all $w\in W_{\oP}^r$ we have $I_{\le k}\pi^{\la}$ and $I_w\pi^{\la}$ are $R\Gamma_{\mf{h}}$-acyclic for any Lie subalgebra $\mf{h}\subset \mf{g}$.
\end{corollary}

\begin{proof}
    Since $\pi$ is injective, $\pi$ is a direct summand of $\mc{C}(\oP_0, L)^{\oplus d}$ for some $d \ge 0$. Thus $\on{ct}-\Ind_{\overline{P}_0}^{\tG_0}\pi$ is a direct summand of $\mc{C}(\tG_0, L)^{\oplus d}$, and so is also injective admissible. Then recall by \cite[Lemma 3.8]{FuDerived} functor $(-)^{\rm{la}}$ sends injectives to $R\Gamma_{\mf{h}}$-acyclic objects. In particular, $\HH^i(\mf{n}, \Ind_{\oP_0}^{\tG_0}\pi^{\la}) =0$ for all $i \ge 1$. But since taking Lie algebra cohomology commutes with colimits in the coefficients, for both $I_{\le k}$ and $I_w$ it suffices to show $I_U(\pi)$ is $R\Gamma_{\mf{h}}$-acyclic for any $U\subset \oP_0\backslash \tG_0$ compact open. But this case then follows from \cref{lemma: LieAlgebraCohomologyDirectSummand}. 
\end{proof}

The following obvious lemma connects $\mf{n}$-cohomology to an object closer to the Jacquet functor.

\begin{lemma}
\label{lemma: GroupVsLiealgebraCohomology}
    For $V$ a complex of locally analytic representations of $B_0$ with a $B$-extension, we have a $T^+$-equivariant topological isomorphism of topological vector spaces
    \[
    \HH^i( V^{N_0})
    \simeq 
        \HH^i( V^{\mf{n}})^{N_0}
    \]
\end{lemma}
\begin{proof}
    Since $N_0$-invariants is abstractly exact on smooth $N_0$-representations over characteristic $0$ fields, the natural map 
    $\HH^i( V^{N_0})
    \to  \HH^i( V^{\mf{n}})^{N_0}$ is a $T^+$-equivariant continuous bijection. Moreover, since the maps $V^{i, N_0} \hookrightarrow V^{i, \mf{n}}$ are in fact closed embeddings, by the definition of the quotient topology, we have that the image of a closed set is automatically closed, and thus the inverse is also continuous.
\end{proof}

\begin{lemma}
\label{lemma: BruhatSES}
    % Let $\pi \in \Rep_{la.c}(\oP)$.  Then for all $i \ge 0$ we have the strict exact sequence: 
    % \[
    % 0 \to \HH^i(\mf{n}, V_k) \to \HH^i(\mf{n}, V_{k+1}) \to \HH^i(\mf{n}, V_{w_{k+1}}) \to 0.
    % \]
    % More generally, 
    Let $\Pi$ be a bounded complex of injective admissible $\oP_0$-representations, equipped with a $\oP$-extension $(\wt{\Pi}^{\bullet}, i, p)$. Then $J_B(I_{\le k}\Pi^{\la}), J_B(I_{w}\Pi^{\la})$ are complexes of essentially admissible representations of $Y = \ip{z^{\pm 1}, T_0}$, and the cohomology groups $\HH^{\ast}(J_B(I_{\le k}\Pi^{\la})), \HH^{\ast}(J_B(I_{w}\Pi^{\la}))$ are essentially admissible representations of $T(F_p^+)$ which are independent of the choice of $\oP$-extension. Moreover, for all $m$ we have a short exact sequence of essentially admissible $T(F_p^+)$-representations
    \[
    0 \to \HH^m(J_{B}(I_{\le k}\Pi^{\la})) \to \HH^m(J_B(I_{\le k+1}\Pi^{\la})) \to \HH^m(J_B(I_{w_{k+1}}\Pi^{\la})) \to 0.
    \]
\end{lemma}
\begin{proof}
From \eqref{eq: InductiveBruhatSequence} we have a short exact sequence of complexes,
    \begin{equation}
    \label{eq: DistTriBruhat}
  0 \to   I_{\le k}\Pi^{\la}\to I_{\le k + 1}\Pi^{\la} \to I_{w_{k+1}}\Pi^{\la} \to 0
    \end{equation}
    which by \eqref{eq: CompatibleHomotopiesForBruhat} respects the $B$-extensions on each term, and thus is equivariant for the action of $\wt{z}:= \wt{p}\circ z \circ \wt{i}$.
    
    Since each of these functors send injectives to $R\Gamma_{\mf{n}}$-acyclic objects, exactness of $N_0$-invariants on smooth representations gives in each degree short exact sequences
    \begin{equation}
    \label{eq: BruhatSESN0invariants}
    0 \to (I_{\le k}\Pi^{\la})^{N_0} \to (I_{\le k + 1}\Pi^{\la})^{N_0} \to (I_{w_{k+1}}\Pi^{\la})^{N_0} \to 0
    \end{equation}
    We can then take the finite slope part of each term with respect to $\wt{z}$, which is an exact operation by \cite[Theorem 1.1]{FuDerived}, so we get a short exact sequence of complexes of $Y$-representations
    \[
    0 \to J_{B}(I_{\le k}\Pi^{\la})\to J_B(I_{\le k + 1}\Pi^{\la}) \to J_B(I_{w_{k+1}}\Pi^{\la}) \to 0.
    \]
    Since we already know that $J_B(\Ind_{\oP_0}^{\tG_0}\Pi^{\la})$ is essentially admissible over $Y$, all the other $J_B(I_{\le k}\Pi^{\la})$ and $J_B(I_{w_{k+1}}\Pi^{\la})$ are as well. The statement about cohomology and independence of choice follows from same proof of \cite[Theorem 4.10]{FuDerived}. For the last statement, by passing to the long exact sequence on cohomology it suffices to show $\HH^m(J_B(I_{\le k}\Pi^{\la})) \to \HH^m(J_B(I_{\le k+1}\Pi^{\la}))$ is injective for all $m \ge 0$. 
    
    Choosing a representative $f \in I_{\le k}\Pi^{\la, m}$, we contemplate its support $\on{supp} f\subseteq \tG_k \subset \overline{P}_0\backslash \tG_0$, which is by definition a compact open subset. Now consider the map $\overline{P}_0 \backslash \tG_0 \times N_0 \to \overline{P}_0\backslash \tG_0$ sending $(\overline{P}_0r,n) \mapsto \overline{P}_0rn$. This map is visibly continuous and thus sends compact sets to compact sets.
    
    In particular, $(\on{supp} f)N_0$ is a compact open subset of $\overline{P}_0\backslash \tG_0$, and moreover since $\tG_k$ is left $P$-invariant and right $B$-invariant, we have $U_f:= (\on{supp} f) N_0\subset \tG_k$. In particular, this implies \[
    I_{\le k}\Pi^{\la,m} = \varinjlim_{\substack{U\subset \tG_{k} \\\text{ compact open}}} I_{U}(\Pi^{\la,m}) = \varinjlim_{\substack{UN_0\subset \tG_k\\ \text{compact open}}}I_{UN_0}(\Pi^{\la,m}).
    \]
    Moreover for any such $UN_0$ an open compact which is left $\overline{P}$-invariant and right $N_0$-invariant, we have the indicator function $\mathbf{1}_U: G \to L$ is locally constant (and thus locally analytic). Thus we have a decomposition as $N_0$-representations $I_{\le k+1}\Pi^{\la,m} \simeq I_U\Pi^{\la, m} \oplus I_{\tG_k- U}\Pi^{\la, m}$. Then passing to the colimit above yields a continuous injection $\HH^i((I_{\le k}\Pi^{\la})^{N_0})\hookrightarrow \HH^i((I_{\le k+1}\Pi^{\la})^{N_0})$. Moreover, the action of the operator $U_z$ for $z \in T^{++}$ \textit{shrinks} the support, and thus $(I_{U}\Pi^{\la, m})^{N_0}$ is stable under the $U_z$-operator. Since $(I_{\le k}\Pi^{\la,m})^{N_0} \hookrightarrow (I_{\le k+1}\Pi^{\la, m})^{N_0}$ is a closed embedding of spaces of compact type, the induced injection on cohomology $\HH^i((I_{\le k}\Pi^{\la})^{N_0})\hookrightarrow \HH^i((I_{\le k+1}\Pi^{\la})^{N_0})$ is continuous, closed, and equivariant for the action of the operator $U_z$! Thus, we can pass to a Hausdorff quotients and take the finite slope parts (both which will preserve this continuous closed injection) to get
    \[
    \HH^m(J_B(I_{\le k}\Pi^{\la})) \hookrightarrow\HH^m(J_B(I_{\le k+1}\Pi^{\la})),
    \]
    yielding the desired result.
\end{proof}

\section{Degree shifting}
\label{Section4: DegreeShifting}

For this section, we specialize to the setup $\tG = \tG(F_p^+)\simeq \GL_{2n}(F_p)$, $G = G(F_p^+) \simeq \GL_n(F_p)\times \GL_n(F_p)$ and $G \subset \oP$ the (opposite) Siegel parabolic. 
Then we let $\Pi = \Pi_{\oP}(\wt{K}^p)_{\wt{\mf{m}}}$ be a bounded complex of injective admissible $\oP_0$-representations equipped with $\oP$-extension, computing the completed cohomology of $\oP$ localized at $\wt{\mf{m}} = S^{\ast}(\mf{m})$. Moreover, $\Pi$ has a $\oP$-equivariant quasi-isomorphism from $\on{Inf}_{G_0}^{\oP_0}\Pi_G(K^p)_{\mf{m}} \to \Pi$, where $\Pi_G(K^p)$ is a bounded complex of injective admissible $G_0$-representations, and this quasi-isomorphism is compatible with homotopies from complexes with actions of $\oP$.

In light of \cref{lemma: BruhatSES}, the goal now is to isolate out a subquotient of $\HH^i(J_B(I_w\Pi^{\la}))$, which resembles a ($w$-twisted) derived Jacquet module of $\Pi_G(K^p)^{\la}$ in a different degree. The main result of this section is \cref{theorem: JacquetSubQuotient}, from which we get consequences about Hecke algebras, and thus eigenvarieties.

\subsection{Lie algebra cohomology}

We first perform many arguments on the level of Lie algebra cohomology. Recall by \cref{corollary: IwInjectivestoAcyclics} we have for any $\tG_0$-representation $V$ such that $V|_H \simeq \mc{C}^{\la}(H,L)^{\oplus m_i}$ for $H \subset \wt{G}_0$ a compact open subgroup that $I_wV$ is $R\Gamma_{\mf{n}}$-acyclic, so the complex 

\[
(I_w\Pi^{\rm{la}})^{\mf{n}}  =  R\Gamma_{\mf{n}}(I_w\Pi^{\la}),
\]
equals the derived $\mf{n}$-invariants,
at least in the derived category of abstract $L[T_0]$-modules. Moreover, the auxiliary choice of $\oP$-extension on $\Pi$ equips these cohomology groups with an action of the entire torus $T$. We equip the cohomology groups $\HH^{\ast}((I_w\Pi^{\la})^{\mf{n}})$ with the (sub)quotient topologies induced from the modules $(I_w\Pi^{\la})^{\mf{n}}$. 

In particular, the cohomology groups $\HH^{\ast}((I_w\Pi^{\la})^{\mf{n}})$ are independent (even topologically) of the choice of representing complex $\Pi$ of injective admissible Banach $\oP_0$-representations by \cref{corollary: IwInjectivestoAcyclics}. Although we will manage to avoid working systematically with derived categories of topological modules, it will be crucial for us to compute this Lie algebra cohomology in another way, using Chevalley--Eilenberg complexes.

Recall that given a Lie algebra $\mf{g}$ over $L$ and an $L[\mf{g}]$-module $M$, that the associated \textit{Chevalley--Eilenberg complex}
\[
\on{CE}_{\mf{g}}(M):= \{M \to \bigwedge^{1}\mf{g}^{\ast} \otimes_L M \to \bigwedge^2\mf{g}^{\ast}\otimes_L M \to \dots  \}
\]
is a finite complex computing $\HH^{\bullet}(\mf{g}, M)$. Recalling $\bigwedge^{k}\mf{g}^{\ast} \otimes_LM\simeq \Hom_L(\bigwedge^k\mf{g}, M)$ can be thought of as alternating multilinear functions on $\mf{g}$, the differentials $d_{\on{CE}}^k:\bigwedge^{k}\mf{g}^{\ast} \otimes_LM \to \bigwedge^{k+1}\mf{g}^{\ast} \otimes_LM$ are given by the formula \[
(d_{\on{CE}}^kf)(X_1,\dots, X_{k+1}):= \sum_{i = 1}^{k+1}X_i\cdot f(\dots, \widehat{X_i},\dots) + \sum_{i,j}(-1)^{i+j}[X_i,X_j]\cdot f([X_i,X_j], X_1,\dots , \wh{X_i}, \dots \wh{X_j},\dots).
\]Note that this complex essentially arises from a version of the Koszul resolution of the trivial $\mf{g}$-module $\mathbf{1}$, and applying $\Hom_{\mf{g}}(- , M)$ to this resolution.

In particular for \textit{any} complex of $\mf{g}$-modules $C^{\bullet}$, the total complex of the double complex $\on{CE}_{\mf{g}}^{\bullet_1}(C^{\bullet_2}):= \bigwedge^{\bullet_1}\mf{g}^{\ast} \otimes_L C^{\bullet_2}$ computes the derived functors $R\Gamma_{\mf{g}}(C^{\bullet})$. We may also equip such cohomology groups with the subquotient topologies arising from this total complex. In the case of a complex of $R\Gamma_{\mf{g}}$-acyclic objects, we note the following lemma about topological cohomology:

\begin{lemma}
Let $V^{\bullet}$ be either: a) a complex of admissible locally analytic $\tG_0$-representations of the form $\mc{C}^{\la}(H, L)^{\oplus m_{\bullet}}$ for some compact open subgroup $H\subset \wt{G}_0$, OR  b) a complex of the form $V^{\bullet} \simeq I_{\le k}C^{\bullet,\la}, I_{w}C^{\bullet,\la}$ for $C^{\bullet}$ a complex of admissible $\oP_0$-representations of the form $\mc{C}^{\la}(\ol{P}_H, L)^{\oplus m_\bullet}$ for $\ol{P}_H\subset \ol{P}_0$ a compact open subgroup. Then the natural map of complexes
\[
V^{\mf{n}} \to \on{Tot}(\bigwedge^{\,\,\,\bullet_1}\mf{n}^{\ast}\otimes_LV^{\bullet_2})
\]
coming from the inclusion(s) $V^{\mf{n}} \hookrightarrow V$ induces an isomorphism
\[
\HH^r(V^{\mf{n}}) \simeq \HH^r(\on{Tot}(\bigwedge^{\,\,\,\bullet_1}\mf{n}^{\ast}\otimes_L V^{\bullet_2})),
\]
of topological vector spaces with continuous $T_0$-action,
where we endow both sides with the induced subquotient topologies.

Moreover, if $V^{\bullet}$ is equipped with a $\tG$-extension (or $C^{\bullet}$-equipped with a $\oP$-extension), then these isomorphisms are in fact $T$-equivariant.
\end{lemma}
\begin{proof}
In either case $V$ is a complex of $R\Gamma_{\mf{n}}$-acyclic modules by \cref{corollary: IwInjectivestoAcyclics}, so we have an isomorphism of abstract vector spaces as in the lemma statement induced by the described map. It remains to verify topological properties. Now let $d_{V^{\mf{n}}}^r, d_{\on{CE}}^r$ denote the differentials in the $r$th entry of $V^{\mf{n}}$ and $\on{Tot}(\bigwedge^{\bullet_1}\mf{n}^{\ast} \otimes V^{\bullet_2,\mf{n}})$, respectively. These differentials are both naturally continuous and $T_0$-linear. Since the topologies on both cohomology groups are induced subquotient topologies, the bijection $\HH^r(V^{\mf{n}}) \to \HH^r(\on{Tot}(\bigwedge^{\bullet_1}\mf{n}^{\ast}\otimes V^{\bullet_2}))$ is continuous and closed if and only if the maps $\ker(d_{\Pi^{\mf{n}}}^r) \to \ker(d_{\on{CE}}^r)$ are continuous and closed, which is in turn true if and only if the composite map $\ker(d_{V^{\mf{n}}}^r)  \to V^{r,\mf{n}} \to  \bigoplus_{i + j = r}\bigwedge^i\mf{n}^{\ast}\otimes_L V^{j, \mf{n}}$ is continuous and closed. But since invariance is a closed condition and the map $d_{V^{\mf{n}}}^r$ is a continuous map of spaces of compact type, this composite map remains a closed embedding. In particular, it is continuous and also maps closed subsets to closed subsets. Thus the induced bijective map on cohomology is a continuous bijection mapping closed sets to closed sets, and thus is also open. 

The final statement regarding $\tG$-extensions or $\oP$-extensions is clear upon noting that the map $V^{n} \to V$ preserves the action of $T$ (up to homotopy).

\end{proof}

\subsection{Constant term maps}

Now for $w \in W^{\oP}$, let $N_w:= wN(\mc{O}_{F^+,p})w^{-1}\cap \oP(F_p^+)$ and $\mf{n}_w:= \Lie N_w$.
Let  $(A^{\bullet}, i, p)$ be a choice of $\oP$-extension for $\Pi$. Then we may consider the complex $\Pi^{\la,\mf{n}_{w}}$, which using the $\oP$-extension structure and some $z \in T^{++}$ carries an action of the monoid $Y^+= \ip{T_0, z}$ given by the formula
\begin{align}
\label{eq: wTwistedT^+Action}
t\cdot v:= \sum_{n \in N_{w}/t^wN_w(t^w)^{-1}}nt^wv, 
\end{align}
where we notate $t^w:= wtw^{-1}$. As per usual, this action depends on the choice of $\oP$-extension, but action on the cohomology groups does not. Note also locally analytic vectors of an injective Banach object is $\mf{n}_w$-acyclic by \cite[Lemma 3.8]{FuDerived}, so this definition is a reasonable analogue of a topological version of $R\Gamma_{\mf{n}_w}(\Pi^{\bullet, \la})$.

% and $\tau_{w}^{-1}$ denotes the the action is given by $\tau_{w}(n)v = wnw^{-1}v$.

For our purposes we use these complexes to define a natural $Y^+$-equivariant maps of complexes
\[
\begin{tikzcd}
\label{eq: ConstantTermsofInductions}
(I_w\Pi^{\bullet, \la})^{\mf{n}}\ar[d, shift left, "\wt{i}"] \ar[r, "{\on{CT}_w}"] &\Pi^{\bullet, \la, \mf{n}_w}\ar[d, shift left, "i"]\\
(I_wA^{\bullet, \la})^{\mf{n}}\ar[u, "\wt{p}"] \ar[r, "{\on{CT}_w}"]& A^{\bullet, \la, \mf{n}_w}\ar[u, "p"]
\end{tikzcd}, \quad 
\begin{tikzcd}
(I_w\Pi^{\bullet, \la})^{N_0}\ar[d, shift left, "\wt{i}"] \ar[r, "{\on{CT}_w}"] &\Pi^{\bullet, \la, N_w}\ar[d, shift left, "i"]\\
(I_wA^{\bullet, \la})^{N_0}\ar[u, "\wt{p}"] \ar[r, "{\on{CT}_w}"]& A^{\bullet, \la, N_w}\ar[u, "p"]
\end{tikzcd}
\]
induced by $f \mapsto f(w)$. These maps are compatible with the homotopy equivalences $(i, p)$ \textit{on the nose}, since we are simply fixing one choice. The fact that the top horizontal maps intertwine the Atkin--Lehner-like trace operators is proven in \cite[Lemma 5.3.5]{10AuthorPaper}.

\begin{remark}
Here lies an important distinction between smooth ($p$-power torsion) and the locally analytic cases: there is no a priori reason that the maps $\on{CT}_{w}$ should be surjective (unlike \cite[Lemma 5.3.5]{10AuthorPaper}). The simple issue is that invariant germs of locally analytic inductions may acquire additional relations in their constant terms. Nevertheless, we will salvage the situation by sectioning off a certain ``highest weight vector,'' and showing surjectivity to this highest weight.
\end{remark}

We now want to use $\Inf_{G_0}^{\oP_0} \Pi_G$ to compute the $Y^{+}$-representations $\HH^{\ast}((\Pi^{\la})^{\mf{n}_{w}})$. The Hochschild--Serre spectral sequence and a central action gives us a good understanding of these groups. For any $T_n(F_p)$-module $V$ let $\tau_wV$ denote the same underlying space $V$ but with twisted action $\tau_{w}(t)(v) := t^{w^{-1}}\cdot v := (w^{-1}tw)\cdot v$.

\begin{lemma}
\label{lemma: SpecSeqDegen}
We have an abstract isomorphism of $T$-representations
\[
\ol{E}:\HH^m(\Pi^{\la, \mf{n}_w}) \simeq \bigoplus_{j = 0}^{\ell(w)}\HH^{m-j}(R\Gamma_{\mf{n_n}}(\tau_{w}^{-1}\Pi_G^{\la} \otimes \tau_{w}^{-1}\HH^j(\mf{n}_{w,\overline{U}}, \mathbf{1}))).
\]
Moreover, the projection on the final factor induces a continuous $(T^+,N_{w})$-equivariant surjection
\[
\ol{E}_{\ell(w)}: \HH^m(\Pi^{\la, \mf{n}_w}) \twoheadrightarrow \tau_{w}^{-1}\HH^{m-\ell(w)}(\Pi_G^{\la,\mf{n}_n} \otimes_L \HH^{\ell(w)}(\mf{n}_{w,\overline{U}}, \mathbf{1})).
\]

\end{lemma}

\begin{proof}
Let $N_w\rtimes_w T^+$ be the the monoid $N_w\times T^+$ equipped with multiplication $(t^wn(t^w)^{-1},1)(1,t) = (1,t)(n,1)$.
To show the decomposition into abstract $L[N_w\rtimes_{w}T^+]$-modules we use the sequence
    \[
    0 \to \mf{n}_{w,\overline{U}} \to \mf{n}_w \to \mf{n}_n  \to 0.
    \]
    Note that $\mf{n}_{n}$ is a quotient of $\mf{n}_{w}$ precisely because $w \in W^{\oP}$ was chosen to be a minimal length representative.
The Hochschild--Serre spectral sequence for Lie algebra cohomology (using that the modules $\Pi^{\rm{la}}$ are acyclic for all the functors involved),
    we have as abstract $L[N_w\rtimes_w Y^+]$-modules that 
    \[
    R\Gamma_{\mf{n}_w}\Pi^{\rm{la}} \simeq \tau_w^{-1}R\Gamma_{\mf{n}_n}(R\Gamma_{\mf{n}_{w,\ol{U}}}\Pi^{\rm{la}}) \simeq \tau_w^{-1}R\Gamma_{\mf{n}_n}(\Pi^{\rm{la}}\otimes R\Gamma_{\mf{n}_{w,\ol{U}}}\mathbf{1}),
    \]
    where the last line follows because the complex $\Pi$ is quasi-isomorphic to a complex $\Pi_G$ inflated from an action of the Levi $G_0$. Recall that $G_0 = G(\mc{O}_{F^+,p}) \simeq \GL_n(\mc{O}_{F_{p}})\times \GL_n(\mc{O}_{F_{p}})$ has centre $\mc{O}_{F}\otimes_{\mc{O}_{F^+}}\mc{O}_{F^+,p}\simeq \mc{O}_{F_{p}}^{\times}\times \mc{O}_{F_{p}}^{\times}$, so we may choose the scalars on each diagonal $n\times n$ component separately. Consider the non-torsion element $1 +p^2 \in \mc{O}_{F_{\nu}}^{\times}$ for all $\nu \mid p$ and the corresponding central element $z_p = \begin{pmatrix}
        ((1 + p^2))_{\nu}\cdot I_n & 0\\
        0                  & I_n
    \end{pmatrix}\in G_0$.
    This element acts on the cohomology groups $\HH^i(\mf{n}_{w,\overline{U}}, \mathbf{1})$ via conjugation by $z_p$ on $N_{w,\overline{U}}$. Viewing $N_{w,\overline{U}}\simeq \Z_p^{\ell(w)}$ as a $\Z_p$-module and considering the Chevalley--Eilenberg complex $\bigwedge^i\mf{n}_{w,\overline{U}}^{\ast}$, $z_p$ acts on $\HH^i(\mf{n}_{w, \overline{U}}, \mathbf{1})$ via multiplication by $(1+p^2)^{i}$. Since these groups are in characteristic zero, all these scalars are distinct, and there is an idempotent $E_j:= \prod_j\frac{z_p-(1+p^2)^j}{(1+p^2)^i-(1+p^2)^j}$ on $R\Gamma_{\mf{n}_{w,\overline{U}}}(\mathbf{1})$ which induces the identity map on $R^i\Gamma_{\mf{n}_{w, \overline{U}}}(\mathbf{1})$ and zero on all other degrees.
    In particular, these maps yield an explicit quasi-isomorphism
    \begin{equation}
    \label{eq: IdempotentSplitting}
    \on{CE}(\mf{n}_{w, \ol{U}}, \mathbf{1}) \xrightarrow{E}\bigoplus \HH^i(\mf{n}_{w, \ol{U}}, \mathbf{1})[-i],
    \end{equation}
    where $E:= \sum_{j = 0}^{\ell(w)}E_j$. This quasi-isomorphism is naturally equivariant for the action of $T$ and $N_w$.
    Now recall that the $\oP_0$-equivariant quasi-isomorphism $\Inf_{G_0}^{\oP_0}\Pi_G \to \Pi$ has a continuous $G_0$-equivariant homotopy inverse $\Pi \to \Inf_{G_0}^{\oP_0}\Pi_G$, which is also equivariant for the $T$-action up to homotopy.
    We can pass to Lie algebra cohomology by considering term-by-term Chevalley--Eilenberg resolutions. In particular, this homotopy equivalence induces another homotopy equivalence 
    \[
    \on{Tot}^m(\bigwedge^{\bullet_1}\mf{n}_{w, \overline{U}}^{\ast}\otimes(\Inf_{G_0}^{\oP_0}\Pi_G^{\la})^{\bullet_2}) \to \on{Tot}^m(\bigwedge^{\bullet_1}\mf{n}_{w, \overline{U}}^{\ast}\otimes\Pi^{\la,\bullet_2}).
    \]
    Note however that the action of $\mf{n}_{w, \overline{U}}$ on $\Pi_G$ is trivial, so the first complex can be rewritten as $\on{CE}(\mf{n}_{w, \ol{U}},\mathbf{1})\otimes \Pi_G^{\la}$, with $T_0$-action given diagonally. Then using \eqref{eq: IdempotentSplitting}, the sum of idempotents $E$ induces an explicit quasi-isomorphism
    \[
    \on{Tot}^m(\on{CE}(\mf{n}_{w, \ol{U}},\mathbf{1})\otimes \Pi_G^{\la})\simeq \bigoplus_{q=0}^{m}\Pi_{G}^{\la,q}\otimes \HH^{m-q}(\mf{n}_{w, \overline{U}}, \mathbf{1}).
    \]

    In particular, the complex $\Pi_G^{\la,\bullet}\otimes \HH^{i}(\mf{n}_{w, \overline{U}}, \mathbf{1})[i]$ is a direct summand of $(\on{CE}(\mf{n}_{w, \ol{U}},\mathbf{1})\otimes \Pi_G^{\la})^{n}$.
    Then as abstract modules we have $N_w\rtimes_w T$-equivariant isomorphisms
    \[
    \ol{E}: R^m\Gamma_{\mf{n}_{w}}\Inf_{G_0}^{\oP_0}\Pi_G^{\la} \simeq \bigoplus_{i = 0}^{\ell(w)}\tau_w^{-1}R^{m-i}\Gamma_{\mf{n}_n}(\Pi_G^{\la} \otimes \HH^i(\mf{n}_{w, \overline{U}}, \mathbf{1})),
    \]
    where the twist $\tau_w^{-1}$ arises from the definition of the action in \eqref{eq: wTwistedT^+Action}.

    Now we prove the map $\ol{E}_{\ell(w)}$ is continuous. First we claim that the map $\tau_w^{-1}\HH^i(\on{CE}(\mf{n}_w,  \Inf_{G_0}^{\oP_0}\Pi_{G}^{\la})) \to \HH^i(\on{CE}(\mf{n}_w, \Pi^{\la}))$ is a topological isomorphism. A priori, the issue is that $\Inf_{G_0}^{\oP_0}\Pi_G \to \Pi$, only admits a $G_0$-equivariant (NOT $\oP_0$-equivariant) homotopy inverse, and so a continuous inverse does not immediately exist. However, we do have a topological isomorphism $\HH^i(\Pi^{\la, \mf{n}_w})\simeq \HH^i(\on{CE}(\mf{n}_w, \Pi^{\la}))$ induced by the inclusion $\Pi^{\la, \mf{n}_w}\hookrightarrow \Pi^{\la}$, and using the \textit{continuous} $G_0$-equivariant homotopy inverse, we in fact get a well-defined map of complexes $\Pi^{\la,\mf{n}_w} \to \on{Tot}^{\bullet}(\on{CE}(\mf{n}_w,\Inf_{G_0}^{\oP_0}\Pi_G^{\la}))$, which provides a quasi-inverse to the natural map in the other direction. The induced map on cohomology provides a \textit{continuous} inverse.
    
    Note since the $\Pi_G^{\bullet}$ is a complex of injective admissible representations of $G_0$, all the $\Pi_G^{\bullet, \la}$ are acyclic for $\mf{n}_n$-cohomology. However, it is not clear whether the $\Pi_G^{\la,\bullet}\otimes \HH^j(\mf{n}_{w, \ol{U}}, \mathbf{1})$ are also since the conjugation action of $\mf{n}_n$ on $\mf{n}_{w, \ol{U}}$ is nontrivial. However, by considering top degree forms in $\bigwedge^{\ell(w)}\mf{n}_{w, \overline{U}}^{\ast}$, a straightforward calculation shows that the $N_n$-action on the $1$-dimensional space $\HH^{\ell(w)}(\mf{n}_{w, \overline{U}}, \mathbf{1})$ is in fact trivial, so $\Pi_G^{\la,\bullet}\otimes \HH^{\ell(w)}(\mf{n}_{w, \ol{U}},\mathbf{1})$ is a complex of $R\Gamma_{\mf{n}_n}$-acyclic objects, and thus for all $j$ we have $\HH^j((\Pi_G^{\la} \otimes \HH^{\ell(w)}(\mf{n}_{w, \overline{U}}, \mathbf{1}))^{\mf{n}_n}) \simeq \HH^j(\Pi_G^{\la, \mf{n}_n}\otimes \HH^{\ell(w)}(\mf{n}_{w,\ol{U}},\mathbf{1}))$. 
        
        In particular, by taking $\mf{n}_n$-invariants we get a continuous map of complexes $\Pi^{\la,\mf{n}_{w}} \to (\Pi^{\la} \otimes \HH^{\ell(w)}(\mf{n}_{w, \ol{U}}, 1))^{\mf{n}_n}[-\ell(w)]$ which on the level of cohomology induces a continuous $N_w\rtimes_w T$-equivariant surjection
    \[
    \ol{E}_{\ell(w)}:\HH^m((\Pi^{\la})^{\mf{n}_w}) \to \tau_w^{-1}\HH^{m-\ell(w)}(\Pi_G^{\la,\mf{n}_n})\otimes \tau_w^{-1}\HH^{\ell(w)}(\mf{n}_{w, \ol{U}}, \mathbf{1}),
    \]
    as required.
    \end{proof}

To clean up the target of this surjection $\ol{E}_{\ell(w)}$, we first define a character $\chi_{w}: T(F_p^+) \to L^{\times}$ via the formula
\begin{equation}
\label{eq: DefinitionOfChi_w}
\chi_w(t) = \varepsilon_{\cyc}^{-1} \circ \Art_{F_p} \circ\det{}_{F_p^+}(\on{Ad}(t^w)|_{\mf{n}_{w,\overline{U}}}) = \varepsilon_{\cyc}^{-1}\circ \Art_{F_p}\circ (w^{-1}w_0^{\tG}\rho^{\tG} - w_0^{\tG}\rho^{\tG}),
\end{equation}
where $\varepsilon_{\cyc}^{-1} \circ \Art_{F_p}: F_p^{\times} \to \Z_p^{\times}$ is the map $x \mapsto \prod_{\nu}\varepsilon_{\cyc}^{-1}\circ \Art_{F_{\nu}}(x_{\nu})$.

The reason for introducing this character is that that we have a $T(F_p^+)$-equivariant (respectively $T(F_p^+)^{+}$-equivariant) isomorphism (see \cite[Proposition 3.1.8]{Hau16})
\[
\label{eq:dettopforms}
\tau_w^{-1}\HH^{\ell(w)}(\mf{n}_{w, \overline{U}}, \mathbf{1}) \simeq L(\chi_w),\quad
\tau_w^{-1}\HH^{\ell(w)}(\mf{n}_{w, \overline{U}}, \mathbf{1})^{N_{w, \overline{U}}}  \simeq L(\chi_w)
\]

This isomorphism, combined with \cref{lemma: SpecSeqDegen} and \cref{eq: ConstantTermsofInductions} gives us a composite continuous map
\begin{equation}
\label{eq: DegreeShiftingLie}
\HH^i((I_w\Pi^{\rm{la}})^{\mf{n}}) \to \HH^i(\Pi^{\la, \mf{n}_w}) \twoheadrightarrow \tau_w^{-1}\HH^{i - \ell(w)}(\Pi_G^{\la,\mf{n}_n}) \otimes \chi_{w}
\end{equation}
% such that upon taking Hausdorff quotients we get a surjection in $\Rep_{\rm{la}.c}(T\times (N_0)_{\rm{sm}})$.

We can use the Chevalley--Eilenberg complexes to give a concrete description of this map as follows: let $X_{-\mf{r}_1}^{\vee}, \dots X_{-\mf{r}_{\ell}}^{\vee}$ be a choice of basis for the dual Lie algebra $(w\mf{n}w^{-1} \cap \ol{\mf{u}})^{\ast}$, dual to elements $X_{-\mf{r}_1}, \dots X_{-r_{\ell}}$ generating the Lie algebra, indexed by the (negative) nilpotent root subspaces they generate. Likewise, let $\{Y_j^{\vee}\}_{j = 1}^{n(n-1)[F_p^+: \Q_p]}$ be a chosen basis for $\mf{n}_n^{\vee}$. Lastly, let $\{Z_l\}_{l = 1}^{n^2[F_p^+:\Q_p] - \ell(w)}$ be a choice of basis for $(w\mf{n}w^{-1} \cap \mf{u})^{\ast}$. Note that all together, the $X_{-\mf{r}_i}^{\vee}, Y_j^{\vee}, Z_{l}^{\vee}$ form a basis for the dual Lie algebra $(w\mf{n}w^{-1})^{\ast}$. For convenience, for any subsets of indices $A = \{i_1,\dots, i_{a}\} \subset \{1,\dots, \dim w\mf{n}w^{-1} \cap \ol{\mf{u}} = \ell(w)\}, B = \{j_1, \dots, j_b\}\subset \{1, \dots, \dim \mf{n}_n\}$ and $C =\{k_1,\dots, k_c\}  \subset \{1, \dots, \dim w\mf{n}w^{-1} \cap \mf{u}\}$, we define 
\[
L_{A,B,C}:= Z_C^{\vee}\wedge Y_B^{\vee} \wedge X_A^{\vee}:= Z_{k_1}^{\vee} \wedge \dots \wedge Z_{k_c}^{\vee}\wedge Y_{j_1}^{\vee} \wedge \dots \wedge Y_{j_b}^{\vee}\wedge X_{-\mf{r}_{i_1}}^{\vee} \wedge \dots \wedge X_{-\mf{r}_{i_a}}^{\vee}   \in \bigwedge^{\,\,\,a + b + c}(w\mf{n}w^{-1})^{\ast}.
\]

Using the total complexes from Chevalley--Eilenberg resolutions of each component, the map
\[
\HH^d\left(\bigoplus_{i + j = \ast}\bigwedge \mf{n}^{\ast} \otimes I_w\Pi^{\rm{la}}\right) \to \HH^d\left(\bigoplus_{i + j = \ast}\bigwedge \mf{n}_w^{\ast}\otimes \Pi^{\la}\right)\xrightarrow{\ol{E}_{\ell(w)}} \tau_w^{-1}\HH^{d- \ell(w)}\left(\bigoplus_{i + j = \ast}\bigwedge \mf{n}_n^{\ast} \otimes \Pi_G^{\rm{la}}\right)\otimes \chi_w.
\]
sends a cocycle 
\begin{equation}
\label{eq: ExplicitCocycleFormula}
[\sum_{A,B,C}w^{-1}(L_{A,B,C})\otimes f_{A,B,C}] \mapsto [\sum_{A,B} X_{A}^{\vee}\wedge Y_{B}^{\vee} \otimes f_{A,B,\emptyset}(w)]\mapsto [\sum_B Y_B^{\vee}\otimes (f_{[\ell(w)], B, \emptyset}(w)\otimes X_{[\ell(w)]}^{\vee})) ],
\end{equation}
where $[\ell(w)]:= \{1,\dots, \ell(w)\}$.
This formula for $\ol{E}_{\ell(w)}$ follows directly from its construction in the proof of \cref{lemma: SpecSeqDegen}. In particular, the projection onto this eigenspace kills all terms except those involving the entire basis of $\mf{n}_{w,\ol{U}}$. The fact that such an expression remains a cocycle with respect to this quotient Lie algebra is one consequence of the spectral sequence from the proof of \cref{lemma: SpecSeqDegen} degenerating. The main claim then is as follows:
\begin{proposition}
\label{prop: MainDegreeShiftProp}
The  map
    \[
   \HH^i((I_w\Pi^{\la})^{\mf{n}}) \to \tau_w^{-1}\HH^{i - \ell(w)}(\Pi_G^{\la, \mf{n}_n}) \otimes L(\chi_{w})
    \]
    is a continuous $T$-equivariant surjection. Moreover, the induced map 
    \[
    \HH^i((I_w\Pi^{\la})^{N_0}) \to \tau_w^{-1}\HH^{i - \ell(w)}(\Pi_G^{\la, N_{n,0}})\otimes L(\chi_w)
    \]
    is also continuous, surjective, and equivariant for the action of $T^+$.
\end{proposition}

\begin{proof}
Since the map is by construction continuous it suffices to show the map is surjective as a map of abstract vector spaces.
We proceed by constructing a section $\tau_w^{-1}\HH^{i-\ell(w)}(\Pi_G^{\la, \mf{n}_n})\otimes \chi_w \to \HH^i((I_{w}\Pi^{\la})^{\mf{n}_w})$. 

For clarity's sake, we start with the easiest case of $w= 1$. Then the statement to prove is that the constant term map $\HH^d((\Ind_{\oP_0}^{\wt{G}_0}\Pi^{\la}(\oP_0 U))^{\mf{n}}) \to \HH^d(\Pi_G^{\la, \mf{n_n}})$ induced by $f \mapsto f(e)$ and a homotopy inverse $\Pi \to \Pi_G$ is surjective. Noting $\HH^d(\Pi^{\la, \mf{n_n}})\simeq \HH^d(\Pi_G^{\la, \mf{n}_n})$, it 
 suffices to show $\HH^d((\Ind_{\oP_0}^{\wt{G}_0}\Pi^{\la}(\oP_0 U))^{\mf{n}}) \to \HH^d(\Pi^{\la, \mf{n}_n})$ is surjective. But the map $(\Ind_{\oP_0}^{\wt{G}_0}\Pi^{\la}(\oP_0 U))^{\mf{n}} \to \Pi^{\la, \mf{n}_n}$ already has a natural section given by identifying $(\Ind_{\oP_0}^{\wt{G}_0}\Pi^{\la}(\oP_0 U))^{\mf{n}}\simeq \mc{C}^{\la}(U, \Pi^{\la})^{\mf{n}} \simeq \mc{C}^{\on{sm}}(U, \Pi^{\la})^{\mf{n}_n}$ and $\Pi^{\la, \mf{n}_n}\simeq \mc{C}^{\la}(U, \Pi^{\la})^{U, \mf{n}_n}$, and using the injection of constant functions $\mc{C}^{\la}(U, \Pi^{\la})^{U, \mf{n}_n} \hookrightarrow \mc{C}^{\la}(U, \Pi^{\la})^{\mf{n}}$.

For a general representative $w \in W^{P}$ of minimal length $\ell(w)$, we note that the set $\Phi_w
:= \Phi^+ \cap w(\Phi^{-})$ has length $\ell(w)$. Choosing an appropriate ordering of the elements $\mf{r}_1, \dots, \mf{r}_{\ell(w)} \in \Phi_w$, we have a factorization into (not necessarily simple) reflections $w = s_{\mf{r}_1}\dots s_{\mf{r}_{\ell}}$. Note that since $w$ is minimal length, for all $i$ we have the associated nilpotent element $X_{\mf{r}_i} \in \mf{u}$ lies in the nilpotent radical of the Siegel parabolic. We likewise set $w':= s_{\mf{r}_1}w = s_{\mf{r}_2}\dots s_{\mf{r}_{\ell}}$. Then in fact we have that $w'(\Phi^+) = (w(\Phi^+) \setminus\{-\mf{r}_1\})\cup \{\mf{r}_1\}$. Now let $\mf{n}_{w'}:= w\mf{n}w^{-1} \cap w'\mf{n}(w')^{-1} \subseteq w\mf{n}w^{-1}$ be the subalgebra generated by the roots in $(w(\Phi^+) \setminus\{-\mf{r}_1\})$. Then in fact $\dim \mf{n}_{w'} = \dim \mf{n} - 1$, as if it generated the entire Lie algebra $w\mf{n}w^{-1}$, then the nilpotent Lie algebra $w'\mf{n}(w')^{-1}$ would contain both $X_{\mf{r}_1}, X_{-\mf{r}_1}$, contradicting nilpotence. Moreover, $\mf{n}_{w'}$ is in fact a Lie ideal, since $[X_{-\mf{r}_1}, X_{\mf{s}}]\in X_{\mf{t}}$ for some $\mf{t}\in w(\Phi^+)$ but it clearly cannot be $-\mf{r}_1$. The same reasoning also shows that in fact $X_{\mf{r}_1}$ also normalizes $\mf{n}_{w'}.$ The quotient $w\mf{n}w^{-1}/\mf{n}_{w'} \simeq \ip{X_{-\mf{r}_1}}$ is then a $1$-dimensional nilpotent Lie algebra with representation coming from the negative root $-\mf{r}_1$.

Now first we claim there is a short exact sequences of complexes 
\begin{equation}
\label{eq: InductiveExactSequenceForr_1}
0 \to (I_w\Pi^{\la})^{\mf{n}} \to (I_w\Pi^{\la})^{w^{-1}\mf{n}_{w'}w} \xrightarrow{w^{-1}X_{-\mf{r}_1}w}(I_w\Pi^{\la})^{w^{-1}\mf{n}_{w'}w}  \to 0.
\end{equation}
The only claim to prove is that the final map is surjective, which is equivalent to showing $(I_{w}\Pi^{\la})^{w^{-1}\mf{n}_{w'}w}$ is an $R\Gamma_{\ip{w^{-1}X_{-\mf{r}_1}w}}$-acyclic object, where $\ip{w^{-1}X_{-\mf{r}_1}w}$ is the $1$-dimensional Lie algebra generated by $w^{-1}X_{-\mf{r}_1}w$. But by the Hochschild--Serre spectral sequence, we have an embedding $\HH^1(\ip{w^{-1}X_{-\mf{r}_1}w}, (I_w\Pi^{\la})^{w^{-1}\mf{n}_{w'}w}) \hookrightarrow \HH^1(\mf{n}, I_w\Pi^{\la})$, and $\HH^1(\mf{n}, I_w\Pi^{\la}) = 0$ by \cref{corollary: IwInjectivestoAcyclics}, which gives the claim.

Passing to the long exact sequence associated to \eqref{eq: InductiveExactSequenceForr_1} and passing to cokernels/kernels gives the sequence 
\[
0 \to \HH^{i-1}((I_w\Pi^{\la})^{w^{-1}\mf{n}_{w'}w})_{w^{-1}X_{-\mf{r}_1}w}\otimes (\ip{w^{-1}X_{-\mf{r}_1}w})^{\ast} \to \HH^{i}((I_w\Pi^{\la})^{\mf{n}}) \to \HH^{i}((I_w\Pi^{\la})^{w^{-1}\mf{n}_{w'}w})^{w^{-1}X_{-\mf{r}_1}w} \to 0.
\]
Precomposing with the natural quotient map gives us a map 
$\HH^{i-1}((I_w\Pi^{\la})^{w^{-1}\mf{n}_{w'}w}) \otimes (\ip{w^{-1}X_{-\mf{r}_1}w})^{\ast} \to \HH^i((I_w\Pi^{\la})^{\mf{n}})$. The inclusion $ (I_w\Pi^{\la})^{w^{-1}w'\mf{n}(w')^{-1}w}\hookrightarrow (I_{w}\Pi^{\la})^{w^{-1}\mf{n}_{w'}w}$ then gives a map 
\begin{equation}
\label{eq: FirstConnectingMapr_1}
    c_{\mf{r}_1}: \HH^{i-1}((I_w\Pi^{\la})^{w^{-1}w'\mf{n}(w')^{-1}w})\otimes \ip{X_{-\mf{r}_1}}^{\ast}\to \HH^i((I_w\Pi^{\la})^{\mf{n}}).
\end{equation}

Combining similar sequences for the invariance of $w^{-1}w'\mf{n}(w^{-1}w')^{-1}$, and iterating for all $\mf{r}_i$, we get a natural map 
\[
s: \HH^{i - \ell(w)}((I_w\Pi^{\la})^{w^{-1}\mf{n}w})\otimes \bigotimes_{i}(w^{-1}X_{-\mf{r}_i}w)^{\ast}\xrightarrow{s} \HH^{i}((I_w\Pi^{\la})^{\mf{n}}),
\]
which is defined by a formula $s:= c_{\mf{r}_{\ell(w)}}\circ \dots \circ c_{\mf{r}_1}$.
Now we first note that as a $T$-module, $\bigotimes_{i}(w^{-1}X_{-\mf{r}_i}w)^{\ast}\simeq \HH^{\ell(w)}(w^{-1}\mf{n}_{w,\ol{U}}w,\mathbf{1})\simeq \tau_w^{-1}\HH^{\ell(w)}(\mf{n}_{w,\ol{U}},\mathbf{1})^{N_{w,\ol{U}}}\simeq L(\chi_w)$. Moreover, by the example in the $w = 1$ case we have a natural map $\iota: \Pi^{\la, \mf{n}_n}  \hookrightarrow (I_w\Pi^{\la})^{w^{-1}\mf{n}w}$ given by mapping a vector $v$ to the associated constant function. On cohomology we then get a map
$\iota_{\ast}: \tau_w^{-1}\HH^{i- \ell(w)}(\Pi^{\la, \mf{n}_n}) \otimes L(\chi_w) \to \tau_w^{-1}\HH^{i- \ell(w)}((I_w\Pi^{\la})^{w^{-1}\mf{n} w}) \otimes L(\chi_w)$, which upon composing with $s$ gives a $T$-equivariant endomorphism $\on{CT}_w \circ s \circ \iota_{\ast}: \tau_w^{-1}\HH^{i- \ell(w)}(\Pi^{\la, \mf{n}_n}) \otimes L(\chi_w) \to \tau_w^{-1}\HH^{i- \ell(w)}(\Pi^{\la, \mf{n}_n}) \otimes L(\chi_w)$. 
Note in fact that the map $\tau_w^{-1}\Pi^{\la, \mf{n}_n} \to I_w(\Pi^{\la})^{w^{-1}\mf{n}w}$ factors through $(I_w\Pi^{\la})^{w^{-1}Uw,w^{-1}\mf{n}_nw}$. After taking $N_{n,0}$ invariants we thus also have an injection $\iota_{0}\tau_w^{-1}\Pi^{\la,N_{n,0}} \hookrightarrow (I_w\Pi^{\la})^{w^{-1}N_0w}$. Likewise, taking ($w'$-twisted for varying $w'$) $N_0$-invariants of \eqref{eq: FirstConnectingMapr_1} and its invariants for all $i$ yields a map
\[
s_{0}\circ \iota_{0,\ast}:\HH^{i-\ell(w)}(\tau_w^{-1}\Pi^{\la,N_{n,0}}\otimes_L L(\chi_w)) \to \HH^i((I_w\Pi^{\la})^{N_0}),
\]
such that the endomorphism $\on{CT}_{w}\circ s_0\circ \iota_{\ast, 0}: \HH^{i-\ell(w)}(\tau_w^{-1}\Pi^{\la,N_{n,0}}\otimes_L L(\chi_w)) \to \HH^{i-\ell(w)}(\tau_w^{-1}\Pi^{\la,N_{n,0}}\otimes_L L(\chi_w))$ is merely the map induced from $\on{CT}_w\circ s \circ \iota_{\ast}$ by passing to $N_{n,0}$-invariants. The proposition then boils down to the following lemma.
\begin{lemma}
    The composition
    \[
    \on{CT}_w \circ s\circ \iota_{\ast}: \HH^{i - \ell(w)}(\Pi^{\la,\mf{n}_n}(\chi_w)) \xrightarrow{\iota_{\ast}} \HH^{i - \ell(w)}((I_{w}\Pi^{\la})^{w^{-1}\mf{n}w}(\chi_w)) \xrightarrow{s} \HH^{i}(I_w\Pi^{\la})^{\mf{n}}) \xrightarrow{\on{CT}_w} \HH^{i - \ell(w)}(\Pi^{\la,\mf{n}_n}(\chi_w))
    \]
    is the identity map.
\end{lemma}
\begin{proof}
    We pass to the associated (total) Chevalley--Eilenberg complexes, and use the explicit description of $\on{CT}_w$ from \eqref{eq: ExplicitCocycleFormula}. In particular, we need to explicitly describe \[
    s:\HH^{i-\ell(w)}(\on{Tot}(\CE(w^{-1}\mf{n}w, I_w\Pi^{\la})))\otimes \bigotimes_{i = 1}^{\ell(w)}\ip{w^{-1}X_{-\mf{r}_{i}}w}^{\ast} \to \HH^i(\on{Tot}(\CE(\mf{n},I_w\Pi^{\la}))).
    \]First, note the map $\iota_{\ast}:\HH^{i-\ell(w)}(\on{Tot}(\CE(\mf{n}_n,\Pi^{\la})) \to \HH^{i-\ell(w)}(\on{Tot}(\CE (w^{-1}\mf{n}w,I_w\Pi^{\la})))$ sends a cocycle $\sum Y_A^{\vee}\otimes f_A \mapsto \sum w^{-1}(Y_A^{\vee})\otimes \iota(f_A)$. Now we first look at the map $c_{\mf{r}_1}: \HH^{i-1}(\on{Tot}(\bigwedge (w^{-1}w'\mf{n}w(w')^{-1}w)^{\ast}\otimes I_w\Pi^{\la}))\otimes \ip{w^{-1}X_{-\mf{r}_1}w} \to \HH^i(\on{Tot}(\CE(\mf{n}, I_w\Pi^{\la})))$. Indeed, $c_{\mf{r}_1}$ arises from an edge/connecting map $\HH^1(\ip{w^{-1}X_{-\mf{r}_1}w}, \HH^{i-1}(\on{Tot}(\on{CE}(w^{-1}\mf{n}_{w'}w, I_w\Pi)))) \to \HH^i(\on{Tot}(\on{CE}(\mf{n}, I_w\Pi^{\la})))$ of a Hochschild-Serre type spectral sequence, so is naturally an inflation map which in terms of explicit cocycles and using the Chevalley--Eilenberg complex for $\ip{w^{-1}X_1w}$ sends the class of 
    \[
[\sum w^{-1}(X_{-\mf{r}_1})\otimes w^{-1}(X_A)\otimes w^{-1}(Y_B^{\vee})\wedge w^{-1}(Z_C^{\vee})\otimes f_{B,C}] \mapsto [\sum w^{-1}(X_{-r_1}^{\vee})\wedge w^{-1}(X_A^{\vee})\wedge w^{-1}(Y_B^{\vee}) \wedge w^{-1}(Z_C^{\vee}) \otimes f].
    \]
Iterating this description for the other $s_{\mf{r}_i}$, we see that the map $s = s_{\mf{r}_{\ell(w)}}\circ \dots \circ s_{\mf{r}_1}$
explicitly sends a cocycle
\begin{equation}
[\sum w^{-1}(Y_B^{\vee}) \wedge w^{-1}(Z_C^{\vee}) \otimes f_{A,B,C}] \mapsto [\sum w^{-1}(X_{[\ell(w)]}^{\vee})\wedge w^{-1}(Y_B^{\vee})\wedge w^{-1}(Z_C^{\vee}) \otimes f_{B,C}].
\end{equation}

Altogether, the map $s \circ \iota_{\ast}$ sends 
\[
[\sum Y_B^{\vee}\otimes f_B] \mapsto [\sum w^{-1}(X_{[\ell(w)]}^{\vee})\wedge w^{-1}(Y_B^{\vee})\otimes \iota(f_B)].
\]
For the purposes of proving that $s \circ \iota_{\ast}$ is a section, we may assume that we start with $v_0 =  f_{\emptyset}\otimes w^{-1}(X_{[\ell(w)]}^{\vee}$ for $f_{\emptyset} \in \Pi^{i-\ell(w),\la,\mf{n}_n}\subset \on{Tot}^{i-\ell(w)}(\on{CE}(\mf{n}_n, \Pi^{\la}))$, since $\Pi^{\la}$ is a complex of $R\Gamma_{\mf{n}_n}$-acyclic objects. Then $(s \circ \iota_{\ast})([v_0]) = [w^{-1}X_{[\ell(w)]}^{\vee}\otimes \iota(f_{\emptyset})]$

On the other hand, from \eqref{eq: ExplicitCocycleFormula} the composition of $f \mapsto f(w)$, together with the projection map
    \[
    \ol{E}_{\ell(w)}:\HH^{i}(\Pi^{\la, \mf{n}_{w}}) \to \HH^{i - \ell(w)}(\tau_w^{-1}\Pi^{\la, \mf{n}_n}\otimes_LL(\chi_w))
    \]
sends 
\[
[\sum w^{-1}(Y_B^{\vee})\wedge w^{-1}(X_{[\ell(w)]}^{\vee})\otimes f] \mapsto
[\sum w^{-1}(Y_B^{\vee})\otimes f(w) \otimes  w^{-1}(X_{[1,\dots, \ell(w)]}^{\vee})],
\]
Then the composition with $s \circ \iota_{\ast}$ sends $[v_0]  \mapsto [\iota(f_{\emptyset})(w)\otimes w^{-1}X_{[\ell(w)]}^{\vee}] = [f_{\emptyset}\otimes w^{-1}X_{[\ell(w)]}^{\vee}] = [v_0]$ which is indeed the identity.

\end{proof}
This lemma finishes the proof of the first part of the proposition, and by our discussion above, implies the second part of the proposition by passing to $N_{n,0}$-invariants.

\end{proof}

\begin{remark}
    One might wonder to what extent this argument could be generalized to give surjectivity on the other terms $\HH^{i-j}(R\Gamma_{\mf{n}_n}(\Pi \otimes \HH^j(\mf{n}_{w,\ol{U}},\mathbf{1})))$ for $j < \ell(w)$. The proof above certainly fails, as there is no clear analogue for this section $s$ we construct.

    More unfortunately, we do not expect such a surjectivity in general. Rather, we expect that the constant term must satisfy an algebraicity condition with respect to the $\mf{t}$-action (Roughly speaking, the $N_0$ action on functions on some $PwN$ intertwines positive and negative roots, which makes torus elements appear). Thus, one should only expect the ``maximal'' degree shift to extend over the whole weight space.
\end{remark}
\subsection{Degree shifting maps}
We can now combine our work with previous sections to realize all possible degrees of Jacquet functors for $G = \GL_n/F$ in the middle degree Jacquet functors of $\wt{G}$.
\begin{theorem}
\label{theorem: JacquetSubQuotient}
    Let $\wt{K}^p \subset \wt{G}(\A_{F^+}^{p,\infty})$ be a good subgroup which is decomposed with respect to $\ol{P}$. Let $\mf{m}\subset \T_G^{S}(K^p)$ be a decomposed generic, non-Eisenstein maximal ideal, and set $\wt{\mf{m}}:= \mc{S}^{\ast}(\mf{m})$
    
    Then for all $i$ there exists $w_i\in W^{\oP}$ such that $J_{B}(\wt{\HH}^d(\widetilde{K}^p)_{\wt{\mf{m}}}^{\la})$ and $\HH^d(J_{B}\Pi_{c}(\wt{K}^p)_{\wt{\mf{m}}}^{\la})$ admits $\HH^{i}(J_B(\tau_{w_i}^{-1}(\Pi(K^p)_{\mf{m}}^{\la})))\otimes_LL(\chi_{w_i})$ as a subquotient essentially admissible representation of $T(F_p^+)$.
\end{theorem}
\begin{proof}
  Note that $\HH^i(J_B(\Pi_{\mf{m}}(K^p)^{\la}))  = 0$ if $i \notin [0,\dim X_{K^pK_p}^G] = [0,d-1]$. Thus we may assume $i \le d-1$.
  
  Recall by \cref{lemma: FundamentalLongExactSequence} and \cref{prop: InductionAsDirectSummandOfBoundary}, we have $\HH^d(J_{B}(\Ind_{\oP_0}^{\tG_0}\Pi(K^p)_{\mf{m}}^{\la}))$ is a quotient of $J_{B}(\wt{\HH}^d(\wt{K}^p, L)_{\wt{\mf{m}}}^{\la})$, and $\HH^{d-1}(J_{B}(\Ind_{\oP_0}^{\tG_0}(\Pi_{\oP}(\wt{K}^p)_{\wt{\mf{m}}}^{\la}))$ is a subspace of $\HH^d(J_{B}(\Pi_{c}(\wt{K}^p)_{\wt{\mf{m}}}^{\la}))$. We now present the rest of the proof for compactly supported cohomology, with the only difference being the choice of $w_i$ (in particular, the value of $\ell(w_i)$).
  
  Choose $w_i$ such that $\ell(w_i) = [F^+:\Q]n^2 - 1- i$, so that $d - \ell(w_i) = i$. Applying \cref{prop: MainDegreeShiftProp} for $w = w_i$, $\Pi = \Pi_{\oP}(\wt{K}^p)_{\wt{\mf{m}}}$ and $\Pi_G = \Pi(K^p)_{\mf{m}}$ yields a continuous $T(F_p^+)^+$-equivariant surjection $\HH^{d-1}((I_{w_i}\Pi_{\oP}(\wt{K}^p)_{\wt{\mf{m}}}^{\la})^{N_0}) \to \HH^{i}(\tau_{w_i}^{-1}(\Pi(K^p)_{\mf{m}}^{\la})^{N_{n,0}})\otimes_LL(\chi_{w_i})$. By \cref{lemma: BruhatSES} and the previous paragraph, it suffices check that this surjection remains a surjection after passing to finite slope parts. 

  Recalling \cite[Proposition 4.9]{FuDerived}, Jacquet functors may be computed as the finite slope part of the Hausdorff quotient of $N_0$-cohomology. In particular, passing to Hausdorff quotients preserves surjectivity, so we get a surjection of $T(F_p^+)^+$-modules $\HH_{\on{haus}}^d((I_{w_i}\Pi_{\oP}(\wt{K}^p)_{\wt{\mf{m}}}^{\la})^{N_0}) \to \HH_{\on{haus}}^{i}(\tau_{w_i}^{-1}(\Pi(K^p)_{\mf{m}}^{\la})^{N_{n,0}})\otimes_LL(\chi_{w_i})$. Now recall by the definition of $W^{\oP}$ we have $w_iN_0w_i^{-1}\supset N_{n,0}$, which implies $w_iT(F_p^+)^+w_i^{-1}\subset T_n(F_p)^+$, and in fact $w_iT(F_p^+)^{++}w_i^{-1}\subset T_n(F_p)^{++}$. Choosing an element $z \in T(F_p^+)^{++}$, we thus have by \cite[Theorem 4.5]{FuDerived}
  a surjection $\HH_{\on{haus}}^d((I_{w_i}\Pi_{\oP}(\wt{K}^p)_{\wt{\mf{m}}}^{\la})_{z-\on{fs}}^{N_0}) \to \tau_{w_i}^{-1}(\HH_{\on{haus}}^{i}((\Pi(K^p)_{\mf{m}}^{\la})^{N_{n,0}})_{w_izw_i^{-1}-\on{fs}})\otimes_LL(\chi_{w_i})$, which in turn is a surjection
  $\HH^d(J_{B}(I_{w_i}\Pi(K^p)_{\mf{m}}^{\la})) \twoheadrightarrow \tau_{w_i}^{-1}\HH^{i}(J_{B_n}(\Pi(K^p)_{\mf{m}}^{\la}))\otimes_L L(\chi_{w_i})$, as desired.
\end{proof}

Consider the map
$j_{w}:\widehat{T(F_p^+)} \to \widehat{T(F_p^+)}$ sending a character $c$ to $w^{-1}(c)\cdot \chi_w$. Concretely, on $\mc{W}$ this map sends a weight $\lambda \mapsto j_w(\lambda) = w^{-1}(\lambda - w_0^{\tG}\rho^{\tG}) + w_0^{\tG}\rho^{\tG}$, where $\rho^{\tG}$ is the half sum of positive roots.

\begin{theorem}
\label{theorem: Embeddings of eigenvarieties}
    For any $i \in [0, d-1]$ let $w_i\in W^{\oP}$ be as in \cref{theorem: JacquetSubQuotient}. Then the map $\mc{S}^{\ast}\times j_{w_i}:\Spec\T^S(K^p)_{\mf{m}}\times \widehat{T(F_p^+)} \to \Spec \wt{\T}^S(\wt{K}^p)_{\wt{\mf{m}}}\times \widehat{T(F_p^+)}$ induces a closed embedding of rigid varieties $\mc{E}^i(K^p)_{\mf{m}} \hookrightarrow
    \mc{E}_c^d(\wt{K}^p)_{\wt{\mf{m}}}$ fitting into the following diagram 
    \[
    \begin{tikzcd}
        \mc{E}^i(K^p)_{\mf{m}}\ar[r, hook]\ar[d] &\mc{E}_{c}^d(\wt{K}^p)_{\wt{\mf{m}}}\ar[d] \\
        \mc{W}\ar[r,"{\lambda \mapsto j_{w_i}(\lambda)}"]&\mc{W}
    \end{tikzcd}.
    \] 
\end{theorem}
\begin{proof}
    
    Let $\mc{M}_{c,\wt{\mf{m}}}^d:= (\HH^d(J_{B}(\Pi_{c}(\wt{K}^p)_{\wt{\mf{m}}}^{\la, N_0})))'$ and $\mc{M}_{\mf{m}}^i := \HH^i(J_{B_n}(\Pi(K^p)_{\mf{m}}^{\la}))$ be the strong duals, which are coadmissible $\mc{O}_{\widehat{T}}(\widehat{T})$-modules. Recall the eigenvarieties $\mc{E}_c^d(\wt{K}^P)_{\wt{\mf{m}}}$ and $\mc{E}^i(K^p)_{\mf{m}}$ are defined via the morphisms $\psi_c^d: \wt{\T}^S(\wt{K}^p)_{\wt{\mf{m}}} \to \End_{\mc{O}_{\widehat{T}}}(\mc{M}_{c,\wt{\mf{m}}}^d)$ and $\psi^i: \wt{\T}^S(\wt{K}^p)_{\wt{\mf{m}}} \to \End_{\mc{O}_{\widehat{T}}}(\mc{M}_{\mf{m}}^i)$, and passing to the associated $\mc{O}_{\widehat{T}}$-algebras $\wt{\mc{A}}, \mc{A}$ generated by the images of $\psi_c^d, \psi^i$ respectively. By \cref{theorem: JacquetSubQuotient} the map $S: \wt{\T}^S(\wt{K}^p)_{\wt{\mf{m}}} \to \T^S(K^p)_{\mf{m}}$ then induces a surjective morphism of $\mO_{\widehat{T}}(\wh{T})$-algebras $\wt{\mc{A}} \to L(\chi_{w_i})\otimes_L\tau_{w_i}^{-1}\mc{A}$.
    Passing to the relative spectra over $\widehat{T}$ gives the desired closed embedding, where we note that the map on $\widehat{T}$ is twisted via $j_{w_i}$ to incorporate the $T$-action on $L(\chi_{w_i})\otimes_L\tau_{w_i}^{-1}A$.
\end{proof}

\section{The unitary group in middle degree}
\label{Section5: UnitaryGroup}
 \cref{theorem: Embeddings of eigenvarieties} enables us to study eigenvarieties for $G$ via a single eigenvariety for $\tG$. The goal of this section is study $\tG$, and in particular prove \cref{theorem: Step 3}. For the rest of this article, we will often suppress tame level from our notation for eigenvarieties and denote $\mc{E}_{c,\wt{\mf{m}}}^d:= \mc{E}_{c}^d(\wt{K}^p)_{\wt{\mf{m}}}, \mc{E}_{\wt{\mf{m}}}^d:= \mc{E}^d(\wt{K}^p)_{\wt{\mf{m}}}$.

The proof is largely inspired by the treatment given in \cite[Section 3]{BHSAnnalen}, which combines locally analytic representations with standard eigenvariety tools from \cite{Buz07}. We note, however, the key input into their discussion is a Banach representation $\Pi$ such that $\Pi|_{H}\simeq \mc{C}(H,L)^{\oplus M}$ for some compact open subgroup $H \subset \tG_0 = \tG(\mO_{F_p^+})\simeq \GL_{2n}(\mc{O}_{F,p})$ (in other words, $\Pi$ is an \textit{injective} admissible representation of $\tG_0$). Such a property is typically \textit{not} satisfied by the (compactly supported) completed cohomology groups $\wt{\HH}_{(c)}^i(\wt{K}^p, L)_{\wt{\mf{m}}}$ whenever $\wt{\mf{m}}$ is Eisenstein. The key fact we will need is that under a genericity assumption on $\wt{\mf{m}}$, the complex $\Pi_c(\wt{K}^p)_{\wt{\mf{m}}}$ can be represented by a finite complex of Banach representations $\Pi_{c, \wt{\mf{m}}}^{\on{bdd}}$ such that $\Pi_{c, \wt{\mf{m}}}^{\on{bdd},i}|_{H}\simeq \mc{C}(H, L)^{\oplus m_i}$ for all $i$, \textit{concentrated in degrees} $[0, d]$. This property gives us good control on derived Jacquet modules in middle degree, and allows us to push through the arguments of \cite{BHSAnnalen}.

\subsection{Compactly supported cohomology}

We first recall some key definitions and properties about torsion and completed cohomology in our setting.
\begin{definition}
\label{definition: DecomposedGeneric}
    A continuous representation $\overline{\rho}: \Gal_F \to \GL_d(\overline{\F}_p)$ is called \textit{decomposed generic} if there exists a prime $\ell$ which splits completely in $F$, and for all places $\nu \mid \ell$ of $F$, $\overline{\rho}|_{\Gal_{F_{\nu}}}$ is unramified with the eigenvalues $\alpha_1,\dots, \alpha_d$ of $\overline{\rho}(\Frob_{\nu})$ satisfying $\alpha_i/\alpha_j \neq \ell$ for all $i \neq j$.
\end{definition}
\begin{theorem}[Caraiani--Scholze \cite{CS19}, Koshikawa \cite{KoshikawaGeneric}]
\label{theorem: CaraianiScholzeTorsionVanishing}
    Let $\wt{\mf{m}}$ be a maximal ideal in the support of $\HH^{\ast}(\widetilde{X}_{\wt{K}^p\wt{K}_p}, \F_p)$ such that $\overline{\rho}_{\wt{\mf{m}}}$ is decomposed generic.
    Then we have $\HH^i(\wt{X}_{\wt{K}^p\wt{K}_p}, \F_p)_{\wt{\mf{m}}} = 0$ for $i < d$, $\HH_c^i(\wt{X}_{\wt{K}^p\wt{K}_p}, \F_p)_{\wt{\mf{m}}} = 0$ for $i > d,$  and $\wt{\HH}^i(\wt{K}^p, \Z_p) = 0$ for $i > d$. Consequently, the integral completed cohomology $\wt{\HH}^i(\wt{K}^p, \Z_p)_{\wt{\mf{m}}}$ is concentrated in degree $i = d$.
\end{theorem}
\begin{proof}
    The first part follows from \cite[Theorem 1.3]{KoshikawaGeneric} (which removes certain hypotheses from \cite{CS19}), and the assertion about completed cohomology follows from \cite[Theorem 2.6.2 Lemma 4.6.2]{CS19}.
\end{proof}

In general we call $\wt{\mf{m}}\subset \wt{\T}^S(\wt{K}^p)$ \textit{decomposed generic} if the associated Galois representation $\ol{\rho}_{\wt{\mf{m}}}$ is. The same definition applies to maximal ideals $\mf{m} \subset \T^S(K^p)$ for $G$. For the rest of this section, we will assume that $\wt{\mf{m}}$ is decomposed generic.

Note that dually, \cref{theorem: CaraianiScholzeTorsionVanishing} implies that $\HH_{i}^{BM}(\wt{X}_{\wt{K}^p\wt{K}_p}, \F_p)_{\wt{\mf{m}}} = 0$ for $i > d$. Now consider the complex $C_{\bullet}^{\on{BM}, \on{BS}}(\wt{K}^p\wt{K}_p, \mO_L[[\wt{K}_p]])_{\wt{\mf{m}}}$ for $\wt{K}_p$ pro-$p$. This is a perfect complex of $\mO_L[|\wt{K}_p|]$-modules. We thus have $C_{\bullet}^{\on{BS}}(\wt{K}^p\wt{K}_p, \mO_L[[\wt{K}_p]])_{\wt{\mf{m}}}\otimes_{\mO_L[[\wt{K}_p]]}^{\mathbb{L}}\mO_L/\mf{m}_L\simeq C_{\bullet}^{\on{BS}}(\wt{K}^p\wt{K}_p, \mO_L[[\wt{K}_p]])_{\wt{\mf{m}}}\otimes_{\mO_L[[\wt{K}_p]]}\mO_L/\mf{m}_L\simeq C_{\bullet}^{\on{BM},\on{BS}}(\wt{K}^p\wt{K}_p, \mO_L/\mf{m}_L)_{\wt{\mf{m}}}$ is concentrated in degrees $[0, d]$, and so the theory of \textit{minimal resolutions} (in the guise of \cite[Proposition 2.1.9]{GeeNewton}) implies the following:

\begin{proposition}
\label{proposition: compactlysupportedBoundedComplex}
$\Pi_c(\wt{K}^p)_{\wt{\mf{m}}}$ is quasi-isomorphic to a complex $\Pi_{c,\wt{\mf{m}}}^{\on{bdd}}$ of admissible Banach representations of $H \subset \tG_0$ concentrated in degrees $[0, d]$ for such that for some pro-$p$ compact open subgroup $H \subset \wt{G}_0$ we have $\Pi_{c,\wt{\mf{m}}}^{\bullet,\on{bdd}}\mid_{H}\simeq \mc{C}(H, L)^{\oplus m_i}$. Similarly, for $\Pi(\wt{K}^p)_{\wt{\mf{m}}}$ is quasi-isomorphic to a complex of injective admissible representations concentrated in degrees $[d,2d]$.
\end{proposition}
\begin{proof}
    Noting that $\Hom_{\on{cont}}(\wt{\HH}_{(c)}^i(\wt{K}^p, \Q_p)_{\wt{\mf{m}}},\Q_p)\simeq \wt{\HH}_i^{(BM)}(\wt{K}^p,\Q_p)_{\wt{\mf{m}}}$, it suffices to work with homology. The discussion above (in particular, \cite[Proposition 2.1.9]{GeeNewton}) implies $C_{\bullet}^{\on{BM},\on{BS}}(\wt{K}^p, \Z_p)_{\wt{\mf{m}}}$ is quasi isomorphism to a complex $B_{\bullet}$ of finite free $\mc{O}[[H]]$-modules, such that $B_{\bullet}\otimes_{\mc{O}[[H]]]}\mc{O}([[H]])/\mf{m}\simeq \bigoplus_i \HH_{i}(\wt{K}^pH, \F_p)_{\wt{\mf{m}}}[-i]$. Since $\HH_{i}^{\on{BM}}(\wt{K}^pH, \F_p)_{\wt{\mf{m}}}$ is concentrated in degrees $[0,d]$, so is $B_{\bullet}$. Then taking $(B_{\bullet}\otimes_{\Z_p}\Q_p)^{\vee}$ yields the desired complex. The statement for $\Pi(\wt{K}^p)_{\wt{\mf{m}}}$ follows from the analogous vanishing of $\HH_i(\wt{K}^p, \F_p)_{\wt{\mf{m}}} = 0$ for $i <d$.
\end{proof}
\begin{remark}
    An alternative proof (and the one we originally conceived of) can also be given using the Poincar\'{e} duality spectral sequence for completed (co)homology \cite[1.3]{CalEmSurvey}.
\end{remark}

\subsection{Middle degree cohomology}
We briefly say something about eigenvarieties $\mc{E}_{\wt{\mf{m}}}^d$ for middle degree completed cohomology. The direction of vanishing for torsion cohomology in fact tells us that $\mc{E}_{\mf{m}}^d$ coincides with Emerton's eigenvarieties from \cite{EmertonInterpolation}.
\begin{lemma}
\label{lemma: DerivedEqualsNormalJacquet}
There is an isomorphism $\HH^d(J_B(\Pi(\wt{K}^p)_{\wt{\mf{m}}}^{\la}))\simeq J_B(\widetilde{\HH}^d(\wt{K}^p, L)_{\wt{\mf{m}}}^{\la})$ of essentially admissible representations of $T(F_p^+)$.
\end{lemma}
\begin{proof}
    By \cref{proposition: compactlysupportedBoundedComplex} we can represent $\Pi(\wt{K}^p)_{\wt{\mf{m}}}$ by a complex of injective admissible $H$-representations $I_{d} \xrightarrow{g_d} \dots \xrightarrow{g_{2d-1}} I_{2d}$ concentrated in degree $[d,2d]$. Thus, $\HH^d(J_B(\Pi(\wt{K}^p)_{\wt{\mf{m}}}^{\la})$ is computed by $(((\ker g_{d})^{\la, N_0})_{\on{haus}})_{\on{fs}}$ by \cite[Proposition 4.9]{FuDerived} (using that taking $N_0$-invariants is left exact, and thus commutes with kernels), and where $(-)_{\on{haus}}$ denote the Hausdorff quotient. However, since $g_{d}$ is a continuous map, $(\ker g_{d})^{\la,N_0}$ is a closed subspace of $I_d^{\la,N_0}$, and is thus automatically Hausdorff. Thus we actually just have $\HH^d(J_B(\Pi(\wt{K}^p)_{\wt{\mf{m}}}^{\la})) = ((\ker g_{d}^{\la})^{N_0})_{\on{fs}} = J_B((\ker g_d)^{\la}) = J_B(\wt{\HH}^d(\wt{K}^p, L)_{\wt{\mf{m}}}^{\la})$, as required.
\end{proof}

When the cohomology $\wt{\HH}^{\ast}(\wt{K}^p, L)_{\wt{\mf{m}}}$ is concentrated in a single degree, the relation between the Jacquet functor and classical automorphic forms is also more lucid. By Emerton's spectral sequence \cite[Theorem 0.5]{EmertonInterpolation}, concentration in one degree implies that the locally algebraic vectors $\HH^d(\wt{K}^p, L)_{\wt{\mf{m}}}^{\lalg} \simeq \bigoplus_{W} W^{\vee}\otimes \HH^d(\wt{K}^p, W)_{\wt{\mf{m}}},$ where $W$ ranges over irreducible algebraic representations of $\wt{G}(F_p^+)$, and  $\HH^d(\wt{K}^p, W^{\vee}):=\varinjlim_{\wt{K}_p}\HH^d(\wt{X}_{\wt{K}^p\wt{K}_p}, W^{\vee})$. So the locally algebraic vectors is directly related to classical cohomology, and by \cite[Proposition 4.3.6]{Emerton_Jacquet_I} $J_B(\wt{\HH}^d(\wt{K}^p,L)_{\wt{\mf{m}}}^{\lalg})\simeq \bigoplus_{\mathsf{L}(\lambda)}J_B(\mathsf{L}(\lambda)\otimes_L\HH^d(\wt{K}^p, \mathsf{L}(\lambda)^{\vee})_{\wt{\mf{m}}})\simeq \bigoplus_{\lambda}\lambda\otimes_LJ_B(\HH^d(\wt{K}^p, \mathsf{L}(\lambda)^{\vee}))(\delta_B^{-1})$, where $\lambda$ ranges over highest weight vectors $\lambda \in (\Z_{\on{dom}}^{2n})^{\Hom(F^+, \ol{\Q}_p)}$.
In general, we call a point $x = (y, \delta) \in \mc{E}_{\wt{\mf{m}}}^d$ \textit{classical} if the space $\Hom_{T(F_p^+)}(\delta, J_B(\wt{\HH}^d(\wt{K}^p,L)_{\wt{\mf{m}}}^{\lalg})[\wt{\T}^S = y]) \neq 0$.

Now fix an isomorphism $\iota: \ol{\Q}_p \simeq \C$ and $\pi$ is an automorphic representation of $\wt{G}(\A_{F^+})$ of weight $\lambda$ which is unramified outside $S$ and thus has Hecke eigenvalues $\iota^{-1}\circ \psi_{\pi}: \wt{\T}^S \to \ol{\Q}_p$.

Suppose moreover $J_B(\wt{\HH}^d(\wt{K}^p,L)_{\wt{\mf{m}}}^{\lalg})[\psi_{\pi}] \neq 0$, then in fact that $J_B(\wt{\HH}^d(\wt{K}^p, \mathsf{L}(\lambda)))[\psi_{\pi}]\neq 0$. Now for all $\ol{\nu} \in S_p(F^+)$ suppose the Jacquet module $J_B(\iota^{-1}\circ \pi_{\ol{\nu}})$ admits a subquotient $\theta_{\ol{\nu}}: T(F_{\ol{\nu}}^+) \to \ol{\Q}_p^{\times}$ and set $\theta:= \prod_{\ol{\nu}}\theta_{\ol{\nu}}: T(F_p^+) \to \ol{\Q}_p^{\times}$. 
 Then by the above discussion we have $\Hom_{T(F_p^+)}(\lambda^{\vee}\theta\cdot \delta_B^{-1}, J_B(\wt{\HH}^d(\wt{K}^p, L)_{\wt{\mf{m}}}^{\la})) \neq 0$, and thus $(\psi, \lambda^{\vee} \cdot \theta\cdot \delta_B^{-1})\in \mc{E}_{\wt{\mf{m}}}^d(\ol{\Q}_p)$. 

\begin{definition}
\label{definition: NumericallyNoncritical}
    A character $\delta_{\ol{\nu}}\in \widehat{(F_{\nu}^{\times})^{2n}}(\overline{\Q}_p)$ is \textit{numerically noncritical} if for all $\tau \in \Hom_{\Q_p}(F_{\nu}, \overline{\Q}_p)$ we have:
    \begin{enumerate}
        \item the weights $\on{wt}_{\tau}(\delta_{\ol{\nu},1}), \dots \on{wt}_{\tau}(\delta_{\ol{\nu},2n})$ are a (strictly) increasing sequence of integers;
        \item for each $i = 1, \dots, 2n-1$ we have 
        \[
        v_p((\delta_{\ol{\nu}, 1}\dots \delta_{\ol{\nu}, i})(p)) < \on{wt}_{\tau}(\delta_{\nu, i+1}) - \on{wt}(\delta_{\ol{\nu}, i}).
        \]
    \end{enumerate}
    The motivation for this definition comes from small-slope classicality criteria. 
\end{definition}
\begin{lemma}
\label{lemma: SmallSlopeClassicalForUsualJacquetFunctor}
    Let $z = (y,\delta) \in \mc{E}_{\wt{\mf{m}}}^{d}$ be such that $j_{2n}(\delta) = \lambda\delta_{\on{sm}},$ where $\lambda$ is a strictly dominant algebraic character, $\delta_{\on{sm}}$ is smooth, and $\delta$ is numerically noncritical. Then $z$ is classical.
\end{lemma}
\begin{proof}
    This result is a formal consequence of Emerton's original criterion \cite[Theorem 4.4.5]{Emerton_Jacquet_I} for when the Jacquet functor commutes with passing to locally algebraic vectors. (See \cite[Remark 2.21]{NT21} for the relationship between $j_n$, numerically noncritical, and non-critical in the sense of \cite[Definition 4.4.3]{Emerton_Jacquet_I}, together with \cref{remark: UzOperatorNormalisation}).
\end{proof}

While this corollary gives a clear understanding of ``classical points'' in $\mc{E}_{\wt{\mf{m}}}^d$,  we remark $\mc{E}_{\wt{\mf{m}}}^d$ is not well-suited to proving some sort of density statement, since the coherent sheaf $\mc{M}_{\wt{\mf{m}}}^d  = (J_B(\wt{\HH}^d(\wt{K}^p, L)_{\wt{\mf{m}}}^{\la}))'$ (locally) contains \textit{torsion} over the weight space $\mc{W}$. Indeed, by \cref{corollary: SmallSupport} and \cref{lemma: FundamentalLongExactSequence}, we have an embedding of a (locally) torsion module $\mc{M}_{\partial, \wt{\mf{m}}}^{d} \hookrightarrow \mc{M}_{\wt{\mf{m}}}^d$. So we focus instead on the compactly supported case.

\subsection{Middle-degree derived Jacquet module}
 We use the complex $\Pi_{c, \wt{\mf{m}}}^{\on{bdd}}$ of injective admissible Banach representations of $H$ from
\cref{proposition: compactlysupportedBoundedComplex} which represents $\Pi_{c}(\wt{K}^p)_{\wt{\mf{m}}}$. Since these are \textit{finite} complexes of injectives, the usual trick of \cite[Corollary 10.4.7]{WeibelHomAlg} upgrades this quasi-isomorphism to a homotopy equivalence of chain complexes in $\Ban_{H}^{\on{ad}}$. This fact provides $H$-equivariant homotopies
\begin{equation}
\label{eq: LiftToHomotopies}
\begin{tikzcd}
\Pi_{c,\wt{\mf{m}}}^{\on{bdd},\bullet}\ar[r, "n"', shift right = 0.5ex] & \Pi_{c, \wt{\mf{m}}}^{\bullet}\ar[l, "s"', shift right = 0.5ex]\ar[r, "i"', shift right = 0.5ex] &A_{c, \wt{\mf{m}}}^{\bullet} \ar[l, "p"', shift right = 0.5ex]
\end{tikzcd}
\end{equation}
such that $s, n$ are $H$-equivariant homotopy inverses, and $i,p$ are $\wt{G}_0$-equivariant homotopy inverses.

Using \eqref{eq: LiftToHomotopies} after choosing an element $z \in T(F_p^+)^{++}$ we can consider the cohomology $\HH^d(J_{B}(\Pi_{c, \wt{\mf{m}}}^{\on{bdd},\la}))$, which by \cref{lemma: IndependenceOfHomotopyForJacquetFunctors} is isomorphic to $\HH^{d}(J_B(\Pi_c(\wt{K}^p)_{\wt{\mf{m}}}^{\la}))$. This module is an essentially admissible $T(F_p^+)$-representation by \cite[Proposition 3.20]{FuDerived}. Let $Y_H \subset T(F_p^+)$ be the closed subgroup generated by $z$ and $T(F_p^+) \cap H$. Since $\Pi_{c, \wt{\mf{m}}}^{\on{bdd}}$ is concentrated in degrees $[0, d]$, there is a natural surjection of essentially admissible $Y_H$-representations 
\[
J_{B}(\Pi_{c,\wt{\mf{m}}}^{d,\on{bdd},\la}) \twoheadrightarrow \HH^d(J_{B}(\Pi_{c}(\wt{K}^p)_{\wt{\mf{m}}}^{\la})).
\]
Let $\mc{M}_c^d, \mc{N}$ be the coherent sheaves on $\widehat{Y}_H\simeq \mc{W}_H \times_{L}\G_m$ associated to $\HH^d(J_{B}(\Pi_{c,\wt{\mf{m}}}^{\on{bdd},\la}))$ and $J_{B}(\Pi_{c, \wt{\mf{m}}}^{\on{bdd},d,\la})$ respectively, so that we have a natural embedding $\mc{M}_c^d \hookrightarrow \mc{N}$. Let $\mc{Z}_{c}^d := \supp_{\mc{W}\times_L\G_m}\mc{M}_c^d$ and $\mc{Z}_{\mc{N}} := \supp_{\mc{W}_H
\times_L\G_m}\mc{N}$ be their supports. Note that the restriction map $\on{res}_H: \mc{W} \to \mc{W}_H$ is a finite \'{e}tale cover, and also that we have an inclusion of analytic closed subsets $\on{res}_H(\mc{Z}_{c}^d) \subset \mc{Z}_{\mc{N}} \subset \mc{W}_H\times_L \G_m$.

We can use standard methods in the theory of eigenvarieties (as presented in \cite[\textsection 3, \textsection 5]{BHSAnnalen}) to prove that in some sense, $\mc{M}$ is ``large'' over the weight space $\mc{W}$. We note that the following arguments have strong similarities with those in \cref{proposition: U_pfactorisationDiagram}. The main difference is that here we work with (a version of ) the ``$U_1z$''-action instead of ``$U_2z$''-action, which are homotopic by \cref{lemma: U_1zhomotopicToU_2z}. First, let $\pi_{\mc{W}}: \widehat{Y}_H \to \mc{W}_H$ denote the projection onto weight space.

\begin{proposition}
\label{prop: FredholmHypersurface}
    The closed analytic subset $\mc{Z}_{\mc{N}} \subset \mc{W}_H\times_L \G_m$ is naturally a Fredholm hypersurface.
    Moreover there exists an affinoid open admissible cover $(U_i'')_i \subset \mc{Z}_{\mc{N}}$ so that the induced maps $\pi_{\mc{W}}: U_i'' \to \pi_{\mc{W}}(U_i'')  = W_i'' \subset \mc{W}_H$ is a finite surjective onto a Zariski open $W_i'' \subset \mc{W}_H$, and $U_i''$ is a connected component of $\pi_{\mc{W}}^{-1}(W_i'')$, over which $\Gamma(U_i'',\mc{N})$ is a finite projective $\mc{O}_{\mc{W}_H}(W_i'')$-module. 

    Correspondingly, there is a admissible affinoid cover $(U_i')_i \subset \mc{Z}_c^d$ such that the induced maps $\pi_{\mc{W}}:U_i' \to W_i':= \pi_{\mc{W}}(U_i)\subset \mc{W}$ are finite surjective onto an affinoid open, and $U_i'$ is a connected component of $\pi_{\mc{W}}^{-1}(W_i')$. Moreover, $\Gamma(U_i', \mc{M}_c^d)$ finitely generated \textit{torsion-free} $\mc{O}_{\mc{W}}(W_i')$-module, and $\on{res}_H(\mc{Z}_{c}^d)$ is a union of irreducible components of $\mc{Z}_{\mc{N}}$.
\end{proposition}

\begin{proof}
    The proof follows the arguments of \cite[Lemme 3.10, Proposition 5.3]{BHSAnnalen} up to replacing certain inputs. We give a summary of the main steps, and how they push through in our setting. Let $V:= \Pi_{c,\wt{\mf{m}}}^{\on{bdd},d,\la}$.
    Suppose that $H\subset \wt{G}_0$ is uniform pro-$p$ with an Iwahori decomposition $H \simeq (\ol{N} \cap H) \times (T \cap H) \times (N \cap H)$. Let $N_H:= N \cap H$. By construction we have that $V|_{H}\simeq \mc{C}(H, L)^{\oplus m_d}$ for some $m_d \ge 0$. For the statements regarding $\mc{N}$, the essential key point is  \cite[Proposition 5.3]{BHSAnnalen} in our setting:
    that is, there exists a countable admissible affinoid open cover $W_1' \subset \dots \subset W_h' \subset \dots$ of $\mc{W}$ and (Pr) Banach $A_h:= \mc{O}_{\mc{W}}(W_h')$-modules $V_h$ such that $((V^{\la})^{N_0})'\simeq\varprojlim_hV_h$, and the dual $U_z$-action induces $A_h$-compact endomorphisms $z_h: V_h \to V_h$ fitting into a commutative diagram
    \[
    \label{diag: UpOperatorFactorization}
    \begin{tikzcd}
    V_{h+1}\ar[r]\ar[d, "z_{h+1}"']&V_{h+1}\otimes_{A_{h+1}}A_h\ar[d, "{z_{h+1}\otimes_{A_{h+1}}1_{A_h}}"']\ar[r,"\beta_h"]&V_h \ar[dl, "\alpha_h"]\ar[d, "z_h"]\\
V_{h+1}\ar[r]&V_{h+1}\otimes_{A_{h+1}}A_h\ar[r, "\beta_h"]&V_h
\end{tikzcd}
    \]
    where the $\beta_h$ are $A_h$-compact.
    In our case, it suffices to show such a description goes through when one only has this action up to homotopy, namely $U_{z}$ as defined via \eqref{eq: LiftToHomotopies}. The key will be that the $V_h$ are (roughly) functorial in $V$, and the morphisms $\alpha_h$ comes from some property of increasing a radius of analyticity, which is also functorial.
    
    More rigorously, just as in \cref{proposition: U_pfactorisationDiagram} we use the functors $(-)_H^{(h)}$constructed in \cite{ColmezDospinescu}. First note the Iwahori decomposition induces a topological isomorphism
    \[
    \mc{C}^{(h)}(H,L)\simeq \mc{C}^{(h)}(N\cap H, L)\wh{\otimes}_L \mc{C}^{(h)}(T\cap H, L)\wh{\otimes}_L\mc{C}^{(h)}(N\cap H, L).
    \]

    A version of Emerton's untwisting lemma \cite[3.6.4]{Emerton_la_reps} readily implies that we have a $T_H$-equivariant isomorphism

\begin{equation}
\label{eq: (h)AnalyticUntwisting}
V^{N_0}|_{H}\simeq \mc{C}^{(h)}(H, L)^{\oplus m_d, N_0} \simeq \mc{C}^{(h)}(\ol{N}\cap H, L)^{\oplus m_d} \widehat{\otimes}_L \mc{C}^{(h)}(T_H, L).
\end{equation}

Now setting $B_h:= D_{p^{-r_h}}(T_H, L)$, where $r_h = \frac{1}{\varphi(p^h)} = \frac{1}{(p-1)p^{h-1}}$. The natural restriction map $B_{h+1} \to B_{h}$ necessarily factors through $D_{<p^{-r_h}}(T\cap H, L) \simeq D^{(h)}(T_H,L):= (C^{(h)}(T \cap H, L))'$. Now set
\begin{align*}
V_h &:= (\mc{C}^{(h)}(\ol{N}\cap H, L)^{\oplus m})'\widehat{\otimes}_L B_h\\
&\simeq (\mc{C}^{(h)}(\ol{N}\cap H, L)^{m} \widehat{\otimes}_L \mc{C}^{(h+1)}( T \cap H, L))'\widehat{\otimes}_{D^{(h)}(T_H, L)}B_h.
\end{align*}
By \eqref{eq: (h)AnalyticUntwisting} and the fact that $\varinjlim_{h}(V_H^h)^{N_H}= V^{N_H}$ we readily have $(V^{N_h})'\simeq \varprojlim_{h}V_h$. By definition, the $V_H$ are ONable Banach modules over $B_h$. To recover a weight space description, note that $T_H\simeq \Z_p^{\oplus M}$ for some $M \ge 0$, so the Amice transform implies that $B_h \simeq \Gamma(W_h'', \mc{O}_{\mc{W}_H})$ for some affinoid open $W_h'' \subset \mc{W}_H$, and the collection of $W_h''$ form an admissible cover. 

We now show the dual Hecke action $U_z^{\ast}$ both descends to the $V_h$ and has the relevant factorization properties in \cref{diag: UpOperatorFactorization}. Shrinking $H$ if necessary, we may assume both $H$ and $zHz^{-1}$ act on $\Pi_{c,\wt{\mf{m}}}^{\on{bdd}}$. Then we claim that if $v \in V_H^{(h)}$, then $U_zv \in V_{zHz^{-1}}^{(h)}$. Indeed, this fact would be true if $V$ had an action of the full group $\wt{G}$, rather than an action up to homotopy. However, note that the homotopies $ s, p, i , n$ are each continuous and $\wt{G}_0$-equivariant, and $()_{H}^{(h)}$, $()_{zHz^{-1}}^{(h)}$ are both functors, so the result follows from the corresponding fact on the cochain complex of Banach representations $A^{\bullet}$. This simple fact is all we need to push through \cite[Proposition 5.3]{BHSAnnalen}.
% The same Iwahori factorizations hold, but with a different exponent
% \[
% V_{zHz^{-1}}^{(h)}\simeq \mc{C}^{(h)}(zHz^{-1}, L)^{\oplus m_d [H_0: zHz^{-1}]},
% \]
% so that 
% \[
% V_{zHz^{-1}}^{(h), zN_0z^{-1}}\simeq \mc{C}^{(h)}(z(\ol{N}\cap H)z^{-1}, L)^{\oplus m_d [H_0: zHz^{-1}]}\widehat{\otimes}_L \mc{C}^{(h)}(T\cap H, L)\simeq V_{zHz^{-1}N_0}^{(h),N_0}
% \]
% $V_{zHz^{-1}}^{(h)}\cap V^{N_0} = (V_{zHz^{-1}N_0}^{(h)})^{N_0}$.

% Meanwhile, we have $zHz^{-1}N_0 \supset H$ with $c:=[zHz^{-1}N_0: H] = [z(\ol{N}\cap H)z^{-1}: \ol{N} \cap H]$, implying that 

% $\Pi_{zHz^{-1}N_0}^{(h)}\simeq \mc{C}^{(h)}(zHz^{-1}N_0)^{m_d/c}$.

% As a result, untwisting implies
% \[
% (\Pi_{zHz^{-1}N_0}^{(h)})^{N_0}\simeq \mc{C}^{(h)}(z(\ol{N} \cap H)z^{-1}, L)^{\oplus m_d/c}\widehat{\otimes}_L\mc{C}^{(h)}(T \cap H, L).
% \]

% Note that in terms of the these topological isomorphisms, the inclusion $ (V_{zHz^{-1}N_0})^{(h)})^{N_0} \to (V_{H}^{(h)})^{N_0}$ is induced by the restriction map 
% \[
% \on{res}_{h}:\mc{C}^{(h)}(z(\ol{N}\cap H)z^{-1}, L)^{\oplus m_d/c} \to \bigoplus_{n \in z(\ol{N}\cap H)z^{-1}/\ol{N}\cap H}\mc{C}^{(h)}(n\ol{N}\cap H, L)^{\oplus m_d/c}\simeq \mc{C}^{(h)}(\ol{N}\cap H, L)^{\oplus m_d},
% \]
% where we restrict a function  $f \in \mc{C}^{(h)}(z(\ol{N}\cap H)z^{-1}, L)$ to each coset $n\ol{N}\cap H$.

Namely, this fact along with \cref{factsaboutColmezDospinescu} implies the Hecke action satisfies
\begin{equation}
\label{eq: HeckeactionAnalyticity}
U_{z}(\mc{C}^{(h)}(\ol{N}\cap H, L)^{m} \widehat{\otimes}_L \mc{C}^{(h)}( T \cap H, L)) \subset \mc{C}^{(h-1)}(\ol{N}\cap H, L)^{m} \widehat{\otimes}_L \mc{C}^{(h)}(T \cap H, L).
\end{equation}
We then get the diagram \cref{diag: UpOperatorFactorization} by setting $z_{h}:= U_{z}^{\ast}$ induced on the first factor (which can be defined since all the homotopies and the group action on $A^{\bullet}$ are continuous), and $\alpha_h:= \alpha_h'\otimes 1$ where $\alpha_h'$ is dual to the Hecke action map \eqref{eq: HeckeactionAnalyticity}, and
$\beta_h = \beta_{h}'\otimes 1$ where $\beta_h'$ is simply the dual map of the natural inclusion 
\[
\mc{C}^{(h)}(\ol{N} \cap H, L)\widehat{\otimes}_L \mc{C}^{(h+1)}(T \cap H, L)\to \mc{C}^{(h+1)}(\ol{N} \cap H, L)\widehat{\otimes}_L \mc{C}^{(h+1)}(T \cap H, L).
\]
It is a general fact that the induced map $\beta_h'$ is compact, and so $\beta_h'$ is a $B_h$-compact map. The rest of the Proposition regarding $\mc{Z}_{\mc{N}}$ then follows exactly as in \cite[Lemme 3.10]{BHSAnnalen}.

We now turn to $\mc{M}_c^d$. First note that for an admissible cover $(U_h'')_h$ of $\mc{Z}_{\mc{N}}$, we have embeddings of $\mc{O}_{\mc{W}_H}(W_h')$-modules $\Gamma(U_h'', \mc{M}_c^d) \hookrightarrow \Gamma(U_h'', \mc{N})$, so that $\Gamma(U_h'', \mc{M}_c^d)$ is a submodule of a finite projective module over a noetherian integral domain, and is thus finitely generated and torsion-free over $\mc{O}_{\mc{W}_H}(W_h'')$. We also have an inclusion $\on{res}_H(\mc{Z}_{\mc{M}_c^d}) \subset \mc{Z}_{\mc{N}}$. In particular, either $\supp_{W_h''}\Gamma(U_h'', \mc{M}_c^d) = \emptyset$ or $W_h''$. As locally on $\mc{Z}_{\mc{N}}$, the projection to $\mc{W}_H$ is finite, it follows that $\dim \mc{Z}_{\mc{M}_c^d}  = \mc{Z}_{\mc{N}}$, and thus $\mc{Z}_{\mc{M}_c^d}$ is a union of components of $\mc{Z}_{\mc{N}}$, and is thus also a Fredholm hypersurface.

To get the statement over the whole $\mc{W}$, we again utilize the finite \'{e}tale cover $\on{res}_H:\mc{W} \to \mc{W}_H$. That is, we define the admissible cover $U_h'$ of $\mc{Z}_{\mc{M}_c^d}$ by $U_h':= (\on{res}_H\times \on{id}_{\G_m})^{-1}(U_h'')$. Moreover, if $W_h''$ is an admissible cover of $\mc{W}_H$ admitting diagrams as in \eqref{diag: UpOperatorFactorization}, then set $A_h:= L[T_0]\otimes_{L[T\cap H]\cap H]}B_h\simeq \Gamma(W_h', \mc{O}_{\mc{W}})$ for $W_h' = \on{res}_H^{-1}(W_h'')$. 
Moreover, since $A_h$ is a finite \'{e}tale $B_h$-algebra,  $\Gamma(U_h', \mc{M}_c^d)$'s is a direct factor of $L[T_0]\otimes_{L[T\cap H]}\Gamma(U_h'', \mc{M}_c^d)$, and thus a submodule of a finite free $A_h$-module.
\end{proof}

\begin{proposition}
\label{proposition: EquidimensionalAndOpennessOfWeightMap}
    There exists an admissible affinoid open cover $(U_i)_{i \in \Z_{\ge0}}$ of $\mc{E}_{c,\wt{\mf{m}}}^d$ such that $\mc{O}_{\mc{E}_c^d}(U_i)$ is a subalgebra of the endomorphism of a finitely generated torsion-free algebra over $\mc{O}_{\mc{W}}(\kappa(U_i))$.

    As a result, the eigenvariety $\mc{E}_{c,\wt{\mf{m}}}^d$ is equidimensional of dimension $\dim \mc{W}$ with no embedded components. Moreover, the image of any irreducible component of $\mc{E}_{c,\wt{\mf{m}}}^d$ along the weight map $\kappa: \mc{E}_{c,\wt{\mf{m}}}^d \to \mc{W}$ is Zariski open.
\end{proposition}
\begin{proof}
This proof is also essentially an adaptation of \cite[Proposition 3.11, Corollaire 3.12, Corollaire 3.13]{BHSAnnalen} to the setting at hand. The only difference is that instead locally for an admissible open $U_i$, we have an embedding $\mc{O}_{\mc{E}_{c,\wt{\mf{m}}}^d}(U_i) \hookrightarrow \End_{\mc{O}_{\mc{W}}(W_i)}(\Gamma(U_i',\mc{M}_c^d))$, where $\Gamma(U_i',\mc{N}_c^d)$ is torsion--free over $A:= \mc{O}_{\mc{W}}(W_i)$ by \cref{prop: FredholmHypersurface}. As a finite $A$-module, $\End_{\mc{O}_{\mc{W}}(W_i)}(\Gamma(U_i',\mc{M}_c^d))$ is torsion-free itself. To prove $\mc{E}_{c,\wt{\mf{m}}}^d$ is equidimensional it suffices to work affine locally and prove $B:=\mc{O}_{\mc{E}_c^d}(U_i)$ is equidimensional. Let $\mf{p}\subset B$ be any associated prime. Then we have an embedding of $B$-modules $B/\mf{p} \hookrightarrow B$, which implies $B/\mf{p}$ must also be $A$-torsion-free. In particular, the structure morphism $A \to B/\mf{p}$ is injective, and is thus an integral extension of Noetherian rings. Standard dimension theory then implies $\dim(B/\mf{p}) = \dim A = \dim \mc{W}$. Since this equality holds for any \textbf{minimal} prime, we have $B$ is equidimensional. Moreover, since this dimension is constant for any associated prime, the only associated primes are minimal primes, and thus there are no embedded components.

Fixing an irreducible component $C\subset \mc{E}_{c,\wt{\mf{m}}}^d$, its image in $\mc{Z}_{\mc{M}_c^d}$ is then irreducible of dimension $\dim \mc{W} = \dim Z_{\mc{M}_c^d}$, and is thus also irreducible. An irreducible component of a Fredholm hypersurface is also a Fredholm hypersurface by \cite[4.2.2]{Con99}, and the openness is then proven in \cite[6.4.4]{Che04}
\end{proof}

We now can prove our abstract density criterion. Let $f:\mc{E}_{c,\wt{\mf{m}}}^d \to \wh{T} \to \mc{W}\times \G_m$, be the structure map. Note that by definition $f$ factors through $\mc{Z}_{\mc{M}_{c}^d}\subset \mc{W}\times \G_m$.

\begin{proposition}
\label{proposition: AbstractDensityCriterion}
    If $S \subsetneq \mc{Z}_{\mc{M}_c^{d}}$ and $Y\subsetneq \mc{W}$ are closed subvarieties such that each irreducible component is positive codimension, then the set \[
        \mc{E}_{S,Y} = \{e\in \mc{E}_{c,\wt{\mf{m}}}^d: f(e) \notin S, \kappa(e) \notin Y, e \text{ locally dominant algebraic}, j_{2n}(e)\text{ numerically noncritical}\}
        \]
        is Zariski dense in $\mc{E}_{c,\wt{\mf{m}}}^d$, and accumulates at any point of algebraic weight.
\end{proposition}
\begin{proof}
    The proof falls into two steps:
    \begin{enumerate}
        \item The weights $w(E)$ accumulates at any algebraic point of $\mc{W}$;
        \item Use the first point to show Zariski-density of $E$.
    \end{enumerate}

    Before proving Step 1, we show Step 1 implies Step 2 (i.e. Zariski-density). It suffices to show any irreducible component $X \subset \mc{E}_{c,\wt{\mf{m}}}^d$ lies in the closure. By \cref{proposition: EquidimensionalAndOpennessOfWeightMap}, $\kappa(X)\subset \mc{W}$ is Zariski open, and thus there is a point $x\in X$ of algebraic weight. Now suppose $\overline{\mc{E}_{S,Y}}\not\supset X$. Then $\dim \overline{\mc{E}_{S,Y}\cap X} < \dim X = \dim \mc{W}$. But recall that locally on the eigenvariety $\mc{E}_{c,\wt{\mf{m}}}^d$ we have an affinoid open $U\ni x$ such that the induced map $\kappa_{U \cap X}: U\cap X \to \kappa(U\cap X)$ is finite surjective. 
As $\kappa_{U \cap X}$ is finite, the image of $\overline{\mc{E}_{S,Y} \cap U \cap X}$ is an analytic closed subset of $\kappa(U \cap X)$ containing $\kappa(\mc{E}_{S,Y} \cap U \cap X)$ which by the accumulation property is both nonempty and equal to $\kappa_X(U)$, which has $\dim \mc{W}$, contradicting the lack of density up top.

Now we prove Step 1: Let $x$ be a point of algebraic weight. Consider a connected affinoid open neighborhood $X\ni U\ni x$ contained in a fixed irreducible component $X \subset \mc{E}_{c,\wt{\mf{m}}}^{d}$, such that $\kappa: U \to \kappa(U)$ is finite surjective and $V:= \kappa(U) \subset \mc{W}$ is itself affinoid open. As $S$ is a nowhere dense closed subset of $X$, we have $S \cap U\subset U$ itself is an affinoid which is a union of irreducible proper closed subsets $Z_{i_1},\dots, Z_{i_k}$. 

Now following \cite[Proposition 3.19]{BHSAnnalen}, on the affinoid open $U$, the function $(y, \delta) \mapsto |j_{2n}(\delta)_i|_p$ is uniformly bounded in its valuation by some absolute constant $p^{C_{\on{max}}}$, where . Since $\kappa|_U$ is finite, the image of the closed subset $S \cap U = Z_{i_1}\cup \dots \cup Z_{i_k}$ remains closed and nowhere Zariski-dense, so the set of ``bad weights'' $B = \kappa(S \cap U) \cup (V \cap Y)$ is again proper closed in the irreducible $V$. Now we may consider the set $Q$ of characters $\delta: T(\mc{O}_{F^+,p}) \to \overline{\Q}_p^{\times}$ regular dominant algebraic of weight $-w_0\lambda$ such that for all $\tau: F^+ \hookrightarrow \overline{\Q}_p$
\begin{align*}
    &\lambda_{1,\tau} - \lambda_{2, \tau} > C_{\on{max}}\\
    &\lambda_{i,\tau} - \lambda_{i+1,\tau} > \lambda_{i - 1, \tau} - \lambda_{i,\tau} + (C_{\on{max}} + 1)
\end{align*}
for some fixed absolute constant $C_{\on{max}}$. By construction, every weight in $Q$ is numerically noncritical. Moreover, $Q \cap (V \setminus B)$ is Zariski dense in $V$, and thus by finiteness and equidimensionality $\kappa^{-1}(Q \cap (V \setminus B)) \subset \mc{E}_{S,Y}$ is Zariski-dense in $U$. This completes Step 1.
\end{proof}

\subsection{Density of cuspidal classical points}
We can now combine our work with \cref{corollary: SmallSupport} and \cref{lemma: FundamentalLongExactSequence} to prove the main theorem of this section.
\begin{definition}
    We call a point $x = (y,\delta)\in \mc{E}_{c,\wt{\mf{m}}}^d$ \textit{sufficiently interior classical} if:
    \begin{enumerate}
        \item $y$ lies in the support of the Hecke module $(J_B(\wt{\HH}^d(\wt{K}^p, L)_{\wt{\mf{m}}}^{\la}))'$;
        \item the character $\delta$ is dominant algebraic for $\tG$ and CTG in the sense of \cref{definition: CTG}.
    \end{enumerate}
\end{definition}
\begin{corollary}
\label{corollary: EcdClassicalPointsDense}
    $\mc{E}_{c,\wt{\mf{m}}}^d$ contains a Zariski-dense and accumulating set of points $(y,\delta)$ of numerically non-critical sufficiently interior classical points.
\end{corollary}
\begin{proof}
    We apply \cref{proposition: AbstractDensityCriterion} in the case $S = \supp_{\mc{W}\times \G_m}\HH^{d-1}(J_{B}(\Pi_{\partial, \wt{\mf{m}}}^{\la}))'$, and $Y = \mc{W}_{\on{non-CTG}}$. In particular, by \cref{corollary: SmallSupport} we have that for any affinoid open subsets $U_i'= \subset \mc{Z}_c^d:=\supp_{\mc{W}\times \G_m}\mc{M}_{c}^d$, we have $\pi_{\mc{W}}(S \cap U_i'), \pi_{\mc{W}}(Y \cap U_i') \subset \Omega:= \pi_{\mc{W}}(U_i')$ are proper Zariski--closed subsets in an affinoid, and are thus finite unions of irreducible components, each of which has dimension $< \dim \Omega = \dim\mc{W}$. Thus, \cref{proposition: AbstractDensityCriterion} applies and so the set $\mc{E}_{S,Y}$ of points $x = (y, \delta)$ with $\delta$ locally dominant algebraic, $j_{2n}(\delta)$ numerically noncritical whose weight also avoids the $S$ and $Y$. By \cref{lemma: FundamentalLongExactSequence} if $x = (y,\delta) \in \mc{E}_{S,Y}$, then $(\mc{M}_{\partial}^{d-1})_{x} = 0$ by Nakayama's lemma (working locally around $f(y)$ so that all the modules are finitely generated) and the definition of $S$, so $(\mc{M}_c^d)_{x} \neq 0 \implies (\mc{M}^d)_x \neq 0$, and thus for any $x \in \mc{E}_{S,Y}$, we get a corresponding point $x \in \mc{E}_{\wt{\mf{m}}}^{d}$ on the middle degree eigenvariety. But then using numerical noncriticality, the classicality criterion \cref{lemma: SmallSlopeClassicalForUsualJacquetFunctor} implies $x$ corresponds to a Hecke eigensystem in some $\HH^d(\wt{K}^p, \mathsf{L}(\lambda))_{\wt{\mf{m}}}$ for $\lambda$ a dominant weight for $\tG$. By the CTG assumption on the weight $\lambda$ (i.e. $Y$), $x$ corresponds to a \textit{cuspidal} automorphic representation for $\wt{G}(\A_{F^+})$.
\end{proof}

\section{Trianguline Galois representations}
\label{Section6: Trianguline}

We can now combine the density of classical (cuspidal) points in $\mc{E}_{c, \wt{\mf{m}}}^{d}$ with degree shifting to derive our main theorem. First, recall that for $K, L/\Q_p$ finite extensions a continuous local Galois representation $r: \Gal_{K} \to \GL_n(L)$ is said to be \textbf{trianguline} of parameter $\delta: (K^{\times})^n \to L^{\times}$ if the associated rigid $(\varphi, \Gamma_K)$-module $D_{\on{rig}}^{\dagger}(r)$ has an increasing filtration by sub-$(\varphi, \Gamma_K)$ modules with rank $1$ (ascending) graded pieces $\mc{R}_{L,K}(\delta_1), \dots, \mc{R}_{L,K}(\delta_n)$. 

Now we need to discuss Galois representations associated to classical automorphic forms for the unitary group $\wt{G}$, along with their local-global compatibility results. For shorthand, if $r: \Gal_F \to \GL_n(L)$ is a continuous representation and $\nu \in S_{\ell}(F)$ is a place, we let $r_{\nu}:= r|_{\Gal_{F_{\nu}}}$.
\begin{theorem}
\label{theorem: ShinGaloisreps}
    Assume $F$ contains an imaginary quadratic field $F_0$, $p$ is a split prime in $F_0$, and $\pi$ is a cuspidal automorphic representation of $\tG(\A_{F^+})$ which $\xi$-cohomological for $\xi$ an irreducible algebraic representation of $(\Res_{F^+/\Q}\tG)_{\C}$. Then for any isomorphism $\iota: \overline{\Q}_p \simeq \C$ there is an associated Galois representation $r_{\pi, \iota}: \Gal_{F} \to \GL_{2n}(\overline{\Q}_p)$ satisfying:
    \begin{enumerate}
        \item For any $\ell \neq p$ which is unramified in $F$ and above which $\pi$ is unramified and for all $\nu \mid \ell$ in $F$, we have $r_{\pi, \iota,\nu}$ is unramified and the characteristic polynomial of $r_{\pi, \iota}(\Frob_{\nu})$ is equal to the image of $\widetilde{P}_{\nu}(X)$ in $\overline{\Q}_p[X]$ coming from the base change of $\iota^{-1}(\pi_{\ol{\nu}})$, where $\nu$ lies above $\ol{\nu}\in S_p(F^+)$ .
        \item For any $\nu \mid p$ in $F$ we have $r_{\pi, \iota,\nu}$ is de Rham with Hodge--Tate weights for each $\tau: F \hookrightarrow \overline{\Q}_p$ given by
        \[
        \on{HT}_{\tau}(r_{\pi, \iota,\nu}) = \{\wt{\lambda}_{\tau, 1} + 2n - 1, \wt{\lambda}_{\tau, 2} + 2n - 2, \dots, \wt{\lambda}_{\tau, 2n}\},
        \]
        where $\wt{\lambda} \in (\Z_{\on{dom}}^{2n})^{\Hom(F, \overline{\Q}_p)}$ is the highest weight of the representation $\iota^{-1}(\xi \otimes \xi)^{\vee}$ of $\GL_{2n}/\ol{\Q}_p$.
    \end{enumerate}
    Suppose moreover that for each $\ol{\nu}|p$ we have $\iota^{-1}\pi_{\ol{\nu}} \simeq \Ind_{\ol{B}_{2n}(F_{\nu})}^{\GL_{2n}(F_{\nu})}\theta_{\ol{\nu}}$ is an unramified principal series, and set $\theta:= \prod_{\ol{\nu}}\theta_{\ol{\nu}}$. Then we have:
    \begin{enumerate}[label = (\alph*)]
        \item $r_{\pi, \iota,\nu}$ is crystalline, and the Frobenius operator on $D_{\on{crys}}(r_{\pi, \iota,\nu})$ is semisimple with eigenvalues
        \[
        \varphi_{\nu, i} = \theta_{\ol{\nu},i}(\varpi_{\nu})|\varpi_{\nu}|_{F_{\nu}}^{-i+1}.
        \]
        \item Moreover if the refinement $j_{2n}(\theta \cdot \wt{\lambda}^{\vee})_{\ol{\nu}}$ is numerically noncritical in the sense of \cref{definition: NumericallyNoncritical} then $r_{\pi, \iota}|_{\Gal_{F_{\nu}}}$ is trianguline with parameter $j_{2n}(\theta\cdot \wt{\lambda}^{\vee})_{\ol{\nu}}$.
    \end{enumerate}
\end{theorem}
\begin{proof}
    The theorem is essentially \cite[Thm. 2.3.3]{10AuthorPaper} (which in turn is an adaptation of \cite{Shin14}), once we note several consequences. Recall (a consequence of) the refined statement of local-global compatibility at finite places due to Caraiani (see \cite[Theorem 2.1.9 (3)]{CN2023}): if $\ell$ is prime which splits in $F_0$, then for each $\nu \in S_p(F)$ in $F$ lying above $\ol{\nu}\in S_p(F^+)$, we have an isomorphism
    \[
    \WD(r_{\pi, \iota}|_{\Gal_{F_{\nu}}})^{\on{F}-ss} \simeq \iota^{-1}\rec_{F_\nu}(\pi_{\ol{\nu}}\otimes |\det|^{(1-2n)/2}).
    \]
    Plugging in $\pi_{\ol{\nu}} = \Ind_{\ol{B}_{2n}(F_{\nu})}^{\GL_{2n}(F_{\nu})}\theta_{\ol{\nu}}$ to the right side yields
    
    \begin{align*}
    \rec_{F_\nu}(\pi_{\ol{\nu}}\otimes |\det|^{(1-2n)/2})&\simeq \theta_{\ol{\nu}}\delta_{B}^{1/2}| \cdot|_{F_{\nu}}^{(1-2n)/2}\circ \Art_{F_{\nu}}^{-1}\\
    &= (\theta_{\ol{\nu},1}|\cdot|_{F_{\nu}}^{(2n-1)/2}, \theta_{\ol{\nu},2}|\cdot|_{F_{\nu}}^{(2n-3)/2},\dots, \theta_{\ol{\nu}, 2n} |\cdot|_{F_{\nu}}^{(1-2n)/2})|\cdot|_{F_{\nu}}^{(1-2n)/2}\circ \Art_{F_{\nu}}^{-1}\\
    &= (\theta_{\ol{\nu},1}\circ \Art_{F_{\nu}}^{-1},\theta_{\ol{\nu},2}|\cdot|_{F_{\nu}}^{-1}\circ \Art_{F_{\nu}}^{-1},\dots, \theta_{\ol{\nu}, 2n}|\cdot|_{F_{\nu}}^{1-2n}\circ \Art_{F_{\nu}}^{-1}),
    \end{align*}
    so the eigenvalues for $\Frob_{\nu}$ are $\theta_{\ol{\nu},1}(\varpi_{\nu}), \dots, \theta_{\ol{\nu},2n}(\varpi_{\nu})|\varpi_{\nu}|_{F_{\nu}}^{1-2n}$.

    Then applying \cite[Lemma 2.8]{NT21} implies $r_{\pi,\iota}|_{\Gal_{F_{\nu}}}$ is trianguline with parameter $\delta$ defined by
    \[
    \delta_{\nu, i}(x) = \theta_{\ol{\nu},i}(x)|x|_{F_{\nu}}^{1-i}\cdot \prod_{\tau}\tau(x)^{-\wt{\lambda}_{\tau, n-i + 1} - i + 1}
    \]
    Now recalling the cyclotomic character $\eps_{\on{cyc}} \circ \Art_{F_{\nu}} = N_{F_{\nu}/\Q_p}(x)|x|_{F_{\nu}}$, we then have $N_{F_{\nu}/\Q_p}(x)^{1-i} = \prod_{\tau}\tau(x)^{1-i} = (\eps_{\cyc}\circ \Art_{F_{\nu}}(x))^{1-i}|x|_{F_{\nu}}^{i-1}$. Thus $\delta_{\nu, i} = \theta_{\ol{\nu}, i}\prod_{\tau}\tau(x)^{\wt{\lambda}_{\tau, n-i + 1}}\cdot (\eps_{\cyc}\circ \Art_{F_{\nu}})^{1-i} = j_{2n, i}(\theta\cdot \lambda^{\vee})_{\ol{\nu}}$. Thus, $r_{\pi, \iota}|_{\Gal_{F_{\nu}}}$ is in fact trianguline of parameter $j_{2n}(\theta\cdot \lambda^{\vee})_{\ol{\nu}}$, as required.
\end{proof}

We recall certain deformation spaces of trianguline (pseudo)representations. Fix a continuous residual representation $\ol{r}: \Gal_{K} \to \GL_n(k_L)$, let $R_{\ol{r}}^{\square}$ be the associated framed Galois deformation ring, and let $X^{\square}(\ol{r})$ be the rigid generic fibre of $\Spf(R_{\ol{r}}^{\square})$. Now let $\mc{T}_L^n:= \widehat{T_n(K)}_L$ be the rigid space of continuous characters of $(K^{\times})^{n}$. Let $\mc{T}_{\on{reg}}^{n} \subset \mc{T}_L^n$ be the Zariski dense open subset of characters $\delta = (\delta_1,\dots, \delta_n)$ such that $\delta_{i}/\delta_j \neq z^{-\mathbf{k}}, (\varepsilon_{\on{cyc}}\circ \Art_{K})\cdot z^{\mathbf{k}}$ for all $i \neq j$ and $\mathbf{k} \in \Z_{\ge 0}^{\Hom(K, \ol{\Q}_p)}$. Now let $U_{\tri}(\rbar) = \{(r,\delta): X^{\square}(\ol{r}) \times \mc{T}_{\on{reg}}^{n}: r \text{ is trianguline of parameter $\delta$}\}$. We then recall the \textit{trianguline variety} $X_{\tri}^{\square}(\ol{r})$ is defined to be the Zariski closure of $U_{\tri}(\rbar)$ in $X^{\square}(\rbar)\times \mc{T}_L^n$. We could also make a similar definition taking representations at all places $\nu \mid p$ at once. Basic structural results (see \cite[Theoreme 2.6]{BHSAnnalen}) tell us that $X_{\tri}^{\square}(\rbar)$ is quasi-Stein rigid space, equidimensional of dimension $n^2 + \frac{n(n+1)}{2}[K:\Q_p]$ and $U_{\on{tri}}(\rbar)$ is Zariski open, dense, and smooth, and the natural map $\omega:U_{\on{tri}}(\rbar) \to \mc{W}_L$ is also smooth. 

Although $X_{\tri}^{\square}(\rbar)$ is defined by a closure, we still know the following about all the points in this space: if $x = (r,\delta) \in X_{\tri}^{\square}(\rbar)$ then by \cite[Theorem 6.3.13]{KPX14} $r$ is trianguline of some parameter $\delta' = (\delta_1', \dots, \delta_n')$ satisfying $\delta_i^{-1}\delta_i': K^{\times} \to \ol{\Q}_p^{\times}$ is an algebraic character. In particular, all representations on $X_{\tri}^\square(\rbar)$ are trianguline, but need not be trianguline of parameter the given character $\delta$. Note, however, that the weights of $\delta$ and $\delta'$ must nevertheless coincide (see \cite[Proposition 2.9]{BHSAnnalen}). 

\begin{proposition}
    \label{prop: CrystallineGaloisRepsAssociatedToClassicalPoints}
    Keep the assumptions of \cref{theorem: ShinGaloisreps}.
    If $z = (y, \delta = \delta_{\on{sm}}\lambda^{\vee}) \in \mc{E}_{c,\wt{\mf{m}}}^d$ is a numerically noncritical sufficiently interior point of weight $\lambda$ such that for all $\nu |p$ of $F$ the representation $\Ind_{\ol{B}_{2n}(F_{\nu})}^{\GL_{2n}(F_{\nu})}(\delta_{\on{sm}, \ol{\nu}})$ is absolutely irreducible, there is an associated Galois representation $r_{z}: \Gal_F \to \GL_{2n}(\ol{\Q}_p)$ matching Frobenius and Hecke eigenvalues at unramified places, such that $r_z$ is crystalline at $\nu | p$ with Hodge--Tate weights given by $\on{wt}_{\tau}(j_{2n}(\delta_{z})_{\ol{\nu}})$ and Frobenius eigenvalues given by $\delta_{\nu, i}(\varpi_{\nu})\prod_{\tau\in \Hom(F_{\nu}, \ol{\Q}_p)}\tau(\varpi_{\nu})^{-\lambda_{\tau, 2n - i + 1} - i + 1}$, and in particular $(r_z, j_{2n}(\delta_z)_{\ol{\nu}}) \in U_{\tri}(\ol{\rho}_{\wt{\mf{m}}}|_{\Gal_{F_{\nu}}})$.

    Consequently, there is a pseudorepresentation $D: \Gal_F \to \mc{O}(\mc{E}_{c, \wt{\mf{m}}}^d)$ such that for any numerically noncritical sufficiently interior classical $z \in \mc{E}_{c,\wt{\mf{m}}}^d(\ol{\Q}_p)$, the specialization $D_{z}: G_F \to \mc{O}(\mc{E}_{c, \wt{\mf{m}}}^d) \to \ol{\Q}_p$ is equal to $\on{tr} r_{z}$.
\end{proposition}
\begin{proof}

    By definition and the discussion before \cref{definition: NumericallyNoncritical} (taking into account that usual Jacquet module $J_B(\Ind_{\ol{B}_{2n}(F_{\nu})}^{\GL_{2n}(F_{\nu})}\theta_{\ol{\nu}})$ admits $\delta_B\cdot \theta_{\ol{\nu}}$ as a subquotient), $z$ takes the form $z = (\psi_{\pi},\delta_{\on{sm}}\cdot \lambda^{\vee})$, where $\pi$ is a cuspidal automorphic representation of $\wt{G}(\A_F^+)$ of weight $\lambda$ such that $\pi_{\ol{\nu}}\simeq \Ind_{\ol{B}_{2n}(F_{\nu})}^{\GL_{2n}(F_{\nu})}\delta_{\on{sm}, \ol{\nu}}$. Thus, the first part of the proposition follows from \cref{theorem: ShinGaloisreps}.
    Now let $\mO(\mc{E}_{c,\wt{\mf{m}}}^d)^{\le 1}\subset \mO(\mc{E}_{c,\wt{\mf{m}}}^d)$ be the subring of bounded elements. 
    The existence of the determinant $D$ is a standard argument (see \cite[Proposition 7.1.1]{Che04}), which is a consequence of:
    the Zariski density of noncritical sufficiently interior classical points, which implies the map $\mc{O}(\mc{E}_c^d)^{\le 1} \to \prod_{z \in Z}\C_p$ is injective, the fact that $\mc{O}(\mc{E}_c^d)^{\le 1}$ is compact, and the integrality of Hecke operators (i.e. the Hecke action $\psi: \wt{\T}^S \to \mc{O}(\mc{E}_c^d)$ factors through $\mc{O}(\mc{E}_c^d)^{\le 1}$). These facts imply that the trace of Frobenius at unramified primes of the product determinant $D_{\prod Z} := \prod_{z} \det r_z$ factors through $\mc{O}(\mc{E}_{c,\wt{\mf{m}}}^d)^{\le 1}$. \cite[Example 2.32]{Che14} then implies we can glue determinants, and $D_{\prod Z}$ factors through the desired determinant.
\end{proof}

The existence of the determinant $D: \Gal_F \to \mc{O}(\mc{E}_{c, \wt{\mf{m}}}^d)$ then yields a closed embedding $\iota_{c, \tG}': \mc{E}_{c, \wt{\mf{m}}}^{d}\hookrightarrow X_{\on{ps}}(\ol{\rho}_{\wt{\mf{m}}})^{2n} \times \widehat{T(F_p^+)}$. To keep in line with \cref{prop: CrystallineGaloisRepsAssociatedToClassicalPoints} and match the weights, we define $\iota_{c, \tG}:= (\on{id}\times j_{2n})\circ \iota_{c,\tG}'$.
We would like to interpolate from classical points to say more about the image of $\iota_{c,\tG}$. However, we run into the issue that $\mc{E}_{c,\wt{\mf{m}}}^{d}$ only carries a \textit{pseudorepresentation}, rather than an actual representation. Along these lines, we can similarly define the \textit{trianguline pseudodeformation space} $X_{\on{ps}, \tri}(\ol{\rho}_{\wt{\mf{m}},\nu})$ to be the Zariski closure of the image of $\bigcup_{\ol{r}: \on{tr} \ol{r} = \on{tr} (\ol{\rho}_{\wt{\mf{m}}})_{\nu}}U_{\tri}(\rbar)$ inside $X_{\ps}(\ol{\rho}_{\wt{\mf{m}},\nu})\times \wh{T(F_p^+)}$. A version of this space was first studied by Hellmann \cite{He12} (who called this a \textit{finite slope space}), and in the case of $K = \Q_p$ and $\ol{r}: \Gal_{K} \to \ol{\F}_p$ a local residual pseudorepresentation he showed $X_{\on{ps}, \tri}(\rbar)$ is equidimensional of dimension $1 + \frac{n(n+1)}{2}$. A refined consequence is the following:
\begin{corollary}
    For all $\nu \in S_p(F)$, the composition $\on{res}_{\nu}\circ \iota_{c, \tG}: \mc{E}_{c, \wt{\mf{m}}}^d \to X_{ps}(\ol{\rho}_{\wt{\mf{m}}}) \times \wh{T(F_p^+)}\to X_{ps}(\ol{\rho}_{\wt{\mf{m}},\nu}) \times \wh{T(F_p^+)}$ factors through $X_{\ps, \tri}(\ol{\rho}_{\wt{\mf{m}},\nu})$.
\end{corollary}
\begin{proof}
    This is a consequence of \cref{prop: CrystallineGaloisRepsAssociatedToClassicalPoints}, the Zariski-density of classical points, the definition of the trianguline pseudodeformation space, and the fact that $\iota_{c,\tG}$ is a closed embedding. 
\end{proof}

For the sake of deducing results for $G = \GL_n/F$ (for example, pointwise triangulinity), however, it is better to work with framed deformation spaces. So we need to pass from pseudorepresentations, to representations. We will do this by using \cite[Lemma 7.8.11]{BC09} by Bella\"{i}che--Chenevier, and the following lemma.

\begin{lemma}
\label{lemma: LocalZariskiDensity}
    Let $X$ be a rigid space equipped with increasing admissible affinoid open cover $X = \bigcup_{n= 1}^{\infty} U_n$ and suppose $Z \subset X$ is a Zariski-dense and accumulating subset. Let $V_m:= \ol{Z \cup U_m} \subset U_m$ denote the Zariski-closure of the subset $Z \cap U_m$. Then for all $j \ge 1$, there exists $j' \ge j$ such that $U_j \subseteq V_{j'}$.
\end{lemma}
\begin{proof}
    This claim follows from the proof of \cite[Lemma 5.9]{CHJ17}\footnote{A separate proof was communicated by Brian Conrad. We thank him for also figuring out the shape of the lemma statement.}. The only difference there is that when $X$ is normal (or smooth, as assumed in \textit{loc. cit.}), the subsets $\ol{Z \cap U_m}$ are a union of \textit{connected} components of $V_m$, and are thus also affinoid open in $X$, implying the $V_m$ themselves form an affinoid open cover.
\end{proof}
\begin{lemma}
\label{lemma: TriangulineInsensitiveToSemisimple}
    Let $r: \Gal_{F} \to \GL_m(\overline{\Q}_p)$ be a continuous representation. Then $r_{\nu}$ is trianguline if and only if $(r^{\on{ss}})_{\nu}$ is trianguline. Moreover if $r_{\nu}$ is trianguline of parameter $\delta$, then $(r^{\on{ss}})_{\nu}$ is also trianguline of parameter $\delta$.
\end{lemma}
\begin{proof}
    This result follows from the definition of triangulinity. Note of course if $(r^{\on{ss}})_{\nu}$ is trianguline of parameter $\delta$, $r_{\nu}$ need not have parameter $\delta$.
\end{proof}

\begin{proposition}
\label{proposition : PointwiseTriangulineforEcd}
    If $x = (y, \delta) \in \mc{E}_{c, \wt{\mf{m}}}^d(\ol{\Q}_p)$ and $t = \iota_{c, \tG}(x) = \on{tr} \rho,$ is the associated determinant for $\rho$ semisimple, then $\rho|_{\Gal_{F_{\nu}}}$ is a trianguline Galois representation with $\on{HT}_{\tau}(\rho_{\nu}) = \on{wt}_{\tau}(j_{2n}(\delta)_{\ol{\nu}})$ and parameter $j_{2n}(\delta)_{\ol{\nu}}\cdot \delta'$, for $\delta'$ an algebraic character.
\end{proposition}
\begin{proof}
    The idea is to reduce to the case of Galois representations. Since a densely pointwise strictly trianguline $(\varphi, \Gamma_K)$-module over a rigid space $X$ is globally pointwise trianguline by \cite[6.3.13]{KPX14}, the corresponding statement is true upon replacing $\mc{E}_{c, \wt{\mf{m}}}^d$ with, say, $X_{\tri}^{\square}(\ol{\rho}_{\wt{\mf{m}},\nu})$. As such a statement is less clear for $X_{\on{ps},\tri}(\ol{\rho}_{\wt{\mf{m}},\nu})$ we instead work directly on the eigenvariety. Write $\mc{E}_{c, \wt{\mf{m}}}^d = \cup_{m}U_m$ as an increasing admissible cover of affinoid open subsets. Likewise let $Z \subset \mc{E}_{c,\mf{m}}^d$ be the set of numerically noncritical, sufficiently interior classical points, which is Zariski-dense and accumulating by \cref{corollary: EcdClassicalPointsDense}.
    If $x = (t, \delta) \in \mc{E}_{c,\wt{\mf{m}}}^{d}(\overline{\Q}_p)\subset X_{ps, \on{tri}}(\ol{\rho}_{\wt{\mf{m}}})\times \widehat{T(F_p^+)}$ and $U_{m'}\ni x$, then recall from \cref{lemma: LocalZariskiDensity} that for some index $m \ge m'$ we have that $U_{m'} \subset V_{m}:= \ol{U_{m} \cap Z}$, which is still an affinoid. The map $j_m:V_m \to X_{ps, \tri}(\ol{\rho}_{\wt{\mf{m}}})$ yields a rank $2n$ pseudorepresentation $T:\Gal_{F} \to j_m^{\ast}\mc{O}_{\mc{E}_{c,\wt{\mf{m}}}^d}(V_m)$. Then by \cite[Lemma 7.8.11]{BC09} there exists a reduced, separated, quasi-compact rigid space $Y$ equipped with a proper, generically finite, surjective map $f: Y \to V_{m}$ such that there exists a locally free coherent $\mc{O}_Y$-module $\mc{M}$ of rank $2n$ with a continuous $\Gal_{F}$-action $R: \Gal_{F} \to \GL_{\mc{O}(Y)}(\mc{M})$, such that the trace $\on{tr} R = f^{\ast}T$, and for any $y \in Y$ the specialization $(M_{Y,y}\otimes_{\mc{O}_{Y,y}} \overline{k(y)})^{ss}\simeq \rho_{f(y)}$, the semisimple representation corresponding to the point $f(y)\in V_m$. Moreover, in a Zariski-dense open subset of $Y$, the fibre $M_{Y,y}\otimes_{\mO_{Y,y}}\ol{k(y)}$ is automatically semisimple!

    By the general theory in \cite[Theorem 2.2.17]{KPX14} we can attach to $M$ a $(\varphi, \Gamma_{F_{\nu}})$-module $D_{\on{rig}}^{\dagger}(\mc{M})$ over the relative Robba ring $\mc{R}_{Y,F_{\nu}}$.  % Assuming the affinoid $V_{m'}$ we started with is connected (pass to the connected component containing $x$),  we may assume that $Y$ is connected itself.
    Looking at the Zariski-closure $Y':= \ol{f^{-1}(Z \cap U_m)} \subset Y$, we note the composite map $f: Y' \to U_{m}$ remains surjective (in particular, $x$ is in the image!) since $f$ is proper and thus maps closed analytic subsets to closed analytic subsets by \cite[Chapter 9, Proposition 2]{BGR84}. Thus renaming $Y'$ as $Y$, we may assume that the preimage of $Z\cap V_{m}$ is Zariski-dense in $Y$.

    Now the natural map $\mc{E}_{c, \wt{\mf{m}}}^d \to \widehat{T(F_p^+)}$, yields a morphism $Y \to V_m \xrightarrow{\pi_{T(F_p^+)}\circ \iota_{c,\wt{G}}} \widehat{T(F_p^+)}$, and thus produces universal characters \[
    \chi = \begin{pmatrix}
        \chi_1&\\
              & \ddots \\
              & & \chi_{2n}\\
    \end{pmatrix}: T(F_p^+) \to \mc{O}(Y)^{\times}
    \]
    By \cref{prop: CrystallineGaloisRepsAssociatedToClassicalPoints} the specializations $(D_{\on{rig}}^{\dagger}(\mc{M})_{Y,y}\otimes \overline{k(y)})^{\on{ss}}$ are strictly trianguline of parameter $\chi_{y}$ for all $y \in f^{-1}(Z\cap U_{m})$, a Zariski dense subset. Moreover, restricting to the Zariski-dense \textit{open} locus where the fibres $\mc{M}_{Y,y}\otimes \overline{k(y)}$ are semisimple, we then get $D_{\on{rig}}^{\dagger}(\mc{M})_{Y,y}$ is densely pointwise strictly trianguline with parameter $\chi_y$. Thus we can apply \cite[Theorem 6.3.13]{KPX14} to get that $D_{\on{rig}}^{\dagger}(\mc{M})$ is pointwise trianguline at \textit{all} points $x$, with parameter given by $\chi_x \cdot \delta_x$ for $\delta_x$ some algebraic character depending on $x$, such that $\on{wt}_{\tau}(\chi_x \cdot \delta_x) = \on{wt}_{\tau}(\chi_x)$ for all $\tau: F_{\nu}\hookrightarrow \ol{\Q}_p$. Then \cref{lemma: TriangulineInsensitiveToSemisimple} implies for all $x \in V_m$ that $\rho_{x,\nu}$ is trianguline of parameter $\chi_x\cdot \delta_x$, as required.

\end{proof}

We can now return to our main interest, $G = \Res_{F/F^+}\GL_n$.
The pointwise triangulinity of eigenvarieties for $G$ follows essentially immediately from \cref{proposition : PointwiseTriangulineforEcd} and \cref{theorem: Embeddings of eigenvarieties}, but to derive statements on the triangulation or weights we need to introduce an extra degree of freedom in our framework.

Keep the assumptions of \cref{theorem: ShinGaloisreps}. Let $\overline{\nu} \in S_p(F^{+})$ and $\nu, \nu^c \in S_p(F)$ be the two places lying over $\overline{\nu}$ (recall we are assuming $p$ splits in the imaginary quadratic extension contained in $F$). 
Then using \cref{theorem: Embeddings of eigenvarieties} we have all $i$ and $w\in W^{\oP}$ such that $\ell(w) = d - 1 = i$, we have a diagram as follows: 

\begin{equation}
\label{eq: TriangulineDiagram}
\begin{tikzcd}[row sep=0.75em, column sep = 0.75em]
    && \mc{E}_c^d(\wt{K}^p)_{\wt{\mf{m}}}\arrow[rrrr, "\iota_{c,\tG}"]\ar[rrrrrrdd, dotted] &&&&
    X_{ps}(\ol{\rho}_{\wt{\mf{m}}})\times \widehat{T(F_p^+)} \ar[rrrr]&&&& X_{ps}(\ol{\rho}_{\wt{\mf{m}},\nu})\times \widehat{T(F_{\ol{\nu}}^+)}
    \\
    \\
    &&&&&&&&X_{ps, \tri}(\ol{\rho}_{\wt{\mf{m}},\nu})\ar[rruu, hook]
    \\
    \\
    && \mc{E}^i(K^p)_{\mf{m}} \ar[uuuu, "j_w"]\arrow[rrrr,"\iota_{\mf{m}}^i"]\ar[rrrrrruu, dotted] &&&& X^{\square}(\ol{\rho}_{\mf{m}})\times \widehat{T_n(F_p)}\ar[rrrr, "{\on{res}_{\nu,\nu^c}}"]\ar[uuuu, "d_{ps,w}"]&&&&(X^{\square}(\ol{\rho}_{\mf{m},\nu})\times \wh{T_n(F_{\nu})})\times (X^{\square}(\ol{\rho}_{\mf{m},\nu^c})\times \wh{T_n(F_{\nu^c})})\ar[uuuu, "d_{ps, w,\ol{\nu}}"]
    \end{tikzcd}
\end{equation}
In the diagram, the map $\on{res}$ is induced by the natural restriction maps $X^{\square, n}(\ol{\rho}_{\mf{m}}) \to X^{\square, n}(\ol{\rho}_{\mf{m},\nu}) \times X^{\square, n}(\ol{\rho}_{\mf{m},\nu^c})$.
Let us spell out more explicitly the maps $d_{ps,w}, d_{ps,w,\ol{\nu}}$.
Recall by construction, the embedding $j_w$ sends a point $(x,\delta = (\delta_{\nu})_{\nu \in S_p(F)}) \in (\Spf \T^S(K^p)_{\mf{m}})^{\on{rig}} \times \widehat{T_n(F_p)}$ to $(x \circ \mc{S}, j_w(\delta))\in (\Spf\wt{\T}^{S}(\wt{K}^p)_{\wt{\mf{m}}})^{\on{rig}}\times \widehat{T(F_p^+)}$, where  \[
(j_w(\delta))_{\ol{\nu}} = w^{-1}((\delta_{\nu},-w_0\delta_{\nu^c}) + \rho_{\cyc,\ol{\nu}}^{\tG}) - \rho_{\cyc,\ol{\nu}}^{\tG}
\]
As we have shifted the two maps from the eigenvarieties to the deformation spaces via $\iota_{c,\wt{G}}(x, \delta) =(\on{tr}(\rho_x), \delta + \rho_{\cyc}^{\tG})$ and $\iota_{\mf{m}}^i(x,\delta) = (r_x, \delta + \rho_{\cyc}^{G})$,  $d_{ps,w}$ is \textit{defined} to be the map 

\begin{align*}
(r, \delta) &\mapsto (\on{tr}(r \oplus r^{c,\vee}(1-2n)), (w^{-1}((\delta_{\nu} -  \rho_{\cyc,\nu}^{G}, -w_0(\delta_{\nu^c} - \rho_{\cyc,\nu^c}^{G})) + \rho_{\cyc}^{\tG}))_{\ol{\nu}})\\
&= (\on{tr}(r \oplus r^{c,\vee}(1-2n)), (w^{-1}(\delta_{\nu}, -w_0\delta_{\nu^c}(1-2n)))_{\ol{\nu}},
\end{align*}
where the second line follows from the fact that $\rho_{\cyc,\ol{\nu}}^{\tG} = (\rho_{\cyc,\nu}^{G},  \rho_{\cyc,\nu}^{G}(-n))$.
We have some flexibility with the transfer map coming from twisting by characters, as in \cite[p. 
 37]{10AuthorPaper}. Let $\chi: \Gal_F \to \mO_{L}^{\times}$ be a continuous character such that $\chi \circ \Art_{F_{\nu}}$ is trivial on $\det(K_{\nu})$ for each $\nu \notin S$ of $F$. Then there is an isomorphism $f_{\chi}: \mc{H}(G^S, K^S) \to \mc{H}(G^S, K^S)$ given by the formula $f_{\chi}(f)(g) = \chi(\Art_{F}(\det(g)))^{-1}f(g).$ Then given $\mf{m}\subset \T^{S}$ a maximal ideal, we define $\mf{m}(\chi) = f_{\chi}(\mf{m})$. Moreover, let $\wt{\mf{m}}(\chi) := \mc{S}^{\ast}(\mf{m}(\chi))$. 
 The next lemma tells us the effect of such a twisting on eigenvarieties. Let $\chi_p: \GL_n(F_p) \to \mO^\times$ be the character defined by $(g_{\nu})_{\nu \in S_p(F)} \mapsto \prod_{\nu \in S_p(F)}\chi(\Art_{F_{\nu}}(\det(g_{\nu})))$.
 \begin{lemma}
 \label{lemma: TwistByCharacters}
     There is a natural isomorphism of essentially admissible representations of $T(F_p^+)$
     \[
     R_{\chi}:\HH^i(J_B(\Pi(K^p)_{\mf{m}}^{\la})) \to \HH^i(J_B(\Pi(K^p)_{\mf{m}(\chi)}^{\la}))\otimes_{L} \chi_{p},
     \]
     which intertwines the actions of $\mc{H}(G^S,K^S)\otimes_{\Z}\mO$ on the left with the action of $\mc{H}(G^S,K^S)\otimes_{\Z}\mO$ by precomposing with $f_{\chi}$ on the right.

     As a result if $\mf{m} \subset \T^S(K^p)$ is a non-Eisenstein maximal ideal and $\chi \mod \mf{m}_L$ is trivial then we have an isomorphism of rigid varieties $\mc{E}^i(K^p)_{\mf{m}} \simeq \mc{E}^i(K^p)_{\mf{m}(\chi)}$ fitting into a diagram 
     \[
     \begin{tikzcd}
         \mc{E}^i(K^p)_{\mf{m}} \ar[r]\ar[d, "\on{tw}_{\chi}"]&X^{\square}(\ol{\rho}_{\mf{m}})\times \widehat{T_n(F_p)}\ar[d, "t_{\chi}"]\\ \mc{E}^i(K^p)_{\mf{m}(\chi)}\ar[r] &X^{\square}(\ol{\rho}_{\mf{m}(\chi)})\times \widehat{T_n(F_p)}
     \end{tikzcd},
     \]
     where the map $t_{\chi}: X^{\square}(\ol{\rho}_{\mf{m}}) \times \widehat{T_n(F_p)} \to X^{\square}(\ol{\rho}_{\mf{m}(\chi)}) \times \widehat{T_n(F_p)} $ sends $(\rho, (\delta_1,\dots, \delta_n)) \mapsto (\rho \otimes \chi, (\delta_{\nu, 1} \otimes \chi \circ \Art_{F_{\nu}}, \dots, \delta_{\nu,n} \otimes \chi \circ \Art_{F_\nu})_{\nu})$
 \end{lemma}
 \begin{proof}
     By \cite[Lemma 2.5]{Hevesi23}, we have for any smooth left $K_p$-module $V$ which is a finite free $\mO/\pi^m$-module, there is an isomorphism
     \[
     R\Gamma(X_K, V) \to R\Gamma(X_K,V,V\otimes \chi_p)
     \]
     which intertwines the natural action of $\mc{H}(G^S,K^S)\otimes_{\Z}\mO$ with the natural action precomposed with $f_{\chi}$. Moreover, the proof in \textit{loc. cit.} implies if $V$ has a bimodule structure, this isomorphism is equivariant for this structure. Passing to the dual by applying $\RHom(-,\mO/\pi^m)$ yields a isomorphism $C_{\bullet}(X_{K^p})\otimes_{\Z[K_p]}V^{\vee}\to C_{\bullet}(X_{K^p})\otimes_{\Z[K_p]}V^{\vee}\otimes \chi_p^{-1}$ in $\mathsf{D}(\mO/\pi^m)$ . Plugging in $V^{\vee} = \mO[K_p/K_p']$ yields a $K_p$-equivariant isomorphism
     $
     C_{\bullet}(X_{K^p})\otimes_{\Z[K_p]}\mO/\pi^{m}[K_p/K_p']\to C_{\bullet}(X_{K^p})\otimes_{\Z[K_p]}\mO[K_p/K_p']\otimes \chi_p^{-1},
     $
     and passing to the derived limits over $K_p'$ and $m$ yields an isomorphism
     $C_{\bullet}(X_{K^p})\otimes_{\Z[K_p]}\mO[[K_p]]\to C_{\bullet}(X_{K^p})\otimes_{\Z[K_p]}\mO[[K_p]]]\otimes \chi_p^{-1}$ in $\mathsf{D}(\mO[[K_p]])$
     intertwining the actions of $\mc{H}(G^S,K^S)$ via a twist by $f_{\chi}^{-1}$. Moreover, using the isomorphism $C_{\bullet}(X_{K^p})\otimes_{\Z[G]}\mO[[G]]\to C_{\bullet}(X_{K^p})\otimes_{\Z[G]}\mO[[G]]\otimes \chi_p^{-1}$, this map is in fact $G$-equivariant. Taking a Borel--Serre complex, we get a continuous $K_p$-linear homotopy equivalence
     \[   C_{\bullet}^{\on{BS}}(X_{K^p})\otimes_{\Z[K_p]}D(K_p)\to C_{\bullet}^{\on{BS}}(X_{K^p})\otimes_{\Z[K_p]}D(K_p)\otimes \chi_p^{-1}.
     \]
     Passing to (Hausdorff) $N_0$-coinvariance, taking Hausdorff cohomology and finite slope parts yields the desired isomorphism 
     $\HH_i(J_B^{\vee}(C^{\on{BS}}(K,D(K_p)))) \simeq \HH_i(J_B^{\vee}(C^{\on{BS}}(K,D(K_p)\otimes \chi_p^{-1})))$ with the desired intertwining of Hecke actions. Running the same argument for the complexes localized at $\mf{m}$ yields an isomorphism
      $\HH_i(J_B^{\vee}(C^{\on{BS}}(K,D(K_p))))_{\mf{m}} \simeq \HH_i(J_B^{\vee}(C^{\on{BS}}(K,D(K_p)\otimes \chi_p^{-1})))_{\mf{m}}$, where the Hecke action on the right is twisted by $f_{\chi}^{-1}.$ Passing to the associated Hecke algebras an eigenvarieties, we get that the map on Hecke algebras $f_{\chi}:\mc{H}^S(G^S,K^S)\otimes_{\Z}\mO \to \mc{H}^S(G^S,K^S)\otimes_{\Z}\mO$ induces an isomorphism of rigid spaces $f_{\chi}^{\ast}:\Spf T^S(K^p)_{\mf{m}(\chi)} \times \widehat{T_n(F_p)} \to \Spf T^S(K^p)_{\mf{m}} \times \widehat{T_n(F_p)}$ which sends $(x: \T^S(K^p) \to \ol{\Q}_p, \delta = (\delta_1,\dots, \delta_n))\mapsto (x \circ f_{\chi}, (\delta_{1,\nu} \otimes \chi_{\nu} \circ \Art_{F_\nu}, \dots, \delta_{n,\nu}\otimes \chi_{\nu} \circ \Art_{F_\nu})_{\nu})$. The isomorphism we have just construct then implies that this map descends to an isomorphism 
      \[
      f_{\chi}^{\ast}:E^i(K^p)_{\mf{m}(\chi)} \simeq E^i(K^p)_{\mf{m}}.
      \]
      We then define $\on{tw}_{\chi} = f_{\chi}^{-1}$. The resulting effect on Galois representations (i.e. the diagram involving the Galois deformations) follows from the construction.
 \end{proof}

Note that for any $\chi$ as above, we get an induced (framed) twisted variant of $d_{ps,w}$ which we denote $d_{w,\chi}': X^{\square}(\ol{\rho}_{\mf{m}(\chi)})\times \widehat{T} \hookrightarrow X^{\square}(\ol{\rho}_{\wt{\mf{m}}(\chi)})\times \widehat{T}$. On representations, this map concretely sends $\rho\otimes \chi \mapsto \rho\otimes \chi \oplus \rho^{c,\vee}\otimes \chi^{c,\vee}(1-2n)$. Note that this map only depends on the reduction $\ol{\chi}$. Twisting this map by $\chi^{-1}$, we then have a map $d_{\chi}:X^{\square,n}(\overline{\rho}_{\mf{m}})\hookrightarrow X^{\square, 2n}(\overline{\rho}_{\wt{\mf{m}}(\chi)}(\chi^{-1}))$. Since any character twist of a (strictly) trianguline representation remains (strictly) trianguline, for any $\chi$ as above we have $\on{tr}\circ\on{res}_{\nu}\circ  d_{\chi}\circ \iota_{\mf{m}(\chi)}^i: \mc{E}^i(K^p)_{\mf{m}} \to X_{ps}(\ol{\rho}_{\wt{\mf{m}}(\chi)}(\chi^{-1})|_{\Gal_{F_{\nu}}}) \times \widehat{T}$ naturally factors through the trianguline pseudodeformation space
$X_{ps,\tri}(\ol{\rho}_{\wt{\mf{m}}(\chi)}(\chi^{-1}))$.
Now recall the shape of algebraic Hecke characters of a CM field. Namely, if $\chi: \Gal_F\to \ol{\Q}_p^{\times}$ is (the Galois representation corresponding to) an algebraic Hecke character of weight $\lambda \in \Z^{\Hom(F,\ol{\Q}_p)}$, then there exists $w \in \Z$ such that $\lambda_{\tau} + \lambda_{\tau c} = w$ for all $\tau$. In particular, $\chi\cdot \chi^{c}$ has constant weight $w$. We now derive the main theorem of the paper:

\begin{proof}[Proof of \cref{theorem: MainTheorem1}]

    The pointwise triangulinity follows by combining \cref{proposition : PointwiseTriangulineforEcd} and \cref{theorem: Embeddings of eigenvarieties}. To pin down the triangulation and weights, the idea is to twist the eigenvariety $\mc{E}^i(K^p)_{\mf{m}}$ by genuinely $p$-adic, non-integral weights using \cref{lemma: TwistByCharacters}. For a given point $(x,\delta)\in \mc{E}^i(K^p)_{\mf{m}}$ we claim we can choose distinct continuous characters $\chi_1,\chi_2: \Gal_{F}\to \mO_L^{\times}$ (taking $L/\Q_p$ sufficiently large) such that $\chi_1,\chi_2$ are trivial$\mod \mf{m}_L$ and the following hold:
    \begin{itemize}
    \item for each $\nu \in S_p(F)$ and $\tau: F_{\nu} \hookrightarrow \ol{\Q}_p$, the representations  $\rho_x, \chi_1^{-1}\chi_1^{c,\vee}\rho_{x}^{c,\vee}(1-2n), \chi_2^{-1}\chi_2^{c,\vee}\rho_{x}^{c,\vee}(1-2n)$ pairwise have no $\tau$-Hodge-Tate-Sen weights differing by an integer.
    \item for each $\nu \in S_p(F)$ the characters $\delta_v, -w_0^{G}\chi_{1,\nu}^{-1}\chi_{1,\nu^c}^{\vee}\delta_{v^c}(1-2n), -w_0^G\chi_{2,\nu}^{-1}\chi_{2,\nu^c}^{\vee}\delta_{v^c}(1-2n) \in \widehat{T_n(F_{\nu})}$ are such that no two pairwise characters differ by an algebraic character.
\end{itemize}   

Indeed, it suffices to show that we can choose $\chi_1,\chi_2$ so that for any $\tau: F_{\nu} \hookrightarrow \ol{\Q}_p$ the characters $a_1:=\chi_{1,\nu}^{-1}\chi_{1,\nu^c}^{\vee},a_2:=\chi_{2,\nu}^{-1}\chi_{2,\nu^c}^{\vee}$ have $\tau$-Sen weights such that $\on{wt}_{\tau}(a_1), \on{wt}_{\tau}(a_2),  \on{wt}_{\tau}(a_1a_2^{-1})$ are all non-integral and $\on{wt}_{\tau}(a_1), \on{wt}_{\tau}(a_2)$ avoid the discrete countable set $ \{(\Z + \weight_{\tau}(\rho_{x,\nu})_i - \weight_{\tau}(\rho_{x,\nu^c}^{\vee})_j)_{1\le i,j\le n}, (\Z + \weight_{\tau}(j_n(\delta)_{\nu,i}) - \weight_{\tau}(j_n(\delta)_{v^c,j}))_{1 \le i, j \le n} \}\subset \Z_p$. But by the above description of algebraic Hecke characters (in particular, composing the norm character with $p$-adic weights of $\mO_L^{\times}$), after passing to all possible $p$-adic characters we have for any $\lambda$ in an open subgroup of $\Z_p$ there exists a continuous character $\chi$ with $\chi^{-1}\chi^{c,\vee}$ having constant weight $\lambda$. As such, we simply choose $\chi_1,\chi_2$ so that $a_1,a_2$ have nonintegral weights whose difference is also nonintegral, and $a_1,a_2$ avoid the sets mentioned above, which is possible since $\Z_p$ is uncountable.

Let $\chi$ be either $\chi_1$ or $\chi_2$ as just constructed. Now, the proof of \cref{proposition : PointwiseTriangulineforEcd} implies for any $z = (x_z,\delta_z)\in \mc{E}_{c,\wt{\mf{m}}(\chi)}^d$ there is a (not necessarily open) affinoid $U_z$, a proper surjective map $f_z:Y_z \to U_z$, and a Galois representation $R_{z}$ on a locally free $\mc{O}_{Y_z}$-module $M$ of rank $2n$ such that $\on{tr} R_z= f_z^{\ast}T|_{U_z}$, and for all $y \in Y_z$, the specialization $(M_{Y_z,y}\otimes_{\mc{O}_{Y_z,y}}\ol{k(y)})^{\on{ss}}\simeq \rho_{f_z(y)}$, and in fact for all $y$, the specialization $(M_{Y_x,y}\otimes_{\mc{O}_{Y_x,y}}\ol{k(y)})^{\on{ss}}\simeq \rho_{f_z(y)}$ is trianguline at $\nu \mid p$ of parameter $j_{2n}(\delta_{f_z(y)})_{\ol{\nu}}\delta'$ for $\delta'$ some algebraic character depending on $y$. Thus via diagram \eqref{eq: TriangulineDiagram} we then have for any $\chi: \Gal_F \to \ol{\Q}_p^{\times}$ satisfying the conditions in \cref{lemma: TwistByCharacters} we have for any $(x,\delta) \in \mc{E}^i(K^p)_{\mf{m}}$ that $\rho_x\oplus \chi^{-1}\chi^{c, \vee}\rho_{x}^{c,\vee}(1-2n)$ is trianguline of parameter $(\delta_{\ol{\nu}}'\cdot w^{-1}(j_{n}(\delta)_{\nu}, -w_0(j_n(\delta)_{\nu^c})(1-2n)\cdot (\chi_{\nu}\chi_{\nu^c})^{-1}\circ \Art_{F_{\nu}}))_{\ol{\nu}}\in \widehat{T(F_p^+)}$, for $\delta'$ an algebraic character. 

Note since $w$ is by construction a minimal length representative, if $j_n(\delta)_{\nu} = (\delta_{1},\dots, \delta_n)$ and $-w_0j_n(\delta_{\nu^c}) = (t_1,\dots, t_n)$, then the relative order of the $\delta_i$ (respectively, the order amongst the $t_i$) is preserved within the conjugated character $d_{ps,w,\chi}(j_n(\delta)):= w^{-1}(j_{n}(\delta)_{\nu}, -w_0(j_n(\delta)_{\nu^c})(1-2n)\cdot (\chi_{\nu}\chi_{\nu^c})^{-1}\circ \Art_{F_{\nu}}))_{\ol{\nu}}$. Thus, if $\mc{F}_{\bullet,\chi_1}$ (respectively, $\mc{F}_{\bullet,\chi_2}$) is a triangulation of $D_{\on{rig}}^{\dagger}(\rho_x) \oplus D_{\on{rig}}^{\dagger}(\chi_{1,\nu}^{-1}\chi_{1,\nu^c}^{-1}\otimes \rho_{x, v^c}^{\vee}(1-2n))$ (respectively, $D_{\on{rig}}^{\dagger}(\rho_x) \oplus D_{\on{rig}}^{\dagger}(\chi_{2,\nu}^{-1}\chi_{2,\nu^c}^{-1}\otimes \rho_{x, v^c}^{\vee}(1-2n))$ with graded pieces given up to twists by algebraic characters by $d_{ps,w,\chi_1}(j_n(\delta))$ (resp, $d_{ps,w,\chi_1}(j_n(\delta))$, then by considering the image of the filtrations when projecting down to $D_{\on{rig}}^{\dagger}(\rho_x)$ for both $\chi_1, \chi_2$, the resulting filtration must have graded pieces of which, up to multiplying by algebraic characters, are in the intersection $\{d_{ps,w,\chi_1}(j_n(\delta))_i\}\cap \{d_{ps,w,\chi_1}(j_n(\delta))_i\} = \{j_{n}(\delta_i)\}$ by the assumptions on the characters $\chi_1,\chi_2$ and $j_n(\delta)$ are in the correct order since $w\in W^{\oP}$. The statement about the Hodge--Tate--Sen weights follows similarly, since by \cref{proposition : PointwiseTriangulineforEcd} we know the $\tau$-Sen weights of $\rho_x \oplus \rho_x^{c,\vee}(1-2n)\chi_i^{-1}\chi_i^{c,\vee}$ for the characters $\chi_i: \Gal_F \to \mO_L^{\times}$.
\end{proof}

\subsection{Towards a deformation-theoretic statement}
\label{subsection: finalsubsection}
We would like to have a more deformation-theoretic version of \cref{theorem: MainTheorem1}, so that we can study eigenvarieties via spaces of trianguline representations as in \cite{BHSIHES}. In essence, our proof of \cref{theorem: MainTheorem1} came from studying maps
\begin{align*}
d_{w, \chi,\ol{\nu}}:X^{\square}(\ol{\rho}_{\mf{m},\nu})\times \wh{T_n(F_{\nu})}\times X^{\square}(\ol{\rho}_{\mf{m},\nu^c})\times \wh{T_n(F_{\nu^c})}  \to X^{\square}(\ol{\rho}_{\wt{\mf{m}}(\ol{\chi}),\nu}(\ol{\chi}_{\nu}^{-1}))\times \wh{T(F_{\ol{\nu}}^+)}\\
((r_{\nu},\delta_{\nu}), (r_{\nu^c}, \delta_{\nu^c}))\mapsto (r_{\nu}\oplus r_{\nu^c}^{\vee}(1-2n)(\chi_{\nu}\chi_{\nu}^c)^{-1},w^{-1}(\delta_{\nu}, -w_0^G\delta_{\nu^c}(1-2n)(\chi_{\nu}\chi_{\nu}^c)^{-1})).
\end{align*}
In actuality, we considered this map after taking the pseudocharacter, we which we denoted $d_{ps,w,\chi, \ol{\nu}}$. Note of course then $d_{ps,w,\chi, \ol{\nu}} = \on{tr}\circ d_{w, \chi,\ol{\nu}}$.
a natural strategy is to constrain the preimage $P_{ps,\chi} := d_{ps,w,\chi,\ol{\nu}}^{-1}(X_{ps, \tri}(\ol{\rho}_{\wt{\mf{m}}(\ol{\chi}),\nu}(\ol{\chi}_{\nu}^{-1}))) \subset X^{\square}(\ol{\rho}_{\mf{m},\nu})\times \widehat{T_n(F_{\nu})}\times X^{\square}(\ol{\rho}_{\mf{m},\nu^c})\times \widehat{T_n(F_{\nu^c})}$, since in \eqref{eq: TriangulineDiagram} the map $\on{res}_{\nu,\nu^c} \circ \iota_{\mf{m}}^i$ from $\mc{E}^i(K^p)_{\mf{m}}$ factors through $P_{ps,\chi}$.

\begin{lemma}
\label{lemma: ContainsTriangulineVariety}
    We have an inclusion $P_{ps,\chi}\supset X_{\tri}^{\square}(\ol{\rho}_{\mf{m,\nu}}) \times X_{\tri}^{\square}(\ol{\rho}_{\mf{m,\nu^c}})$.
\end{lemma}
\begin{proof}
    Letting $c_{w,\chi}: \widehat{T(F_{\ol{\nu}}^+)}\simeq \widehat{T(F_{\ol{\nu}}^+)}$ denote the isomorphism sending $(\delta_{\nu}, -w_0\delta_{\nu^c}) \to w^{-1}(\delta, -w_0\delta_{\nu^c}(1-2n)(\chi_{\nu}\chi_{\nu^c})^{-1})$. 
    Note that we have Zariski dense and open inclusions  $c_{w,\chi}^{-1}(\widehat{T(F_{\ol{\nu}}^+)}_{\on{reg}}) \subset \widehat{T_n(F_{\nu})}_{\on{reg}} \times \widehat{T_n(F_{\nu^c})}_{\on{reg}} \subset \widehat{T(F_{\ol{\nu}}^+)}$.

    Let $U_{w,\chi}:= \{((r_1,\delta_1), (r_2,\delta_2))\in U_{\tri}(\ol{\rho}_{\mf{m,\nu}}) \times U_{\tri}(\ol{\rho}_{\mf{m,\nu^c}}): c_{w,\chi}(\delta_1, -w_0^G\delta_2) \in \widehat{T(F_{\ol{\nu}}^+)}_{\on{reg}}\}$ be the pullback of $U_{\tri}(\ol{\rho}_{\mf{m,\nu}}) \times U_{\tri}(\ol{\rho}_{\mf{m,\nu^c}})$ to $c_{w,\chi}^{-1}(\widehat{T(F_{\ol{\nu}}^+)}_{\on{reg}})$. Then since the morphism $U_{\tri}(\ol{\rho}_{\mf{m,\nu}}) \times U_{\tri}(\ol{\rho}_{\mf{m,\nu^c}}) \to \widehat{T(F_{\ol{\nu}}^+)}$ is smooth and $c_{w,\chi}^{-1}(\widehat{T(F_{\ol{\nu}}^+)}_{\on{reg}}) \subset \widehat{T_n(F_{\nu})}_{\on{reg}} \times \widehat{T_n(F_{\nu^c})}_{\on{reg}}$ is Zariski dense and open, the $U_{w,\chi}$ is also Zariski dense and open in $U_{\tri}(\ol{\rho}_{\mf{m,\nu}}) \times U_{\tri}(\ol{\rho}_{\mf{m,\nu^c}})$. Since $P_{ps,\chi}$ is by construction Zariski-closed in (the product of) deformation spaces, it contains the Zariski-closure of $U_w$, which is $X_{\tri}^{\square}(\ol{\rho}_{\mf{m,\nu}}) \times X_{\tri}^{\square}(\ol{\rho}_{\mf{m,\nu^c}})$.
\end{proof}
 \noindent Note that $\on{tr}(X_{\tri}^{\square}(\ol{\rho}_{\wt{\mf{m}}(\ol{\chi}),\nu}(\ol{\chi}_{\nu}^{-1}))) \subset X_{ps, \tri}(\ol{\rho}_{\wt{\mf{m}}(\ol{\chi}),\nu}(\ol{\chi}_{\nu}^{-1}))$, and thus we have an inclusion of Zariski closed subspaces $d_{w,\chi, \ol{\nu}}^{-1}(X_{\tri}^{\square}(\ol{\rho}_{\wt{\mf{m}}(\ol{\chi}),\nu}(\ol{\chi}_{\nu}^{-1})))=: P_{\chi} \subset P_{ps,\chi}$ for any $\chi$. Moreover, the same proof as in \cref{lemma: ContainsTriangulineVariety} shows $P_{\chi} \supset X_{\tri}^{\square}(\ol{\rho}_{\mf{m},\nu})\times X_{\tri}^{\square}(\ol{\rho}_{\mf{m},\nu^c})$. Ultimately, we would like to reduce the study of $P_{ps,\chi}$ to $P_{\chi}$, since $X_{\tri}^{\square}$ has been more thoroughly studied. We end the article by studying $P_{\chi}$ and discuss some lemmas that might be helpful in showing $P_{\chi} \subset P_{ps,\chi}$ might actually be an equality.

% We moreover know that on a Zariski dense and open subset of $Y$, the image of $V_i$ is semisimple!

\begin{lemma}
    Let $\rho:\Gal_F \to \GL_{2n}(\mO_L)$ be a continuous representation such that $\overline{\rho}^{\on{ss}} \simeq \overline{\rho}_{\mf{m}} \oplus \overline{\rho}_{\mf{m}}^{c,\vee}(1-2n)$, where $\overline{\rho}_{\mf{m}}: \Gal_F \to \GL_{n}(\overline{\F}_p)$ is an absolutely irreducible representation. Then if $(\rho|_{\Gal_{F_{\nu}}}, \delta)\in X_{\on{tri}}^{\square}(\ol{\rho}_{\wt{\mf{m}}, \nu})$, we also have $(\rho^{\on{ss}}|_{\Gal_{F_{\nu}}}, \delta)\in X_{\on{tri}}^{\square}(\ol{\rho}_{\wt{\mf{m}},\nu})$.
\end{lemma}

\begin{proof}
    If $\rho$ is irreducible, then it is automatically semisimple, and there is nothing to prove. Otherwise, we have $\rho^{ss}\simeq V \oplus W$, where $V$ and $W$ are lifts of the irreducible representations $\overline{\rho}_{\mf{m}}, \overline{\rho}_{\mf{m}}^{c,\vee}$, respectively. Suppose WLOG that $V$ is a subrepresentation of $\rho$. As a result, $\rho$ is conjugate to a representation of the form 
    \[
    \begin{pmatrix}
        V & \ast\\
        0 & W
    \end{pmatrix},
    \]
    for some potentially nontrivial $\ast$. Now for all $t$, we may consider the auxiliary representation 
    \[
    \rho_{t}:= \begin{pmatrix}
        V & t\cdot \ast\\
        0 & W
    \end{pmatrix}: \Gal_F \to \GL_2(\mc{O}_{L}[[t]])
    \]
    For $t \neq 0$, we have $\rho_{t}\simeq \rho$, as they are conjugate by the matrix $\begin{pmatrix}
        t\cdot I_n & 0\\
        0 & I_n
    \end{pmatrix}.$ 
    The universal property of the deformation space then gives a map 
    \[
    f:(\on{Spf}(\mO_L[[t]]))^{\on{rig}} \to X^{\square,2n}(\ol{\rho}_{\wt{\mf{m}},\nu}) \times \widehat{T(F_p^+)},
    \]
    where on the second factor it is just the constant map $\delta \in \widehat{T}$. We know that restricting to $f|_{(\on{Spf}(\Z_p[[t]]))^{\on{rig}}\setminus \{0\}}$, the image $(\rho_t, \delta) \in X_{\on{tri}}^{\square}(\ol{\rho}_{\mf{m},\nu})$ since $(\rho, \delta) \in X_{\tri}^{\square, 2n}(\ol{\rho}_{\mf{m}})$ and conjugation by $\begin{pmatrix}
        t\cdot I_n & 0\\
        0 & I_n
    \end{pmatrix}$ is an automorphism on the trianguline variety. As $X_{\on{tri}}^{\square}(\ol{\rho}_{\mf{m}})$ is Zariski-closed and $(\Spf(\mO_L[[t]]))^{\on{rig}}\setminus \{0\}$ is a dense open in $\Spf(\mO_L[[t]])$, we have a factorization $f: \on{Spf}(\mO_L[[t]]) \to X_{\on{tri}}^{\square}(\ol{\rho}_{\wt{\mf{m}}})$. Specializing, to $t = 0$, we have $\rho_0 \simeq \begin{pmatrix}
        V & 0\\
        0 & W
    \end{pmatrix} \simeq \rho^{ss}$ and thus $(\rho^{\on{ss}},\delta)\in X_{\tri}^{\square}(\ol{\rho}_{\wt{\mf{m}},\nu})$, as required.
\end{proof}

We now study $P_{\chi}$ for varying $\chi$. In particular, for all characters $\chi: \Gal_F \to \mO_L^{\times}$ as in \cref{lemma: TwistByCharacters} we are led to study $W:= \bigcap_{\chi}P_{\chi}\supset X_{\tri}^{\square}(\ol{\rho}_{\mf{m},\nu})\times X_{\tri}^{\square}(\ol{\rho}_{\mf{m},\nu^c})$. To get a grasp on this intersection, we first look at the open dense subset $U_{\tri}(\ol{\rho}_{\wt{\mf{m}}(\chi),\nu}(\ol{\chi}_{\nu}^{-1}))$ for varying $\chi$.
% Let $\pi_{1}: X^{\square}(\ol{\rho}_{\mf{m},\nu})\times \widehat{T_n(F_{\nu})}\times X^{\square}(\ol{\rho}_{\mf{m},\nu^c})\times \widehat{T_n(F_{\nu^c})} \to X^{\square}(\ol{\rho}_{\mf{m},\nu})\times \widehat{T_n(F_{\nu})}$ be the projection on the representation at $\nu$.

\begin{lemma}
\label{lemma: MatchesOnRegularTrianguline}
         There is an inclusion $\bigcap_{\chi} d_{w,\chi,\ol{\nu}}^{-1}(U_{\on{tri}}(\ol{\rho}_{\wt{\mf{m}}(\chi),\nu}(\ol{\chi}_{\nu}^{-1})))\subset U_{\tri}(\ol{\rho}_{\mf{m},\nu})\times U_{\tri}(\ol{\rho}_{\mf{m},\nu^c})$. In fact for any $x = ((\rho_{x,\nu}, \delta_{\nu}), (\rho_{x,\nu}, \delta_{\nu^c}))$, there exists two continuous characters $\chi_1,\chi_2: \Gal_{F} \to \mO_L^{\times}$ such that if $d_{w,\chi_1,\ol{\nu}}(x)\in U_{\on{tri}}(\ol{\rho}_{\wt{\mf{m}}(\chi_1),\nu}(\ol{\chi}_{1,\nu}^{-1}))$ and $d_{w,\chi_2,\ol{\nu}}(x)\in U_{\on{tri}}(\ol{\rho}_{\wt{\mf{m}}(\chi_2),\nu}(\ol{\chi}_{2,\nu}^{-1}))$, then $x\in U_{\tri}(\ol{\rho}_{\mf{m},\nu})\times U_{\tri}(\ol{\rho}_{\mf{m},\nu^c})$.
\end{lemma}
\begin{proof}
    We are tasked with showing that if $\rho_{x,\nu} \oplus \chi_{\nu}^{-1}\chi_{\nu^c}^{-1}\rho_{x,\nu^c}^{\vee}(1-2n)$ is trianguline of parameter $w^{-1}(\delta_{\nu}, (-w_0\delta_{\nu^c})(1-2n)\chi_{\nu}^{-1}\chi_{\nu^c}^{-1}) \in \widehat{T(F_{\ol{\nu}}^+)}$ for many $\chi$ then $\rho_{x,\nu}$ is trianguline of parameter $\delta\in \wh{T_n(F_{\nu})}_{\on{reg}}$.
    We follow the same strategy as in the proof of \cref{theorem: MainTheorem1}, where we choose two characters $\chi_1, \chi_2$ such that $\rho , \rho^{c,\vee} \otimes \chi_1^{c,\vee}\chi_1^{-1}(1-2n), \rho^{c,\vee} \otimes \chi_2^{c, \vee}\chi_2^{-1}(1-2n)$ have no pairwise Sen weights (between each of the representations) differing by an integer, and same for the triangulations. If both $w^{-1}(\delta_{\nu}, (-w_0\delta_{\nu^c})(1-2n)\chi_{1,\nu}^{-1}\chi_{1,\nu^c}^{-1})$ and $w^{-1}(\delta_{\nu}, (-w_0\delta_{\nu^c})(1-2n)\chi_{2,\nu}^{-1}\chi_{2,\nu^c}^{-1})$ are regular parameters, then these triangulations are unique. Then by considering the  filtrations $\mc{F}_{\bullet, \chi_i}$ on $D_{\on{rig}}^{\dagger}(\rho_{x,\nu}) \oplus D_{\on{rig}}^{\dagger}(\rho_{x,\nu^c}^{\vee}\otimes\chi_{i,\nu^c}^{-1}\chi_{i,\nu}^{-1}(1-2n))$, the weights of the parameters/subquotients of the triangulation of which map nontrivially to (a subquotient of) $D_{\on{rig}}^{\dagger}(\rho)$ must be contained in the intersection of the sets $\{(\weight_{\tau}(\delta_{\nu,i}))_i, (\weight_{\tau}(\delta_{\nu^c,j}\chi_1^{c,\vee}\chi_1^{-1}(1-2n)))_j\} \cap \{(\weight_{\tau}(\delta_{\nu,i}))_i, -(\weight_{\tau}((\delta_{\nu^c,j}\chi_2^{c,\vee}\chi_2^{-1}(1-2n)))_j\}= \{(\weight_{\tau}(\delta_i))_i\}$, so by fact that the pairwise weights cannot differ by integers, only the $\mc{R}_L(\delta_{\nu,i})$ can admit non-zero maps to $\rho_{x,\nu}$, and they give a triangulation. We then similarly get a triangulation on $\rho_{x,\nu^c}^{\vee}$, and thus on $\rho_{x,\nu^c}$.
\end{proof}
% \begin{corollary}
%     Let $x := (\chi,\delta) \in \mc{E}^i(K^p)_{\mf{m}}$ be a point such that $\delta \in \mc{T}^{\on{reg}}$. Then there exists an affinoid open neighborhood $x\in U \subset \mc{E}_{\mf{m}}^i$ such that the map $r_i:U \to X^{\square}(\overline{r}_{\mf{m}}|_{G_{F_{\nu}}}) \times \wh{T}$ factors through $X_{\on{tri}}^{\square,n}$.
% \end{corollary}
% \begin{proof}
%     This fact follows from \cref{lemma: MatchesOnRegularTrianguline}.
% \end{proof}
\noindent We end with discussing two other examples.
\begin{example}[$n = 2$, $F$ imaginary quadratic]
\label{example: Bianchi Case on}
    We can try to pin down precisely the conditions we need in the Bianchi case. Here, $F^+ = \Q$ and $p = \nu \nu^c$ is split in $F$. Thus in this case, $T(F_{\ol{\nu}}^+) = T_2(F_{\nu})\times T_2(F_{\nu^c})\simeq (\Q_p^{\times})^4$.  We can also say a little more about the element $w \in W^{\oP}$. Since we are localizing at a non-Eisenstein maximal ideal and $F$ is imaginary quadratic, the only degrees of cohomology we are interested in are $i = 1,2$. In this case the middle degree $d = 4$, so the lengths of the needed Weyl group representatives are $\ell(w_1) = 2, \ell(w_1) = 1$. In the set $W^{\oP}\simeq (S_2\times S_2)\backslash S_4$, two such representatives are $w_1  = \begin{pmatrix}
       &1&& \\
        &&1&\\
        1&&&\\
        &&&1
    \end{pmatrix}$ and $w_2  = \begin{pmatrix}
       1&&& \\
        &&1&\\
        &1&&\\
        &&&1
    \end{pmatrix}$.
Given $(e_1,e_2, f_1, f_2) \in \widehat{T(F_{\ol{\nu}}^+)}$, we then have $w_1^{-1}(e_1,e_2, f_1,f_2) = (f_1, e_1, e_2, f_2)$ and $w_2^{-1}(e_1,e_2, f_1, f_2) = (e_1, f_1, e_2, f_2)$. For simplicity let us focus on the case of $w = w_1$, which would correspond to studying $\HH^1(J_{B_2}(\Pi(K^p)_{\mf{m}}^{\la}))\simeq J_{B_2}(\wt{\HH}^1(K^p,L)_{\mf{m}}^{\la})$, this equality holding since one can show by hand that $\wt{\HH}^0(K^p,L)_{\mf{m}}$ is Eisenstein.
    In this case we study a variant of the map $d_{w_1,\chi, \ol{\nu}}$ from before, defined as 
    \[
    d_{\chi}: X^{\square}(\ol{\rho}_{\mf{m},\nu}) \times \wh{(\Q_p^{\times})^2} \times X^{\square}(\ol{\rho}_{\mf{m},\nu^c})\times \widehat{(\Q_p^{\times})^2} \hookrightarrow X^{\square}(\ol{\rho}_{\wt{\mf{m}}(\chi),\nu}(3)\ol{\chi}_{\nu^c}) \times \widehat{(\Q_p^{\times})^4}
    \]
    \begin{align*}
    (\rho_{\nu}, (\delta_{\nu,1},\delta_{\nu, 2}) , \rho_{\nu^c},  (\delta_{\nu^c,1}, \delta_{\nu^c,2})) &\mapsto (\chi_{\nu}\chi_{\nu^c}\rho_{\nu}(3) \oplus \rho_{\nu^c}^{\vee},d_{\chi}(\delta))\\
    d_{\chi}(\delta)&:= (-\delta_{\nu^c, 2}, \delta_{\nu,1}\chi_{\nu}\chi_{\nu^c}(3), \delta_{\nu,2}\chi_{\nu}\chi_{\nu^c}(3), -\delta_{\nu^c,1})
    \end{align*}
    We are interested in the preimage of $X_{\on{tri}}^{\square}(\ol{\rho}_{\wt{\mf{m}}(\chi),\nu}(3)\ol{\chi}_{\nu^c})$ under $d_{\chi}$ for various characters $\chi$. By \cite[Theorem 6.3.13]{KPX14}, if $(\chi_{\nu}\chi_{\nu^c}\rho_{\nu}(3) \oplus \rho_{\nu^c}^{\vee}, d_{\chi}(\delta)) \in X_{\on{tri}}^{\square}(\ol{\rho}_{\wt{\mf{m}}(\chi),\nu}(3)\ol{\chi}_{\nu^c})$ then $\chi_{\nu}\chi_{\nu^c}\rho_{\nu}(3) \oplus \rho_{\nu^c}^{\vee}$ is trianguline of parameter $d_{\chi}(\delta)\cdot (z^a,z^b, z^c, z^d)$ where $a,b,c,d\in \Z$ satisfy $a + b + c + d = 0$. Choosing $\chi_1,\chi_2$ appropriately as in the proof of \cref{theorem: MainTheorem1}, we may assume that $\rho_{\nu^c}^{\vee}$ is trianguline of parameters $-w_0\delta_{\nu^c} = (-\delta_{\nu^c,2}, -\delta_{\nu^c,2})\cdot (z^r,z^s)$ for some $r + s = 0$, and $\chi_{\nu}\chi_{\nu}^{-1}\delta_{\nu,i}\delta_{\nu^c,j}$ has non-integral weight for $1\le i,j \le 2$. In particular, the non-zero map  $\mc{R}_{L, \Q_p}(-\delta_{\nu^c,2}z^a) \hookrightarrow D_{\on{rig}}^{\dagger}(\chi_{\nu}\chi_{\nu^c}\rho_{\nu}(3) \oplus \rho_{\nu^c}^{\vee})$ induces a nonzero map $\mc{R}_{L, \Q_p}(-\delta_{\nu^c,2}z^a) \to D_{\on{rig}}^{\dagger}(\rho_{\nu^c}^{\vee})$. Similarly, there is a nonzero map $D_{\on{rig}}^{\dagger}(\rho_{\nu^c}^{\vee}) \to \mc{R}_{L, \Q_p}(-\delta_{\nu^c, 1}z^d)$ through which the triangulation on $\rho_{\nu}\chi_{\nu}\chi_{\nu^c}(2n-1)\oplus \rho_{\nu^c}$ factors. This gives a resulting triangulation of $\rho_{\nu^c}^{\vee}$, and thus $a+ d= 0$. Moreover, by \cite[Lemma 3.9]{CNT24} we in fact have $a \le 0$. Then by \cite[Proposition 6.2.8]{KPX14} there is a nonzero map $\mc{R}_{L, \Q_p}(-\delta_{\nu^c,2}) \to \mc{R}_{L, \Q_p}(-\delta_{\nu^c,2}\cdot z^{a})$, and thus a nonzero map $\mc{R}_{L, \Q_p}(-\delta_{\nu^c,2}) \to D_{\on{rig}}^{\dagger}(\rho_{\nu^c}^{\vee})$. Then by dualizing, to show $(\rho_{\nu^c}, \delta_{\nu^c}) \in X_{\tri}^{\square}(\ol{\rho}_{\mf{m},\nu^c})$, we are reduced to the following question.
    \begin{conjecture}
    When $F_{\nu} = \Q_p$, we have $(\rho, \delta_1,\delta_2) \in X_{\on{tri}}^{\square}(\ol{\rho}_{\mf{m},\nu^c})$ if and only if there is a nonzero map $D_{\on{rig}}^{\dagger}(\rho) \to \mc{R}_{L, \Q_p}(\delta_2)$, $\det \rho = \delta_1\delta_2$, and with $\tau$-Sen weights $(\on{wt}_{\tau}(\delta_1), \on{wt}_{\tau}(\delta_2))$.
    \end{conjecture} 
\noindent This conjecture would imply at least that the projection of $\bigcap_{\chi}d_{\chi}^{-1}(X_{\on{tri}}^{\square}(\ol{\rho}_{\wt{\mf{m}}(\chi),\nu}(3)\ol{\chi}_{\nu^c}))$ to deformation space at $\nu^c$ lands in $X_{\tri}^{\square}(\ol{\rho}_{\mf{m},\nu^c})$. By symmetry, a similar result would also hold for $\nu$.
\end{example}
\subsection{The case of generic crystalline representations}
A key case in which there is strong control on points of the trianguline variety comes from the local models of \cite{BHSIHES}, and the extension to non-regular representations by \cite{Wu24}. Let $r: \Gal(\ol{K}/K) \to \GL_n(\ol{\Q}_p)$ be a crystalline representation, and denote its reduction by $\ol{r}$. Associated to $r$ then is the rank $n$ $\varphi$-module $D_{\on{crys}}(r)$ over $L\otimes_{\Q_p}K_0$, where $K_0\subset K$ denotes the maximal unramified subextension. Fixing an embedding $\tau: K_0 \hookrightarrow L$, let $\varphi_1, \dots, \varphi_n$ be eigenvalues of $\varphi^{[K_0:\Q_p]}$ on $L\otimes_{1\otimes \tau_0, L\otimes_{\Q_p}K_0}D_{\on{crys}}(r)$. We say $r$ is \textit{generic} if we have $\varphi_{i}\varphi_j^{-1} \notin \{1, p^{[K_0:\Q_p]}\}$ for all $i\neq j$. This notion is independent of the choice of $\tau_0$.

If $r$ is crystalline generic, we fix an ordering $\ul{\varphi} = (\varphi_1,\dots, \varphi_n)$ of the Frobenius eigenvalues. Our goal now is to describe the classification of points on $X_{\on{tri}}^{\square}(\ol{r})$ with representation $r$. Let $\lambda = (\on{HT}_{\tau}(r))_{\tau} = (\lambda_{\tau, 1}, \dots, \lambda_{\tau, n})_{\tau: K \hookrightarrow \ol{\Q}_p}$ be the dominant vector of Hodge--Tate weights, such that the distinct weights appear with multiplicities $m_{\tau, 1}, \dots, m_{\tau, s_{\tau}}$. Let $W_{\on{stab}}\subset W$ be the stabilizer subgroup of $\lambda$. Then for any $w \in W/W_{\on{stab}}$, we consider a character of $(K^{\times})^n$, denoted by $z^{-w(\lambda)}\on{unr}(\ul{\varphi})$, with components given by 
\[
(z^{-w(\lambda)_1}\on{unr}(\varphi_1),\dots, z^{-w(\lambda)_n}\on{unr}(\varphi_n))\in \widehat{T_L^n}
\]
Let $x_w \in X_{\ol{r}}^{\square}\times \widehat{T_L^n}$ be the point $x_w = (r,z^{-w(\lambda)}\on{unr}(\ul{\varphi})).$ Associated to the point $x:= x_{w_0}\in U_{\tri}(\ol{r})$ we can construct a certain permutation $w_x \in (S_n)^{[K:\Q_p]}$ (see \cite[\textsection 4.1]{Wu24} for a brief summary).

\begin{theorem}[{\cite[Theorem 4.2.3]{BHSIHES}, \cite[Theorem 4.1]{Wu24}}]
\label{theorem: LocalCompanionPoints}
    $x_w\in X_{\on{tri}}^{\square}(\rbar)$ iff $w \ge w_x$ in $W/W_{\on{stab}}$. In fact, the set of $(r,\delta') \in X^{\square}(\ol{r})\times \widehat{T}_L$ such that $(r,\delta') \in X_{\tri}^{\square}(\ol{r})$ is precisely 
    \[
    \bigcup_{\mc{R} = \ul{\varphi}}\{(r,z^{-w(\lambda)}\on{unr}(\ul{\varphi}): w_x \le w\},
    \]
    where the union runs over all possible orderings $\mc{R}$ of $\varphi_1,\dots, \varphi_n$, i.e., \emph{refinements}.
\end{theorem}
\begin{corollary}
\label{corollary: GenericCrystallineOnTriangulineVariety}
    For any $i \ge 0$, if $x = ((r_{\nu},\delta_{\nu}),(r_{\nu^c},\delta_{\nu^c})) \in X^{\square, n}(\ol{\rho}_{\mf{m},\nu})\times \widehat{T}_L \times X^{\square, n}(\ol{\rho}_{\mf{m},\nu^c})\times \widehat{T}_L$ is such that $r_{\nu}, r_{\nu^c}$ are crystalline generic, then there exists an algebraic Hecke character $\chi:\Gal_F \to \ol{\Z}_p^{\times}$ such that if $d_{w,\chi, \ol{\nu}}(x) \in X_{\tri}^{\square}(\ol{\rho}_{\wt{\mf{m}}(\chi),\nu}\ol{\chi}_{\nu}^{-1})$, then $x \in X_{\on{tri}}^{\square}(\ol{\rho}_{\mf{m},\nu})\times X_{\on{tri}}^{\square}(\ol{\rho}_{\mf{m},\nu})$.
    % In fact there is a Zariski open neighborhood $U \ni x$ such that $r_n(U) \subset X_{\on{tri}}^{\square, n}(\ol{\rho}_{\mf{m}}|_{G_{F_{\nu}}}).$
\end{corollary}
\begin{proof}
First, we claim we can choose a crystalline character $\chi: \Gal_F \to \mO_L^{\times}$ (enlarging $L$ if necessary) such that $r_{\nu}\oplus (\chi_{\nu}\chi_{\nu^c})^{-1}r_{\nu^c}^{\vee}(1-2n)$ is also crystalline generic and the representations $r_{\nu},(\chi_{\nu}\chi_{\nu^c})^{-1}r_{\nu^c}^{\vee}(1-2n)$ pairwise share no Hodge--Tate weights in common. The statement about Hodge--Tate weights follows from letting the weight of $(\chi_{\nu}\chi_{\nu^c})^{-1}$ be very large, and the statement about genericity amounts to choosing $\chi_{\nu}\chi_{\nu^c}$ to have crystalline Frobenius avoiding a finite list of elements of some p-adic field $L$, which we can ensure holds by taking the field of rationality of the Frobenius to be sufficiently large. Let $\mathbf{h}$ be the dominant vector of Hodge--Tate weights for $\rho_{\nu}$, $a_{\chi}:= (-\on{wt}_{\tau}(\chi_{\nu}\chi_{\nu^c}),\dots, -\on{wt}_{\tau}(\chi_{\nu}\chi_{\nu^c}))_{\tau} \in (\Z^n)^{\Hom(F_{\nu},\ol{\Q}_p)}$, and for any $k \in \Z$ write $(\ul{k}) = (k, \dots, k) \in (\Z^n)^{\Hom(F_{\nu}, \ol{\Q}_p)}$. Moreover, we may also choose $\chi$ so that $(\mathbf{h}^{c,\vee} + a_{\chi} + \ul{(2n-1)}, \mathbf{h})\in (\Z^{2n})^{\Hom(F_{\ol{\nu}}^+,\ol{\Q}_p)}$ is a dominant vector. Now suppose $d_{w,\chi,\ol{\nu}}(x) \in X_{\tri}^{\square}(\ol{\rho}_{\wt{\mf{m}}(\chi),\nu}\ol{\chi_{\nu}}^{-1})$. Letting $R_{r_{\nu}}, R_{(\chi_{\nu}\chi_{\nu^c})^{-1}r_{\nu^c}^{\vee}}$ be refinements for $r_{\nu}, (\chi_{\nu}\chi_{\nu^c})^{-1}r_{\nu^c}^{\vee}$, we let $\varphi_{\chi} = w^{-1}(\mc{R}_{r_{\nu}}, \mc{R}_{(\chi_{\nu}\chi_{\nu^c})^{-1}r_{\nu^c}}^{\vee})$ be a refinement ordered according to the fixed $w\in W^{\oP}$. We can then apply \cref{theorem: LocalCompanionPoints} to $d_{w,\chi,\ol{\nu}}(x)$ to write $d_{w,\chi,\ol{\nu}}(x) = (r_{\nu} \oplus (\chi_{i,\nu}\chi_{\nu^c})^{-1}r_{\nu^c}^{\vee}(1-2n), z^{-w'(\mathbf{h}^{c,\vee} + a_{\chi} + \ul{(2n-1)}, \mathbf{h})}\on{unr}(\varphi_{\chi}))$ for $\mathbf{h} \in (\Z^n)^{\Hom(F_{\nu},\ol{\Q}_p)}$, for some $w' \ge w_{\varphi_{\chi}}$. By definition of $d_{w,\chi, \ol{\nu}}$, and the assumption on pairwise sharing no Hodge--Tate weights, we have $w' \in w^{-1}(S_n\times S_n)^{\Hom(F_{\ol{\nu}}^+,\ol{\Q}_p)}$. On the other hand by construction we have $w_{\varphi_{\chi}} \in (w)^{-1}(S_n\times S_n)^{\Hom(F_{\ol{\nu}}^+,\ol{\Q}_p)}$. Thus writing out $w' = w^{-1}(w_1, w_2)$ and $w_{\mc{R}} = (w)^{-1}(w_{r_{\nu}},w_{r_{\nu^c}^{\vee}})$, we claim the inequality $w'\ge w_{\mc{R}}$ implies $w_1\ge w_{r_{\nu}}$. Indeed, this follows from the fact that $w\in W^{\oP}$ is a minimal length representative. Then another application of \cref{theorem: LocalCompanionPoints} implies that $(r, z^{-w_1(\mathbf{h})}\on{unr}(\varphi)) \in X_{\tri}^{\square}(\ol{\rho}_{\mf{m},\nu})$, as desired. The same argument works for $r_{\nu^c}^{\vee}$.
\end{proof}
\noindent An extension of significant interest would be to show if there is a Zariski open neighborhood $d_{w,\chi,\ol{\nu}}^{-1}(X_{\tri}^{\square}(\ol{\rho}_{\wt{\mf{m}}(\chi),\nu}(\ol{\chi}_{\nu}^{-1}))\supset U \ni x$ such that $U \subset X_{\on{tri}}^{\square}(\ol{\rho}_{\mf{m},\nu}).$
In particular, at regular crystalline points then one could attempt to locally model eigenvarieties as closed subvarieties of the trianguline local models constructed in \cite{BHSIHES}.

\bibliographystyle{alpha}
\bibliography{references}

\end{document}